\documentclass[a4paper]{amsart}
\usepackage[headings]{fullpage}
\usepackage{amsmath, amsfonts, amsthm}
\usepackage{amssymb, latexsym}  
\usepackage{hyperref}
\usepackage{cmap}

\newtheorem{thm}{Theorem}[section]
\newtheorem{lemma}[thm]{Lemma}
\newtheorem{cor}[thm]{Corollary}
\newtheorem{claim}{Claim}[thm]
\newtheorem{prop}[thm]{Proposition}
\newtheorem{fact}[thm]{Fact}

\theoremstyle{definition}
\newtheorem{definition}[thm]{Definition}
\newtheorem{example}[thm]{Example}

\theoremstyle{remark}
\newtheorem*{remark}{Remark}
\newtheorem*{Q}{Question}

\newcommand\diagonal{\bigtriangleup}

\newcommand\forces{\Vdash}
\newcommand\s{\subseteq}
\newcommand\sq{\sqsubseteq}
\newcommand{\sqleft}[1]{\mathrel{_{#1}{\sqsubseteq}}}
\newcommand{\sqx}{\sqleft{\chi}}
\newcommand{\sql}{\sqleft{\lambda}}

\newcommand{\sqaz}{\sqleft{\aleph_0}}
\newcommand{\nsqleft}[1]{\mathrel{_{#1}{\not\sqsubseteq}}}
\newcommand{\nsqx}{\nsqleft{\chi}}
\newcommand\bks{\setminus}
\newcommand\br{\blacktriangleright}
\DeclareMathOperator{\suc}{succ}

\newcommand{\colonnn}{\mathbin{\vbox{\baselineskip=3pt\lineskiplimit=0pt\hbox{.}\hbox{.}\hbox{.}}}}

\newcommand\sd{\framebox[3mm][l]{$\diamondsuit$}\hspace{0.5mm}{}}
\newcommand\sqc{\framebox[3mm][l]{$\clubsuit$}\hspace{0.5mm}{}} 
\newcommand\fsd{\framebox[2.9mm][l]{$\diamondsuit$}\hspace{0.5mm}{}}

\renewcommand{\restriction}{\mathbin\upharpoonright}    

\DeclareMathOperator{\height}{ht}   
\DeclareMathOperator{\cf}{cf}
\DeclareMathOperator{\col}{Col}
\DeclareMathOperator{\dom}{dom}
\DeclareMathOperator{\rng}{Im}
\DeclareMathOperator{\otp}{otp}
\DeclareMathOperator{\acc}{acc}
\DeclareMathOperator{\add}{Add}
\DeclareMathOperator{\nacc}{nacc}

\newcommand\axiomfont[1]{\textsf{\textup{#1}}}
\newcommand\zfc{\axiomfont{ZFC}}
\newcommand\gch{\axiomfont{GCH}}
\newcommand\mm{\axiomfont{MM}}
\newcommand\pfa{\axiomfont{PFA}}

\newcommand\sh{\axiomfont{SH}}
\newcommand\on{\axiomfont{ON}}

\newcommand\ap{\textup{AP}}
\newcommand\ch{\textup{CH}}
\newcommand\ns{\textup{NS}}

\DeclareMathOperator{\p}{P}

\subjclass[2010]{Primary 03E05; Secondary 03E65, 03E35, 05C05}
\keywords{Souslin tree construction, microscopic approach, unified approach, parameterized proxy principle,
coherence relation, slim tree, complete tree, coherent tree, diamond principle, square principle, axioms, consistency.}

\begin{document}
\begin{abstract} We propose a parameterized proxy principle from which $\kappa$-Souslin trees with various additional features can be constructed, regardless of the identity of $\kappa$.
We then introduce \emph{the microscopic approach}, which is a simple method for deriving trees from instances of the proxy principle.
As a demonstration, we give a construction of a coherent $\kappa$-Souslin tree that applies also for $\kappa$ inaccessible.

We then carry out a systematic study of the consistency of instances of the proxy principle, distinguished by the vector of parameters serving as its input.
Among other things, it will be shown that all known $\diamondsuit$-based constructions of $\kappa$-Souslin trees
may be redirected through this new proxy principle.
\end{abstract}
\author{Ari Meir Brodsky}
\author{Assaf Rinot}
\address{Department of Mathematics, Bar-Ilan University, Ramat-Gan 5290002, Israel.}
\urladdr{http://u.math.biu.ac.il/~brodska/}
\urladdr{http://www.assafrinot.com}

\title[Microscopic approach. Part I]{A Microscopic approach to Souslin-tree constructions. Part I}

\maketitle
\section{Introduction}
Recall that the real line is that unique {separable}, dense linear ordering with no endpoints in which
every bounded set has a least upper bound. A problem posed by Mikhail Souslin around the year 1920 \cite{Souslin} asks whether
the term \emph{separable} in the above characterization may be weakened to \emph{ccc}. (A linear order is said to be \emph{separable}
if it has a countable dense subset. It is \emph{ccc} --- short for satisfying the \emph{countable chain condition} --- if every pairwise-disjoint family of open intervals is countable.)
The affirmative answer to Souslin's problem is known as Souslin's Hypothesis, and abbreviated \sh.\footnote{For more on the history of Souslin's problem, see the surveys \cite{MR0270322}, \cite{MR1711574} and \cite{MR2882649}.}

Amazingly enough, the resolution of this single problem led to key discoveries in set theory:
the notions of Aronszajn, Souslin and Kurepa trees \cite{kurepa1935ensembles},
forcing axioms and the method of iterated forcing \cite{MR0294139},
the diamond and square principles $\diamondsuit(S)$, $\square_\kappa$ \cite{MR309729},
and the theory of iteration without adding reals \cite{MR384542}.

\medskip

Kurepa \cite{kurepa1935ensembles} proved that \sh\ is equivalent to the assertion that every tree of size $\aleph_1$ contains either an uncountable chain
or an uncountable antichain. A counterexample tree is said to be a \emph{Souslin tree}. This concept admits a natural generalization to higher cardinals,
in the form of $\kappa$-Souslin trees for regular uncountable cardinals $\kappa$.
There is a zoo of consistent constructions of $\kappa$-Souslin trees,
and these constructions often depend on whether $\kappa$ is a successor of regular, a successor of singular of countable cofinality,
a successor of singular of uncountable cofinality, or an inaccessible. This poses challenges, since, in applications, it is often needed that the constructed Souslin trees
have additional features (such as rigidity, homogeneity, or admitting/omitting an ascent path),
and this necessitates revisiting each of the relevant constructions and tailoring it to the new need.
Let us exemplify.

\begin{example} The behavior of $\kappa$-Aronszajn and $\kappa$-Souslin trees depends heavily on the identity of $\kappa$:\footnote{All tree-related terminology will be defined in Section~2 below.}
\begin{itemize}
\item (Aronszajn, see \cite[\S9.5, Theorem~6 and footnote~3, p.~96]{kurepa1935ensembles}) There exists a $\kappa$-Aronszajn tree, for $\kappa=\aleph_1$;
\item (Specker, \cite{MR0039779}) If \gch\ holds, then for every regular cardinal $\lambda$, there exists a special $\lambda^+$-Aronszajn tree;
\item (Magidor-Shelah, \cite{MgSh:324}) \gch\ is consistent with the nonexistence of any $\lambda^+$-Aronszajn tree at some singular cardinal $\lambda$ (modulo large cardinals);
\item (Erd\H{o}s-Tarski, \cite{MR0167422}) If $\kappa$ is a weakly compact cardinal, then there exists no $\kappa$-Aronszajn tree;
\item (Jensen, see \cite{MR384542}) The existence of an $\omega_1$-Souslin tree is independent of \gch;
\item (Baumgarner-Malitz-Reinhardt, \cite{MR0314621}) Any $\omega_1$-Aronszajn tree can be made special in some cofinality-preserving extension;
\item (Devlin, \cite{MR732661}) If $V=L$, then for every uncountable cardinal $\lambda$, there exists a $\lambda^+$-Souslin tree
that {remains non-special} in any cofinality-preserving extension.
\end{itemize}
\end{example}

\begin{example} Consistent constructions of $\kappa$-Souslin trees that depend on the identity of $\kappa$:\footnote{See the appendix below for a list of combinatorial principles, and some of their known interactions.}
\begin{itemize}
\item (Jensen, \cite{MR309729})
Suppose that $\lambda$ is a regular cardinal.

If {$\lambda^{<\lambda}=\lambda$ and $\diamondsuit(E^{\lambda^+}_\lambda)$ holds}, then there exists a  $\lambda^+$-Souslin tree;\footnote{
Here, $E^\beta_\alpha$ denotes the set of all ordinals below $\beta$ whose cofinality is equal to $\alpha$.
The sets  $E^\beta_{\geq\alpha}$, $E^\beta_{<\alpha}$ and $E^\beta_{\neq\alpha}$ are  defined in a similar fashion.}

\item (Jensen, \cite{MR309729})
Suppose that $\lambda$ is a singular cardinal.

If $\gch+\square_\lambda$ holds, then there exists a $\lambda^+$-Souslin tree;
\item (Baumgartner, see \cite{MR732661}, building on Laver \cite{LvSh:104}) {$\gch+\square_{\aleph_1}$} implies the existence of an $\aleph_2$-Souslin tree
that remains non-special in any cofinality-preserving extension;
\item (Cummings, \cite{MR1376756}) Suppose that $\lambda$ is a regular uncountable cardinal.

If $\sd_{\lambda}$ holds, and $\lambda^{<\lambda}=\lambda$, then there exists a $\lambda$-complete $\lambda^+$-Souslin tree
that remains non-special in any cofinality-preserving extension;
\item (Cummings, \cite{MR1376756})  Suppose that $\lambda$ is a singular cardinal of {countable cofinality}.

If $\square_\lambda+\ch_\lambda$ holds, and $\mu^{\aleph_1}<\lambda$ for all $\mu<\lambda$,
then there exists a $\lambda^+$-Souslin tree that remains non-special in any cofinality-preserving extension;\footnote{We write  $\ch_\lambda$ for the assertion that $2^\lambda=\lambda^+$.}

\item (Cummings, \cite{MR1376756})  Suppose that $\lambda$ is a singular cardinal of {uncountable cofinality}.

If $\square_\lambda+\ch_\lambda$ holds, and $\mu^{\aleph_0}<\lambda$ for all $\mu<\lambda$,
then there exists a $\lambda^+$-Souslin tree that remains non-special in any cofinality-preserving extension;
\item (Jensen, see \cite{MR384542}) $\diamondsuit(\omega_1)$ entails a homogeneous $\omega_1$-Souslin tree;
\item (Rinot, \cite{rinot11}) Suppose that $\lambda$ is a singular cardinal.

If $\square_\lambda+\ch_\lambda$ holds, then there exists a homogeneous $\lambda^+$-Souslin tree.
\end{itemize}
\end{example}

The focus of the present research project, of which this paper is a core component,
is on developing new foundations for constructing $\kappa$-Souslin trees.
Specifically, we propose a single parameterized proxy principle from which $\kappa$-Souslin trees with various additional features can be constructed, regardless of the identity of $\kappa$.

In this paper and in the next one \cite{rinot23} (being Part~I and Part~II, respectively) we establish, among other things, that all known $\diamondsuit$-based constructions of $\kappa$-Souslin trees
may be redirected through this new proxy principle. This means that any $\kappa$-Souslin tree with additional features that will be shown to follow from the proxy principle
will automatically be known to hold in many unrelated models.

But the parameterized principle gives us more:

$\br$ It suggests a way of calibrating how fine is a particular class of Souslin trees, by pinpointing the weakest vector of parameters sufficient for the proxy principle to enable construction of a member of this class.

For instance, the existence of a \emph{coherent} Souslin tree entails the existence of a \emph{free} one (see
\cite{MR2697978},\cite{ShZa:610},\cite{MR2725894}),
while it is consistent that there exists a free $\kappa$-Souslin tree, but not a coherent one, for $\kappa=\aleph_1$ \cite{MR2697978}.
This is also consistently true for $\kappa=\aleph_2$, as can be verified by the model of \cite[Theorem 4.4]{MR631563}. And, indeed,
the vector of parameters sufficient to construct a free $\kappa$-Souslin tree is weaker than the corresponding one for a coherent $\kappa$-Souslin tree.

$\br$ It allows to compare and amplify previous results.

Recall that it is a longstanding open problem whether $\gch$ entails the existence of an $\aleph_2$-Souslin tree,
and a similar problem is open concerning $\lambda^+$-Souslin trees for $\lambda$ singular (see \cite{MR2162107} and \cite[Question 14]{rinot_s01}).
A milestone result in this vein is the result from \cite{MR485361} and its improvement \cite{KjSh:449}.
In recent years, new weak forms of $\diamondsuit$ at the level of $\lambda^+$ for $\lambda=\cf(\lambda)>\aleph_0$ were proposed and shown to entail $\lambda^+$-Souslin trees.
This includes the club-guessing principle $\curlywedge^*(\chi,S)$ from \cite{MR2320769} and the \emph{reflected-diamond} principle $\langle T\rangle_S$ from \cite{rinot09}.\footnote{See the appendix.}
In this paper, we put all of these principles under a single umbrella by computing the corresponding vector of parameters that holds in each of the previously studied configurations.
From this and the constructions we present in a future paper, it follows, for example, that the Gregory configuration \cite{MR485361} suffices for the construction of a \emph{specializable} $\lambda^+$-Souslin tree,
and the K\"onig-Larson-Yoshinobu configuration \cite{MR2320769} suffices for the construction of a free $\lambda^+$-Souslin tree.

$\br$ It allows to obtain completely new results.

Here, and mostly in the other papers of this project, we develop a very simple method for deriving Souslin trees from the proxy principle --- the \emph{microscopic approach}.
This approach involves devising a library of miniature actions and an apparatus for recursively invoking them at the right timing against a witness to the proxy principle.
The outcome will always be a tree, but its features depend on which actions where invoked along the way, and to which vector of parameters of the proxy principle the witness was given.
Once the construction of Souslin trees becomes so simple, it is then easier to carry out considerably more complex constructions, and this has already been demonstrated in \cite{rinot20},
where we gave the first example of a Souslin tree whose reduced powers behave independently of each other, proving, e.g.,
that $V=L$ entails an $\aleph_7$-Souslin tree whose reduced $\aleph_n$-power is $\aleph_7$-Aronszajn iff $n\in \{0,1,4,5\}$.

And there is another line of new results --- finding new configurations in which there exist Souslin trees,
by finding new configurations in which the proxy principle holds.
In Part~II, we shall extend Gregory's theorem \cite{MR485361} from dealing with successor of regulars to dealing also with successor of singulars.
We shall also prove that Prikry forcing over a measurable cardinal $\lambda$ validates the proxy principle, and hence introduces a $\lambda^+$-Souslin tree.

\medskip

In the next subsection, we define the proxy principle in its full generality, but before that, we would like to give two simplified versions of it, $\boxtimes^-(S)$ and $\boxtimes(S)$,
which may be thought of as generic versions of the principle $\square(\kappa)$. For this, let us set up some notation, as follows.
Suppose that $D$ is a set of ordinals.
Write $\acc^+(D) := \{\alpha<\sup(D)\mid \sup (D\cap\alpha) = \alpha>0 \}$, $\acc(D) :=D\cap\acc^+(D)$ and $\nacc(D) := D \setminus \acc^+(D)$.
For any  $j < \otp(D)$, denote by $D(j)$ the unique $\delta\in D$ for which $\otp(D\cap\delta)=j$,
e.g., $D(0) = \min(D)$.
For any ordinal $\sigma$, write
\begin{align*}
\suc_\sigma(D) &:= \{ \delta \in D \mid \otp(D \cap \delta) = j+1 \text{ for some } j < \sigma \}\\
&=\{ D(j+1)\mid j<\sigma\ \&\ j+1<\otp(D)\}.
\end{align*}
Notice that always $\suc_\sigma(D) \subseteq \nacc(D) \setminus \{\min(D)\}$,
and if $D$ is a closed subset of  $\sup(D)$, then $\suc_\sigma(D) = \nacc(D) \setminus \{\min(D)\}$
provided that $\sigma \geq \otp(D) > 0$.

\begin{definition}
For any regular uncountable cardinal $\kappa$ and any stationary $S\s\kappa$,
$\boxtimes^-(S)$ asserts
the existence of a sequence $\langle C_\alpha\mid\alpha<\kappa\rangle$ such that:
\begin{itemize}
\item $C_\alpha$ is a club subset of $\alpha$ for every limit ordinal $\alpha<\kappa$;
\item $C_{\bar\alpha}=C_\alpha\cap\bar\alpha$ for every ordinal $\alpha < \kappa$ and every $\bar\alpha\in\acc(C_\alpha)$;
\item for every cofinal $A\s\kappa$, there exist stationarily many $\alpha\in S$ such that $\sup(\nacc(C_\alpha)\cap A)=\alpha$.
\end{itemize}
\end{definition}

In Section 2, we shall present a construction of a \emph{slim} $\kappa$-Souslin tree from $\boxtimes^-(\kappa)+\diamondsuit(\kappa)$.
We hope that the reader would find this construction appealing for its simplicity and uniform nature, allowing, e.g., a single construction of a $\kappa$-Souslin tree that is relevant in $L$ to any regular uncountable cardinal
that is not weakly compact. As it is trivial to prove that $\diamondsuit(\omega_1)\Rightarrow\clubsuit(\omega_1)\Rightarrow \boxtimes^-(\omega_1)$,
we think that this exposition is altogether preferable even at the level of $\omega_1$.

We shall also prove that modulo a standard arithmetic hypothesis, $\boxtimes^-(E^\kappa_{\ge\chi})+\diamondsuit(\kappa)$ entails a $\chi$-complete $\kappa$-Souslin tree.
Note that unlike the classical approach that derived a $\chi$-complete $\kappa$-Souslin tree from $\diamondsuit(E^\kappa_{\ge\chi})$,
here we settle for $\diamondsuit(\kappa)$. To appreciate this difference, we mention that the model of Example~\ref{prop112} below
witnesses that $\boxtimes(E^{\aleph_2}_{\aleph_1})+\diamondsuit(\omega_2)$ is consistent together with the failure of $\diamondsuit(E^{\aleph_2}_{\aleph_1})$.

But we haven't yet defined the stronger principle $\boxtimes(S)$. We do so now:

\begin{definition}
For any regular uncountable cardinal $\kappa$ and any stationary $S\s\kappa$,
$\boxtimes(S)$ asserts
the existence of a sequence $\langle C_\alpha\mid\alpha<\kappa\rangle$ such that:
\begin{itemize}
\item $C_\alpha$ is a club subset of $\alpha$ for every limit ordinal $\alpha<\kappa$;
\item $C_{\bar\alpha}=C_\alpha\cap\bar\alpha$ for every ordinal $\alpha < \kappa$ and every $\bar\alpha\in\acc(C_\alpha)$;
\item for every sequence $\langle A_i\mid i<\kappa\rangle$ of cofinal subsets of $\kappa$, there exist stationarily many $\alpha\in S$ such that
$\sup\{\beta<\alpha\mid \suc_\omega(C_\alpha\setminus\beta)\s A_i\}=\alpha$ for all $i<\alpha$.
\end{itemize}
\end{definition}

Furthermore,
in Section 2, we shall present a construction of a coherent $\kappa$-Souslin tree from $\boxtimes(\kappa)+\diamondsuit(\kappa)$.
In Section 3, we shall prove that $\square_\lambda+\ch_\lambda$ entails $\boxtimes(\lambda^+)$ for every singular cardinal $\lambda$,
and that $\diamondsuit(\omega_1)$ entails $\boxtimes(\omega_1)$. Note that neither of the two implications is trivial.

In the next paper, Part~II, we shall introduce the considerably weaker principle $\boxtimes^*(S)$ (a relative of Jensen's weak square principle $\square^*$),
which still suffices to entail the existence of a (plain) $\kappa$-Souslin tree, in the presence of $\diamondsuit(\kappa)$.

\subsection{The Proxy Principle}

\begin{definition}[Proxy principle]\label{proxy}
Suppose that:
\begin{itemize}
\item
$\kappa$ is any regular
uncountable cardinal;
\item
$\nu$ and $\mu$ are cardinals such that $2 \leq \nu \leq \mu \leq \kappa$;
\item
$\mathcal R$ is a binary relation over $[\kappa]^{<\kappa}$;
\item
$\theta$ is a cardinal such that $1 \leq \theta \leq \kappa$;
\item
$\mathcal S$ is a nonempty collection of stationary subsets of $\kappa$;
\item
$\sigma$ is an ordinal $\leq \kappa$;
and
\item
$\mathcal E$ is an equivalence relation over a subset of $\mathcal P(\kappa)$.
\end{itemize}
The principle
$\p^-(\kappa, \mu, \mathcal R, \theta, \mathcal S,  \nu,\sigma,\mathcal E)$
asserts the existence of a sequence
$\left< \mathcal C_\alpha \mid \alpha < \kappa \right>$
such that:
\begin{itemize}
\item for every limit $\alpha < \kappa$, $\mathcal C_\alpha$ is a collection of club subsets of $\alpha$;
\item for every ordinal $\alpha < \kappa$, $0 < \left| \mathcal C_\alpha \right| < \mu$, and $C \mathrel{\mathcal E}D$ for all $C,D\in\mathcal C_\alpha$;
\item for every ordinal $\alpha < \kappa$, every $C \in \mathcal C_\alpha$, and every $\bar\alpha \in \acc(C)$, there exists $D \in \mathcal C_{\bar\alpha}$ such that $D \mathrel{\mathcal R} C$;
\item
for every sequence $\langle A_i \mid i < \theta \rangle$ of cofinal subsets of $\kappa$,
and every $S \in \mathcal S$, there exist stationarily many $\alpha \in S$ for which:
\begin{itemize}
\item $\left| \mathcal C_\alpha \right| < \nu$; and
\item
for every $C \in \mathcal C_\alpha$ and $i < \min\{\alpha, \theta\}$:
\[
\sup\{ \beta \in C \mid \suc_\sigma (C \setminus \beta) \subseteq A_i \} = \alpha.
\]
\end{itemize}
\end{itemize}

We shall sometimes say that \emph{the sequence has width ${<\mu}$},
and refer to the \emph{$\mathcal R$-coherence} of the sequence.

\begin{definition} $\p(\kappa, \mu, \mathcal R, \theta, \mathcal S,  \nu,\sigma,\mathcal E)$
asserts that both $\p^-(\kappa, \mu, \mathcal R, \theta, \mathcal S,  \nu,\sigma,\mathcal E)$
and $\diamondsuit(\kappa)$ hold.
\end{definition}

We will often shorten the statement of the proxy principle
(both $\p^-(\kappa, \dots)$ and $\p(\kappa, \dots)$)
when some of the parameters take on their weakest useful values,
as follows:
\begin{itemize}
\item
If we omit $\mathcal E$, then $\mathcal E= (\mathcal P(\kappa))^2$.
\item
If we also omit $\sigma$,
then $\sigma = 1$.
\item
If we also omit $\nu$,
then $\nu = \mu$.
\item
If we also omit $\mathcal S$, then $\mathcal S = \{\kappa\}$.
\item
If we also omit $\theta$, then $\theta = 1$.
\end{itemize}
\end{definition}

\begin{definition}
The binary relations $\mathcal R$ over $[\kappa]^{<\kappa}$ used as the third parameter
in the proxy principle include $\sqsubseteq$, $\sqsubseteq^*$, $\sqx$, $\sqx^*$, and $\sqsubseteq_\chi$,
which we define as follows, where $\chi \leq \kappa$ can be any ordinal:
\begin{itemize}
\item
$D \sqsubseteq C$ iff
there exists some ordinal $\beta < \kappa$ such that $D = C \cap \beta$,
that is, $C$ \emph{end-extends} $D$;
\item
$D \sqsubseteq^* C$ iff there exists $\gamma < \sup(D)$ such that
$D \setminus \gamma \sqsubseteq C \setminus \gamma$,
that is, $C$ end-extends $D$ modulo a subset bounded below $\sup D$;
\item $D \sqx C$ iff (($D \sqsubseteq C$) or ($\cf(\sup(D))<\chi$));
\item $D \sqx^* C$ iff (($D \sqsubseteq^* C$) or ($\cf(\sup(D)) <\chi$));
\item $D \sqsubseteq_\chi C$ iff (($D \sqsubseteq C$) or ($\otp(C)<\chi$ and $\nacc(C)$ consists only of successor ordinals)).
\end{itemize}
\end{definition}

\begin{definition}
The equivalence relations $\mathcal E$ over a subset of $\mathcal P(\kappa)$ used as the eighth parameter
in the proxy principle include the default $(\mathcal P(\kappa))^2$, as well as $=^*$ and $\mathcal E_\chi$,
which we define as follows, where $\chi \leq \kappa$ can be any ordinal:
\begin{itemize}
\item $D=^*C$ iff there exists some $\gamma<\sup(D)$ such that $D\setminus\gamma=C\setminus\gamma$;
\item $D\mathrel{\mathcal E_\chi}C$ iff ($(\otp(D)\le\chi)$ and $(\otp(C)\le\chi)$).
\end{itemize}
\end{definition}

We shall establish below that by an appropriate choice of a vector of parameters, the proxy  principles $\p^-$ and $\p$ allow to express
many of the combinatorial principles considered in the appendix, including $\clubsuit_w(S)$, $\diamondsuit(S)$, $\langle\lambda\rangle^-_S$, $\square_\lambda$, $\boxminus_{\lambda,\ge\chi}$, $\sd_\lambda$, and $\square_\lambda+\ch_\lambda$.

\medskip

At some point, we realized that our paper is growing large and decided to split it into two parts.
This paper, being Part~I, concentrates solely on the case where the second parameter of the proxy principle
takes its strongest possible value, that is, $\mu=2$. Part~II concentrates on the case $2<\mu\le\kappa$.

What is so nice about the case $\mu=2$, is that for every ordinal $\alpha < \kappa$, $\mathcal C_\alpha$ is a singleton, say $\{C_\alpha\}$,
and the parameter $\nu$ conveys no additional information.
Therefore, in this case, it makes sense to simplify the notation and say that a sequence $\langle C_\alpha \mid \alpha < \kappa\rangle$ witnesses
$\p^-(\kappa, 2, \mathcal R, \theta, \mathcal S, 2, \sigma, \mathcal E)$
whenever:
\begin{itemize}
\item
for every limit $\alpha < \kappa$, $C_\alpha$ is a club subset of $\alpha$, and $C_\alpha\in\dom(\mathcal E)$;
\item
for all ordinals $\alpha < \kappa$
and $\bar\alpha \in \acc(C_\alpha)$,
we have
$C_{\bar\alpha} \mathrel{\mathcal R} C_\alpha$;
\item
for every sequence $\langle A_i \mid i < \theta \rangle$ of cofinal subsets of $\kappa$,
and every $S \in \mathcal S$, there exist stationarily many $\alpha \in S$ satisfying,
for every $i < \min\{\alpha, \theta\}$:
\[
\sup\{ \beta \in C_\alpha \mid \suc_\sigma (C_\alpha \setminus \beta) \subseteq A_i \} = \alpha.
\]
\end{itemize}

To conclude this subsection, let us point out how the simplified axioms from the previous subsection
may be defined using the proxy principle $\p^-$.

\begin{definition} For a regular uncountable cardinal $\kappa$ and a stationary subset $S\s\kappa$:
\begin{itemize}
\item $\boxtimes(S)$ asserts that $\p^-(\kappa,2,{\sq},\kappa,\{S\},2,\omega)$ holds;
\item $\boxtimes^-(S)$ asserts that $\p^-(\kappa,2,{\sq},1,\{S\},2,1)$ holds;
\item $\boxtimes^*(S)$ asserts that $\p^-(\kappa,\kappa,{\sqleft{\chi}^*},1,\{S\},\kappa,1)$ holds for $\chi:=\min\{\cf(\alpha)\mid \alpha\in S\text{ limit}\}$.
\end{itemize}
\end{definition}

We similarly define $\boxtimes_\lambda(S), \boxtimes^-_\lambda(S),\boxtimes^*_\lambda(S)$ by letting the eighth parameter be $\mathcal E_\lambda$ in each respective part of the above definition.

\subsection{Sample Corollaries} To give an idea of the flavor of consequences the results of this paper entail, we state here a few sample corollaries.
We remind the reader that the definitions of all relevant combinatorial principles may be found in the appendix section below.

Our first corollary lists sufficient conditions for the proxy principle to hold
with its parameters taking on their strongest useful values:

\begin{cor}\label{sample1}
$\p(\kappa,2,{\sq},\kappa,\mathcal S,2,\omega)$ holds,
assuming any of the following:
\begin{enumerate}
\item $\kappa=\aleph_1, \mathcal S=\{S\}, S\s\omega_1$ and $\diamondsuit(S)$ holds;
\item $\kappa=\lambda^+$, $\lambda$ is a singular cardinal, $\mathcal S=\{E^{\lambda^+}_{\cf(\lambda)}\}$ and $\square_\lambda+\ch_\lambda$ holds;
\item $\kappa=\lambda^+$, $\lambda$ is a regular uncountable cardinal, $\mathcal S=\{E^{\lambda^+}_\lambda\}$, and $\sd_\lambda$ holds;
\item $\kappa=\lambda^+$, $\lambda$ is not subcompact, $\mathcal S=\{E^{\lambda^+}_{\cf(\lambda)}\}$ and $V$ is a Jensen-type extender model of the form $L[E]$;
\item $\kappa$ is a regular uncountable cardinal that is not weakly compact, $\mathcal S=\{E^\kappa_{\ge\chi}\mid \chi<\kappa\ \&\ \forall\lambda<\kappa(\lambda^{<\chi}<\kappa)\}$ and $V=L$;
\item $\kappa=\lambda^+$, $\lambda$ is regular uncountable, $\mathcal S=\{S\s E^{\lambda^+}_\lambda\mid S\text{ is stationary}\}$ and $V=W^{\add(\lambda,1)}$,\footnote{Here, $\add(\lambda,1)$ stands for the notion of forcing for adding a Cohen subset to $\lambda$.}
where $$W\models \zfc+\square_\lambda+\ch_\lambda+ \lambda^{<\lambda}=\lambda.$$
\end{enumerate}
\end{cor}
\begin{proof}
\begin{enumerate}
\item By Theorem \ref{thm32}.
\item By Corollary \ref{cor64}.
\item By Theorem \ref{thm32}.
\item For $\kappa=\aleph_1$, appeal to Clause~(1); For $\kappa=\lambda^+$ where $\lambda$ is singular, appeal to Clause~(2) and \cite[Theorem 0.1]{MR2081183}. Finally, if $\kappa$ is a successor of a regular uncountable cardinal $\lambda$ that is not subcompact,
then appeal to Clause~(3) and \cite[Theorem 2.3]{MR2713313}, taking $D=E^{\lambda^+}_\lambda$ there.
\item For $\kappa$ successor, this follows from Clause~(4) and the fact that no cardinal in $L$ is subcompact. For $\kappa$ inaccessible not weakly compact, appeal to Fact~\ref{thm66}.
\item By Theorem~\ref{thm67}.
\qedhere
\end{enumerate}
\end{proof}

\begin{cor}\label{sample2} Suppose that $\ch_\lambda$ holds for a regular uncountable cardinal $\lambda$, and $S\s E^{\lambda^+}_\lambda$ is stationary.
Then $(1)\Rightarrow(2)\Rightarrow(3)\Rightarrow(4)\Leftrightarrow(5)\Leftarrow(6)$:
\begin{enumerate}
\item $\diamondsuit(S)$;
\item $\curlywedge^-(\chi,S)$ for all uncountable cardinals $\chi\le\lambda$;
\item $\curlywedge^-(\chi,S)$ for some uncountable cardinal $\chi\le\lambda$;
\item $\langle\lambda\rangle^-_S$;
\item $\p(\lambda^+,2,{\sql},\lambda^+,\{S\},2,\omega,\mathcal E_\lambda)$;
\item $V=W^{\add(\lambda,1)}$, where $W\models\zfc+\ch_\lambda+\lambda^{<\lambda}=\lambda$.
\end{enumerate}
\end{cor}
\begin{proof} $(1)\Rightarrow(2)$ By Fact~\ref{factRinot09}.

$(3)\Rightarrow(4)$ By Theorem \ref{thm312}.

$(4)\Leftrightarrow(5)$ By Corollary~\ref{c55}.

$(6)\Rightarrow(5)$ By Theorem~\ref{thm313}.
\end{proof}

\begin{cor}\label{c115} Suppose that $\boxminus_{\lambda,\ge\chi}+\ch_\lambda$ holds for given infinite cardinals $\cf(\chi)=\chi\le\theta<\lambda$. Then:
\begin{enumerate}
\item $\p(\lambda^+,2,{\sq_\chi},\theta,\{E^{\lambda^+}_\eta\mid \aleph_0 \leq \cf(\eta) = \eta < \lambda\},2,\omega,{\mathcal E_\lambda})$ holds;
\item If $\lambda$ is singular, then $\p(\lambda^+,2,{\sq_\chi},\lambda^+,\{E^{\lambda^+}_{\cf(\lambda)}\},2,\omega,\mathcal E_\lambda)$ holds;
\item If $\lambda$ is regular, then $\p(\lambda^+,2,{\sqx^*},\chi,\{E^{\lambda^+}_{\lambda}\})$ holds.
\end{enumerate}
\end{cor}
\begin{proof}
(1) By Theorem~\ref{thm38}(1) and Theorem~\ref{boxminus-succ}(1).

(2) By Theorem \ref{thm38}(3).

(3) By Corollary~\ref{c93}(2).
\end{proof}

Note that if $\chi=\aleph_0$, then $\boxminus_{\lambda,\ge\chi}$, $\sq_{\chi}$ and $\sqleft{\chi}^*$ are respectively equivalent to $\square_\lambda$, $\sq$ and $\sq^*$.

\begin{cor} Suppose that $\square_\lambda+\ch_\lambda$ holds for a given uncountable cardinal $\lambda$.

Then $\p(\lambda^+,2,{\sq^*},1,\{E^{\lambda^+}_{\cf(\lambda)}\},2,\omega)$ holds.
\end{cor}
\begin{proof} By Corollary~\ref{c93}(1) with $\chi = \aleph_0$ for $\lambda$ regular, and Corollary~\ref{cor64} for $\lambda$ singular.
\end{proof}
\begin{cor} If $\ch_\lambda$ holds for a given regular uncountable cardinal $\lambda$,
and there exists a nonreflecting stationary subset of $E^{\lambda^+}_{<\lambda}$, then
$\p(\lambda^+,2,{\sql}^*,\theta,\{E^{\lambda^+}_{\lambda}\},2,\omega)$ holds for all regular cardinals
$\theta <\lambda$.
\end{cor}
\begin{proof} By Theorem~\ref{thm317}.
\end{proof}

So far, it seems like all of our hypotheses are in the spirit of ``$V=L$''. The next model shows that the proxy principle is also consistent with strong forcing axioms.
\begin{cor}\label{c113} Assuming the consistency of a supercompact cardinal,
there exists a model of \zfc\ that satisfies simultaneously:
\begin{enumerate}
\item Martin's Maximum holds, and hence:
\begin{enumerate}
\item $\square^*_\lambda$ fails for every singular cardinal $\lambda$ of countable cofinality;
\item $\square_{\lambda,\aleph_1}$ fails for every regular uncountable cardinal $\lambda$;
\end{enumerate}
\item $\p(\lambda^+,2,{\sq_{\aleph_2}},\lambda^+,\{E^{\lambda^+}_{\cf(\lambda)}\},2,\omega, \mathcal E_\lambda)$ holds for every singular cardinal $\lambda$;
\item $\p(\lambda^+,2,{\sql},\lambda^+,\{E^{\lambda^+}_\lambda\},2,\omega, \mathcal E_\lambda)$ holds for every regular uncountable cardinal $\lambda$.
\end{enumerate}
\end{cor}
\begin{proof} The definitions of $\square^*_\lambda$ and $\square_{\lambda,\aleph_1}$ may be found in \cite{MR2811288},
and the model we construct is a slight variation of the model $V_3$ from Section~3 of that paper.
Specifically, we start by working in the model $V_1$ from \cite[$\S3$]{MR2811288} so that
$\kappa$ is a supercompact cardinal indestructible under $(<\kappa)$-directed-closed notions of forcing, and $\ch_\lambda$ holds for all cardinals $\lambda\ge\kappa$.
Now we do an iteration of length $\on$ with Easton support;
for each singular cardinal $\lambda>\kappa$, we force with Baumgartner's poset $Q(\kappa,\lambda)$,\footnote{See \cite[$\S3$]{MR2811288} or the proof of Corollary~\ref{cor116} below.}
and for each regular cardinal $\lambda\ge\kappa$, we force with Cohen's poset $\add(\lambda^+,1)$.
The resulting model $V_2$ satisfies:
\begin{itemize}
\item Cardinals and cofinalities are preserved;
\item $\kappa$ is supercompact;
\item $\ch_\lambda$ holds for all cardinals $\lambda\ge\kappa$;
\item $\boxminus_{\lambda,\ge\kappa}$ holds for every singular cardinal $\lambda>\kappa$;
\item $\diamondsuit(E^{\lambda^+}_\lambda)$ holds for every regular cardinal $\lambda\ge\kappa$.
\end{itemize}
Now we force over $V_2$ with the standard forcing for the consistency of \mm. The
forcing poset is semi-proper and \emph{$\kappa$-cc} with cardinality $\kappa$.
After forcing we get a model $V_3$ in which:
\begin{itemize}
\item $\aleph_1$ is preserved, and $\kappa$ is the new $\aleph_2$;
\item  \mm\ holds. In particular, $2^{\aleph_0}=2^{\aleph_1}=\aleph_2$;
\item $\ch_\lambda$ holds for all cardinals $\lambda\ge\aleph_2$;
\item $\boxminus_{\lambda,\ge\aleph_2}$ holds for every singular cardinal $\lambda$;
\item $\diamondsuit(E^{\lambda^+}_\lambda)$ holds for every regular cardinal $\lambda\ge\aleph_2$.
\end{itemize}

By \mm\ and \cite{MR2811288}, Clauses (a) and (b) hold.
Since \mm\ implies \pfa,
we get from \cite[$\S7$]{MR776640} that $\diamondsuit(E^{\lambda^+}_\lambda)$ holds also for $\lambda=\aleph_1$.

$\br$ By  Theorem~\ref{thm616a}, we infer that $\p(\lambda^+,2,{\sql},\lambda^+,\{E^{\lambda^+}_\lambda\},2,\omega,\mathcal E_\lambda)$ holds for every regular uncountable cardinal $\lambda$.

$\br$ By Theorem~\ref{thm38}, we infer that $\p(\lambda^+,2,{\sq_{\aleph_2}},\lambda^+,\{E^{\lambda^+}_{\cf(\lambda)}\},2,\omega,\mathcal E_\lambda)$ holds
for every singular cardinal $\lambda$.
\end{proof}

Here is another anti-``$V=L$'' scenario.

\begin{cor} If $\lambda$ is a successor of a regular cardinal $\theta$, and $\ns\restriction E^{\lambda}_{\theta}$ is saturated,
then $\ch_\lambda$ entails $\p(\lambda^+,2,{\sql}^*,\theta,\{E^{\lambda^+}_{\lambda}\},2,\theta)$.

\end{cor}
\begin{proof} This is Theorem~\ref{cor329}.
\end{proof}

To justify each of the eight parameters of the proxy principle, for each parameter (except for the second and sixth, which always have value $2$ in this paper), we give an example where one value fails and another one holds.
Note that the techniques given here give rise to many more models distinguishing various vectors of parameters.
We also remark that a thorough study of implications between various vectors will be carried out in Part~II.
\begin{example}[Distinguishing the 1st parameter] The following statements are mutually consistent:
\begin{itemize}
\item $\p(\aleph_1,2,{\sq})$ fails;
\item $\p(\aleph_2,2,{\sq})$ holds.
\end{itemize}
\end{example}
\begin{proof} Force over $L$ with a \emph{ccc} poset to get a model in which Martin's axiom holds and $2^{\aleph_0}=\aleph_2$.
As Martin's axiom holds, there are no $\aleph_1$-Souslin trees, and hence $\p(\aleph_1,2,{\sq})$,
which is the same as $\boxtimes^-(\kappa) + \diamondsuit(\kappa)$, fails by Proposition~\ref{prop23}.
In the extension, as we have $\square_{\aleph_1}+\ch_{\aleph_1}$, by Corollary~\ref{thm61},
$\p(\aleph_2, 2, {\sq}, 1, \{\aleph_2\}, 2, \omega_2, \mathcal E_{\omega_1})$ holds, let alone
$\p(\aleph_2, 2, {\sq})$.
\end{proof}

\begin{example}[Distinguishing the 3rd parameter] Relative to a weakly compact cardinal,
the following statements are mutually consistent:
\begin{itemize}
\item $\p(\aleph_2,2,{\sqleft{\aleph_0}})$ fails;
\item $\p(\aleph_2,2,{\sqleft{\aleph_1}})$ holds.
\end{itemize}
\end{example}
\begin{proof} Suppose that in $V$, $\kappa$ is weakly compact, and let $\mathbb P:=\col(\aleph_1,{<\kappa})$ be Levy's notion of forcing for collapsing $\kappa$ to $\omega_2$.
Note that as $\kappa$ is strongly inaccessible,
for every $\alpha<\kappa$, the collection $\mathcal N_\alpha:=\{ \tau\in V_{\alpha+1}\mid \tau\text{ is a }\mathbb P\text{-name}\}$ has size $<\kappa$.

Let $G$ be $\mathbb P$-generic over $V$, and work in $V[G]$.

For all $\alpha<\kappa$, let $\mathcal A_\alpha:=\{ \tau_G\mid \tau\in\mathcal N_\alpha\}\cap\mathcal P(\alpha)$.
Since $\mathbb P$ has the \emph{$\kappa$-cc}, $\langle \mathcal A_\alpha\mid \alpha<\kappa\rangle$ forms a $\diamondsuit^+(\aleph_2)$-sequence,
and in particular, $\diamondsuit(E^{\aleph_2}_{\aleph_1})$ holds.\footnote{See \cite{MR597342} for the definition of $\diamondsuit^+(\kappa)$
and the proof that it implies $\diamondsuit(S)$ for every stationary $S\s\kappa$.}
So, by Corollary~\ref{sample2}, $\p(\aleph_2,2,{\sqleft{\aleph_1}},\aleph_2,\{E^{\aleph_2}_{\aleph_1}\},2,\omega, \mathcal E_{\omega_1})$ holds,
let alone $\p(\aleph_2,2,{\sqleft{\aleph_1}})$.

As for the first bullet, by \cite[Theorem 5]{MR830071}, $\square(\aleph_2)$ fails.
And then, by Lemma~\ref{l1023}, $\p(\aleph_2,2,{\sq})$ fails as well.
\end{proof}

\begin{example}[Distinguishing the 3rd parameter]\label{prop116} Relative to a supercompact cardinal, the following statements are mutually consistent for some uncountable limit cardinals $\kappa<\lambda$:
\begin{itemize}
\item $\p(\lambda^+,2,{\sq_\theta})$ fails for all $\theta<\kappa$;
\item $\p(\lambda^+,2,{\sq_\kappa})$ holds.
\end{itemize}
\end{example}
\begin{proof} Work in the model from Corollary~\ref{cor116}. Then $\kappa$ is supercompact and $\p(\lambda^+,2,{\sq_\kappa},\lambda^+,\{\lambda^+\},2,1,\mathcal E_\lambda)$ holds,
for $\lambda=\kappa^{+\omega}$.
In particular, $\p(\lambda^+,2,{\sq_\kappa})$ holds, establishing the second bullet.

As for the first bullet, as $\kappa$ is supercompact,
we have that any pair of stationary subsets of $E^{\lambda^+}_{<\kappa}$ reflect simultaneously.
It now follows from Lemma~\ref{prop118} that  $\p^-(\lambda^+,2,{\sq_\theta})$ must fail for all $\theta<\kappa$.
\end{proof}

\begin{example}[Distinguishing the 4th parameter]\label{prop114} The following statements are mutually consistent:
\begin{itemize}
\item $\p^-(\aleph_2,2,{\sqleft{\aleph_1}},\aleph_1,\{\aleph_2\},2,2,\mathcal E_{\omega_1})$ fails;
\item $\p^-(\aleph_2,2,{\sqleft{\aleph_1}},\aleph_0,\{\aleph_2\},2,2,\mathcal E_{\omega_1})$ holds.
\end{itemize}
\end{example}
\begin{proof} Let $\mathbb P$ be, in $L$, the forcing notion from \cite{MR3307877}.
Then $\mathbb P$ is a $\sigma$-closed, cofinality-preserving notion of forcing,
and in $L^{\mathbb P}$, for every sequence $\langle C_\delta\mid \delta\in E^{\aleph_2}_{\aleph_1}\rangle$
of local clubs of order-type $\omega_1$, there exists some club $D\s\omega_2$
for which $\sup\{ \beta<\delta\mid \suc_2(C_\delta\setminus\beta)\s D\}<\delta$ for all $\delta\in E^{\aleph_2}_{\aleph_1}$.

Work in the extension. Towards a contradiction, suppose that $\p^-(\aleph_2,2,{\sqleft{\aleph_1}},\aleph_1,\{\aleph_2\},2,2,\mathcal E_{\omega_1})$ holds,
as witnessed by $\langle C_\delta\mid \delta<\omega_2\rangle$. Let $D$ be a club such that $\sup\{ \beta<\delta\mid \suc_2(C_\delta\setminus\beta)\s D\}<\delta$ for all $\delta\in E^{\aleph_2}_{\aleph_1}$.
Let $\langle A_i \mid i<\omega_1\rangle$ be some partition of $D$ into stationary sets.
Then, there must exist some limit ordinal $\delta<\omega_2$ such that $\sup\{ \beta<\delta\mid \suc_2(C_\delta\setminus\beta)\s A_i\}=\delta$ for all $i<\omega_1$.
As $A_i\cap A_j=\emptyset$ for all $i<j<\omega_1$, we infer that $|C_\delta|\ge\aleph_1$.
Recalling that the eighth parameter is $\mathcal E_{\omega_1}$, we conclude that $\cf(\delta)=\omega_1$,
thus, yielding a contradiction, and establishing the first bullet.

As $L^{\mathbb P}$ is a $\sigma$-closed, cofinality-preserving extension of $L$, we have that $\square_{\aleph_1}+\diamondsuit(\omega_1)$ holds in the extension,
let alone $\square_{\aleph_1}+\clubsuit(\omega_1)+(\aleph_2)^{\aleph_0}=\aleph_2$.
By $\square_{\aleph_1}+\clubsuit(\omega_1)$ and the main result of \cite{MR1218206}, there exists a sequence $\langle S_\alpha\mid \alpha\in E^{\aleph_2}_{\aleph_0}\rangle$ such that for every $\alpha\in E^{\aleph_2}_{\aleph_0}$, $S_\alpha$ is a cofinal subset of $\alpha$ of order-type $\omega$,
and such that for every uncountable $X\s\omega_2$, there exists some $\alpha\in E^{\aleph_2}_{\aleph_0}$ for which $S_\alpha\s X$.
Using the fact that $(\aleph_2)^{\aleph_0}=\aleph_2$, it is now easy to build a witness to $\p^-(\aleph_2,2,{\sqleft{\aleph_1}},\aleph_0,\{E^{\aleph_2}_{\aleph_0}\},2,2,\mathcal E_{\omega_1})$,
establishing the second bullet.
\end{proof}

\begin{example}[Distinguishing the 5th parameter] The following statements are mutually consistent:
\begin{itemize}
\item $\p^-(\aleph_2,2,{\sqleft{\aleph_1}},1,\{E^{\aleph_2}_{\aleph_1}\},2,2,\mathcal E_{\omega_1})$ fails;
\item $\p^-(\aleph_2,2,{\sqleft{\aleph_1}},1,\{E^{\aleph_2}_{\aleph_0}\},2,2,\mathcal E_{\omega_1})$ holds.
\end{itemize}
\end{example}
\begin{proof} This is the same model and the virtually the same proof as Example~\ref{prop114}.
\end{proof}

\begin{example}[Distinguishing the 7th parameter]\label{prop112} The following statements are mutually consistent:
\begin{itemize}
\item $\p(\aleph_2,2,{\sq},\aleph_2,\{E^{\aleph_2}_{\aleph_1}\},2,\omega_1,\mathcal E_{\omega_1})$ fails;
\item $\p(\aleph_2,2,{\sq},\aleph_2,\{E^{\aleph_2}_{\aleph_1}\},2,\omega,\mathcal E_{\omega_1})$ holds.
\end{itemize}
\end{example}
\begin{proof} Let $\mathbb P$ be, in $L$, the forcing notion from \cite{MR597452},
so that $L^{\mathbb P}\models \neg\diamondsuit(E^{\aleph_2}_{\aleph_1})$.
Then  $\mathbb P$ is $\sigma$-closed, $\omega_1$-distributive, and has the \emph{$\aleph_3$-cc}.
So $L$ and $L^{\mathbb P}$ share the same cardinals structure, $\gch$ holds in the extension,
and so does $\square_{\aleph_1}$.
Now, force with $\add(\omega_1,1)$ over $L^{\mathbb P}$. By Theorem \ref{thm67a}(2),
 $\p(\aleph_2,2,{\sq},\aleph_2,\{E^{\aleph_2}_{\aleph_1}\},2,\omega,\mathcal E_{\omega_1})$ holds in the extension.
Since $\add(\omega_1,1)$ has the \emph{$\aleph_2$-cc}, a well-known argument of Kunen
entails that $\diamondsuit(E^{\aleph_2}_{\aleph_1})$ remains failing in the extension, and then Theorem~\ref{6.22} finishes the proof.
\end{proof}

\begin{example}[Distinguishing the 8th parameter] Relative to a Mahlo cardinal, the following statements are mutually consistent for some cardinal $\kappa$:
\begin{itemize}
\item $\p(\kappa,2,{\sq},1,\{\kappa\},2,1,\mathcal E_{\chi})$ fails for all $\chi<\kappa$;
\item $\p(\kappa,2,{\sq},1,\{\kappa\},2,1,\mathcal E_{\kappa})$ holds.
\end{itemize}
\end{example}
\begin{proof} Work in $L$, and let $\kappa$ be a Mahlo cardinal that is not weakly compact (say, the first Mahlo).
Since $\{ \alpha<\kappa\mid \cf(\alpha)=\alpha\}$ is cofinal in $\kappa$,
$\p^-(\kappa,2,{\sq},1,\{\kappa\},2,1,\mathcal E_{\chi})$ fails for all $\chi<\kappa$.
On the other hand, by Theorem~\ref{thm39}, $\p(\kappa,2,{\sq},1,\{\kappa\},2,\kappa)$ holds.
\end{proof}

\subsection{Organization of this paper} In Section 2, we fix some terminology and notation,
demonstrate the microscopic approach by constructing various $\kappa$-Souslin trees from instances of the proxy principle,
and highlight the differences between the classical approach and the new one.

Sections 3 to 6 are the heart of the matter, were we build the bridge between the old foundations and the new one.
That is, we establish instances of the proxy principle from various hypotheses that were previously known to entail $\kappa$-Souslin trees.
The division between these sections is based on the third parameter of the proxy principle. Specifically, Section~3 deals with $\sq$, Section~4 with $\sq_\chi$,
Section~5 with $\sql$, and Section~6 deals with $\sqx^*$. Modulo arithmetic hypotheses, each of these coherence relations suffices for the construction of $\kappa$-Souslin trees (though, of varying quality).

Finally, Section~7 is an appendix that briefly provides the necessary background regarding the combinatorial principles used in this context.

\section{Constructing Souslin trees from the proxy principle}
\label{section:constructions}

In this section we demonstrate the microscopic approach by constructing various $\kappa$-Souslin trees from instances of the proxy principle.
We begin by recalling the relevant terminology and fixing some notation.

A \emph{tree} is a partially ordered set $(T,<_T)$ with the property that for every $x\in T$,
the downward cone $x_\downarrow:=\{ y\in T\mid y<_T x\}$ is well-ordered by $<_T$. The \emph{height} of $x\in T$, denoted $\height(x)$,
is the order-type of $(x_\downarrow,<_T)$. Then, the $\alpha^{\text{th}}$ level of $(T,<_T)$ is the set $T_\alpha:=\{x\in T\mid \height(x)=\alpha\}$.
We also write $T \restriction X := \{t \in T \mid \height(t) \in X \}$.
A tree  $(T,<_T)$ is said to be \emph{$\chi$-complete} if any $<_T$-increasing sequence of elements from $T$,
and of length ${<\chi}$, has an upper bound in $T$. On the other extreme,
the tree  $(T,<_T)$ is said to be \emph{slim} if $\left|T_\alpha\right| \leq \max \{\left|\alpha\right|, \aleph_0\}$ for every ordinal $\alpha$.
A tree $(T,<_T)$ is said to be \emph{normal} if for all ordinals $\alpha < \beta$
and every $x \in T_\alpha$, if $T_\beta\neq\emptyset$
then there exists some $y \in T_\beta$ such that $x <_T y$.
A tree $(T,<_T)$ is said to be \emph{splitting} if every node in $T$ admits at least two immediate successors.

Throughout, let $\kappa$ denote a regular uncountable cardinal.
A tree $(T,<_T)$ is a \emph{$\kappa$-tree} whenever $\{ \alpha\mid T_\alpha\neq\emptyset\}=\kappa$, and $|T_\alpha|<\kappa$ for all $\alpha<\kappa$.
A subset $B\s T$ is a \emph{cofinal branch} if $(B,<_T)$ is linearly ordered and  $\{ \height(t)\mid t\in B\}=\{\height(t)\mid t\in T\}$.
A \emph{$\kappa$-Aronszajn tree} is a $\kappa$-tree with no cofinal branches.
A \emph{$\kappa$-Souslin tree} is a $\kappa$-Aronszajn that has no antichains of size $\kappa$.

Of special interest is the case where $<_T$ is simply $\subset$, and $T$ is a downward closed subset of ${}^{<\kappa}\kappa$.
The trees that we construct here will be of this form.
In such a setup,
each node $t$ of the tree $T$ is a function $t:\alpha \to \kappa$ for some ordinal $\alpha < \kappa$,
and we require that if $t:\alpha \to\kappa$ is in $T$, then $t \restriction \beta \in T$ for every $\beta < \alpha$.
For any node $t \in T$, the height of $t$ in $T$ is just its domain, that is, $\height(t) = \dom(t)$,
and  its set of predecessors, $t_\downarrow$, is simply $\{ t \restriction \beta \mid \beta < \dom(t) \}$.
Note that $T_\alpha=T \cap {}^\alpha\kappa$ for every $\alpha<\kappa$.
Finally, any function $f : \kappa \to\kappa$ determines a cofinal branch through $^{<\kappa}\kappa$, namely
$\{ f \restriction \alpha \mid \alpha < \kappa \}$, which ought not to end up being a subset of $T$
if $T$ is to form a $\kappa$-Aronszajn tree.

The main advantage of this approach is the ease of completing a branch at a limit level.
Suppose that, during the process of constructing $T$, we have already inserted into $T$
a $\subset$-increasing sequence of nodes $\eta:=\langle t_\alpha \mid \alpha < \beta\rangle$ for some $\beta<\kappa$.
The (unique) limit of this sequence, which may or may not become a member of $T$,
is nothing but $\bigcup\rng(\eta)$, that is, $\bigcup_{\alpha<\beta}t_\alpha$.
Furthermore, compatibility of nodes in the tree is easily expressed:
For $x, y \in T$, $x$ and $y$ are compatible iff $x \cup y \in T$.

A subtree $T\s{}^{<\kappa}\kappa$ is \emph{prolific} if for every $\alpha<\kappa$ and every $t\in T\cap{}^\alpha\kappa$, we have $\{ t{}^\smallfrown\langle i\rangle \mid i<\max \{\omega, \alpha\}\}\s T$.
Notice that a prolific tree is always splitting. On the opposite extreme from prolific,
a $\kappa$-tree is said to be \emph{binary} if it is a downward-closed
subtree of the complete binary tree ${}^{<\kappa}2$.
A subtree $T\s{}^{<\kappa}\kappa$ is \emph{coherent} if for every $\alpha<\kappa$ and $s,t\in T\cap{}^\alpha\kappa$, the set $\{ \beta<\alpha\mid s(\alpha)\neq t(\alpha)\}$ is finite.

While classical constructions of $\kappa$-Souslin trees typically involve
a recursive process of determining a partial order $<_T$ over $\kappa$
by advising with a $\diamondsuit(\kappa)$-sequence,
here, the order is already known (being $\subset$),
and the recursive process involves the determination of a subset of ${}^{<\kappa}\kappa$.
For this reason, it is more convenient to work with the following variation of $\diamondsuit(\kappa)$:

\begin{definition}\label{def_Diamond_H_kappa}
$\diamondsuit(H_\kappa)$ asserts the existence of a partition $\langle R_i \mid  i < \kappa \rangle$ of $\kappa$
and a sequence $\langle \Omega_\beta \mid \beta < \kappa \rangle$ of subsets of $H_\kappa$
such that for every $p\in H_{\kappa^{+}}$,  $i<\kappa$, and $\Omega \subseteq H_\kappa$,
there exists an elementary submodel $\mathcal M\prec H_{\kappa^{+}}$ such that:
\begin{itemize}
\item $p\in\mathcal M$;
\item $\mathcal M\cap\kappa\in  R_i$;
\item $\mathcal M\cap \Omega=\Omega_{\mathcal M\cap\kappa}$.
\end{itemize}
\end{definition}

\begin{lemma}\label{Diamond_H_kappa}
$\diamondsuit(\kappa)$ is equivalent to $\diamondsuit(H_\kappa)$ for any regular uncountable cardinal $\kappa$.
\end{lemma}

\begin{proof}
\begin{description}
\item[($\impliedby$)]
Given $\langle R_i\mid i<\kappa\rangle$ and $\langle \Omega_\beta\mid \beta<\kappa\rangle$ as in the definition of $\diamondsuit(H_\kappa)$, let for all $\beta<\kappa$:
\[
Z_\beta:=\begin{cases}\Omega_\beta,&\text{if }\Omega_\beta\s\beta;\\\emptyset,&\text{otherwise.}\end{cases}
\]

To show that $\left< Z_\beta \mid \beta < \kappa \right>$ is a $\diamondsuit(\kappa)$-sequence,
consider any $Z\s\kappa$, and we must show that $\{ \beta < \kappa \mid Z \cap \beta = Z_\beta \}$ is stationary.
Thus, let $D\s\kappa$ be a club, and we must find some $\beta \in D$ such that $Z \cap \beta = Z_\beta$.
Put $p:=D$. As $p\in H_{\kappa^{+}}$, and $Z\s\kappa\s H_\kappa$,
we may pick $\mathcal M\prec H_{\kappa^{+}}$ with $D\in\mathcal M$ such that $\mathcal M\cap\kappa\in\kappa$ and $\mathcal M\cap Z=\Omega_{\mathcal M\cap\kappa}$.
Let $\beta:=\mathcal M\cap\kappa$.
Then $\Omega_\beta = \mathcal M\cap Z=\mathcal M\cap\kappa\cap Z=\beta\cap Z$
and hence $Z_\beta=\Omega_\beta=Z\cap\beta$.
As  $D$ is club in $\kappa$ and $D\in\mathcal M$, we have (by elementarity of $\mathcal M$)
$\beta=\sup(\mathcal M\cap\kappa)\in D$.

\item[($\implies$)]
By $\diamondsuit(\kappa)$, fix $\langle Z_\beta\mid\beta<\kappa\rangle$ such that $\{ \beta<\kappa\mid Z\cap\beta=Z_\beta\}$ is stationary for all $Z\s\kappa$.
Fix also bijections $\phi:\kappa\setminus\{0\}\leftrightarrow H_\kappa$ and $\pi:\kappa\times\kappa\leftrightarrow\kappa$. Let:
\begin{itemize}
\item $\Omega_\beta:=\{\phi(\alpha)\mid \pi(\alpha,1)\in Z_\beta, \alpha\neq 0\}$;
\item $R_0:=\kappa\setminus\bigcup_{0<i<\kappa}R_i$, where for nonzero $i<\kappa$:
\item $R_i:=\{\beta\in\acc(\kappa)\mid \{ j<\kappa \mid \pi(0,j)\in Z_\beta\}=\{i\}\}$.
\end{itemize}

To see that $\langle \Omega_\beta\mid \beta<\kappa\rangle$ and $\langle R_i\mid i<\kappa\rangle$ are as requested,
let $p\in H_{\kappa^{+}}$, $i<\kappa$, and $\Omega\s H_\kappa$ be arbitrary.
Define $f:\kappa\rightarrow\kappa$ by letting for all $\alpha<\kappa$:
\[
f(\alpha):=\begin{cases}i,&\text{if }\alpha=0;\\
1,&\text{if }\alpha>0\ \&\ \phi(\alpha)\in\Omega;\\
0,&\text{otherwise.}\end{cases}
\]
Let $p':=\{p,\phi,\pi,f\}$. Let $\langle M_\beta\mid \beta<\kappa\rangle$ be an $\epsilon$-chain of elementary submodels of $H_{\kappa^{+}}$, each of size $<\kappa$
and containing $p'$. Evidently, $E:=\{\beta<\kappa\mid M_\beta\cap\kappa=\beta\}$ is a club in $\kappa$.
Recalling that $f \subseteq \kappa \times \kappa$, so that $\pi[f] \subseteq \kappa$,
consider the stationary set of guesses $G:=\{ \beta<\kappa\mid \pi[f]\cap\beta=Z_\beta\}$,
and pick a nonzero $\beta\in E\cap G$.
We shall show that $\mathcal M = M_\beta$ satisfies the required properties.

By $\pi,f\in M_\beta$ and $M_\beta\cap\kappa=\beta$, we have $f[\beta]\s\beta$ and $\pi[\beta\times\beta]=\beta$. Thus, $\pi[f\restriction\beta]=\pi[f]\cap\beta=Z_\beta$.

\begin{claim} $\beta\in R_i$.
\end{claim}
\begin{proof}
Clearly $\beta \in \acc(\kappa)$ since $\beta = M_\beta \cap \kappa$.
Since $Z_\beta \subseteq \beta$, also $\pi^{-1}[Z_\beta] \subseteq \beta \times \beta$,
so that $\{ j<\kappa \mid \pi(0,j)\in Z_\beta\} \subseteq \beta$.
Also $i = f(0) < \beta$.
Furthermore, for all $j < \beta$, we have
\begin{align*}
\pi(0,j)\in Z_\beta &\iff f(0)=j  &&\text{since } \pi[f\restriction\beta] = Z_\beta  \\
                              &\iff j=i       &&\text{by definition of } f.
\end{align*}
By definition of $R_i$, it follows that
$\beta\in R_i$.
\end{proof}

\begin{claim}
$\phi[\beta \setminus \{0\}]=M_\beta\cap H_\kappa$.
\end{claim}
\begin{proof}
For all $x\in H_\kappa$, by $\phi\in M_\beta$, we have $x\in M_\beta$ iff $\phi^{-1}(x)\in M_\beta\cap\dom(\phi)$ iff $\phi^{-1}(x)\in\beta \setminus \{0\}$.
\end{proof}
\begin{claim} $M_\beta\cap\Omega=\Omega_{\beta}$
\end{claim}
\begin{proof}
We have
\begin{align*}
M_\beta\cap\Omega &= M_\beta\cap H_\kappa\cap\Omega  &&\text{since } \Omega\s H_\kappa  \\
                                  &= \phi[\beta \setminus \{0\}]\cap\Omega &&\text{by previous claim}          \\
                                  &= \{ \phi(\alpha)\mid 0<\alpha<\beta, f(\alpha)=1\} &&\text{by definition of } f  \\
                                  &= \{\phi(\alpha)\mid 0\neq \alpha, \pi(\alpha,1)\in Z_\beta\}
                                                                 &&\text{since } \pi[f\restriction\beta]=Z_\beta  \\
                                  &= \Omega_\beta            &&\text{by definition of } \Omega_\beta.
\qedhere
\end{align*}
\end{proof}
So $\mathcal M=M_\beta$ is as required.
\qedhere
\end{description}
\end{proof}

We commence with a simple construction using $\boxtimes^-(\kappa)$.

\begin{prop}\label{prop23} If $\kappa$ is a regular uncountable cardinal and $\boxtimes^-(\kappa)+\diamondsuit(\kappa)$ holds, then there exists a normal, binary, splitting, slim $\kappa$-Souslin tree.
\end{prop}
\begin{proof} Let $\langle C_\alpha\mid \alpha<\kappa\rangle$ be a witness to $\boxtimes^-(\kappa)$.
Let $\langle R_i \mid i<\kappa\rangle$ and $\langle \Omega_\beta\mid\beta<\kappa\rangle$ together witness $\diamondsuit(H_\kappa)$. Let $\lhd$ be some well-ordering of $H_\kappa$.
We shall recursively construct a sequence $\langle T_\alpha\mid \alpha<\kappa\rangle$ of levels
whose union will ultimately be the desired tree $T$.

Let $T_0:=\{\emptyset\}$, and for all $\alpha<\kappa$, let $T_{\alpha+1}:=\{ t{}^\smallfrown\langle0\rangle, t{}^\smallfrown\langle1\rangle\mid t\in T_\alpha\}$.

Next, suppose that $\alpha$ is a nonzero limit ordinal, and that $\langle T_\beta\mid \beta<\alpha\rangle$ has already been defined.
Constructing the level $T_\alpha$ involves deciding which branches
through $(T \restriction \alpha, \subset)$ will have their limits placed into the tree.
We need $T_\alpha$ to contain enough nodes to ensure that the tree is normal,
so the idea is to attach to each node $x \in T \restriction C_\alpha$ some node $\mathbf b^\alpha_x:\alpha\rightarrow2$  above it,
and then let $$T_\alpha := \{ \mathbf b^\alpha_x \mid x \in T \restriction C_\alpha\}.$$

Let $x\in T\restriction C_\alpha$ be arbitrary.
As $\mathbf b^\alpha_x$ will be the limit of some branch through $(T\restriction\alpha,\subset)$ and above $x$,
it makes sense to describe $\mathbf b^\alpha_x$ as the limit $\bigcup\rng(b^\alpha_x)$ of a sequence $b^\alpha_x\in\prod_{\beta\in C_\alpha\setminus\dom(x)}T_\beta$ such that:
\begin{itemize}
\item $b^\alpha_x(\dom(x))=x$;
\item $b_x^\alpha(\beta') \subset b_x^\alpha(\beta)$ for all $\beta'<\beta$ in $(C_\alpha\setminus\dom(x))$;
\item $b^\alpha_x(\beta)=\bigcup\rng(b^\alpha_x\restriction\beta)$ for all $\beta\in\acc(C_\alpha\setminus\dom(x))$.
\end{itemize}
Of course, we have to define $b^\alpha_x$ carefully, so that the resulting tree doesn't include large antichains.
We do this by recursion:

Let $b^\alpha_x(\dom(x)):=x$.
Next, suppose $\beta^-<\beta$ are successive points of $(C_\alpha\setminus\dom(x))$, and $b^\alpha_x(\beta^-)$ has already been defined.
In order to decide $b^\alpha_x(\beta)$, we advise with the following set:
$$Q^{\alpha, \beta}_x := \{ t\in T_\beta\mid \exists s\in \Omega_{\beta}[ (s\cup b^\alpha_x(\beta^-))\s t]\}.$$
Now, consider the two possibilities:
\begin{itemize}
\item If $Q^{\alpha,\beta}_x \neq \emptyset$, let $b^\alpha_x(\beta)$ be its $\lhd$-least element.
\item Otherwise, let $b^\alpha_x(\beta)$ be the $\lhd$-least element of $T_\beta$ that extends $b^\alpha_x(\beta^-)$.
Such an element must exist, as the level $T_\beta$ was constructed so as to preserve normality.
\end{itemize}

Finally, suppose $\beta \in \acc(C_\alpha\setminus\dom(x))$ and $b^\alpha_x\restriction\beta$ has already been defined.
As promised, we let $b^\alpha_x(\beta):=\bigcup\rng(b^\alpha_x\restriction\beta)$.
It is clear that $b^\alpha_x(\beta) \in {}^\beta 2$,
but we need more than that:

\begin{claim}\label{coherence}
$b^\alpha_x (\beta) \in T_\beta$.
\end{claim}

\begin{proof}
It suffices to prove that $b^\alpha_x \restriction \beta=b^\beta_x$, as this will imply that $b^\alpha_x(\beta)=\bigcup\rng(b^\beta_x)=\mathbf{b}^\beta_x\in T_\beta$.

First, note that since $\beta \in \acc(C_\alpha)$ and  $\langle C_\alpha \mid \alpha < \kappa \rangle$ is a $\boxtimes^-(\kappa$)-sequence,
we have $\dom(b^\beta_x) = C_\beta \setminus \dom(x) =
C_\alpha \cap \beta \setminus \dom(x) = \dom(b^\alpha_x) \cap \beta$. Call the latter by $d$.
Now, we prove by induction that
for every $\gamma \in d$,
the value of $b^\beta_x(\gamma)$ was determined in exactly the same way as $b^\alpha_x(\gamma)$:
\begin{itemize}
\item Clearly, $b^\beta_x(\min(d)) = x = b^\alpha_x(\min(d))$.
\item Suppose $\gamma^-<\gamma$ are successive points of $d$.
Notice that the definition of $Q^{\alpha, \gamma}_x$ depends only on
$b^\alpha_x(\gamma^-)$, $\Omega_\gamma$, and $T_\gamma$,
and so if $b^\alpha_x (\gamma^-) = b^\beta_x(\gamma^-)$,
then $Q^{\alpha, \gamma}_x = Q^{\beta, \gamma}_x$,
and hence $b^\alpha_x(\gamma) = b^\beta_x(\gamma)$.

\item For $\gamma \in \acc(d)$:
If the sequences are identical up to $\gamma$, then their limits must be identical. \qedhere
\end{itemize}
\end{proof}

This completes the definition of $b^\alpha_x$ for each $x\in T\restriction C_\alpha$,
and hence of the level $T_\alpha$.

Having constructed all levels of the tree, we then let
\[
T := \bigcup_{\alpha < \kappa} T_\alpha.
\]

Notice that for every $\alpha < \kappa$,
$T_\alpha$ is a subset of $^\alpha2$ of size $\le\max\{\aleph_0,|\alpha|\} < \kappa$.

\begin{claim}\label{c232}
Suppose $A \subseteq T$ is a maximal antichain.
Then the set
\[
B := \{ \beta \in R_0 \mid  A\cap(T\restriction\beta)= \Omega_\beta\text{ is a maximal antichain in }T\restriction\beta \}.
\]
is a stationary subset of $\kappa$.
\end{claim}

\begin{proof}
Let $D\s\kappa$ be an arbitrary club.
We must show that $D \cap B \neq \emptyset$.
Put $p := \{A,T,D\}$.
Using the fact that the sequences $\langle R_i \mid i<\kappa\rangle$ and $\langle \Omega_\beta\mid\beta<\kappa\rangle$
witness $\diamondsuit(H_\kappa)$,
pick $\mathcal M\prec H_{\kappa^{+}}$
with $p\in\mathcal M$ such that $\beta :=\mathcal M\cap\kappa$ is in $R_0$ and $\Omega_\beta=\mathcal M \cap A$.
Since $D\in\mathcal M$ and $D$ is club in $\kappa$, we have $\beta\in D$. We claim that $\beta\in B$.

For all $\alpha<\beta$, by $\alpha,T\in \mathcal M$, we have $T_\alpha\in \mathcal M$,
and by $\mathcal M\models |T_\alpha|<\kappa$,
we have $T_\alpha\s \mathcal M$. So $T\restriction\beta\s \mathcal M$.
As $\dom(z)\in \mathcal M$ for all $z \in T\cap \mathcal M$,
we conclude that $T\cap \mathcal M=T\restriction\beta$.
Thus, $\Omega_\beta=A\cap(T\restriction\beta)$.
As $H_{\kappa^{+}}\models A\text{ is a maximal antichain in }T$,
it follows by elementarity that
$\mathcal M \models A\text{ is a maximal antichain in }T$.
Since $T\cap \mathcal M=T\restriction \beta$,
we get that $A\cap(T\restriction\beta)$ is a maximal antichain in $T\restriction\beta$.
\end{proof}

It is clear that $(T, \subset)$ is a normal, binary, splitting, slim $\kappa$-tree.
We would like to prove that it is $\kappa$-Souslin.
As any splitting $\kappa$-tree with no antichains of size $\kappa$ also has no chains of size $\kappa$,
it suffices to prove the following.

\begin{claim}\label{c233}
If $A \subseteq T$ is a maximal antichain, then $|A|<\kappa$.
\end{claim}
\begin{proof}
Let $A \subseteq T$ be a maximal antichain.
By the previous claim,
$B := \{ \beta \in R_0 \mid  A\cap(T\restriction\beta)= \Omega_\beta\text{ is a maximal antichain in }T\restriction\beta \}$
is a cofinal subset of $\kappa$.
Thus we apply $\boxtimes^-(\kappa)$ to obtain
a limit ordinal
$\alpha < \kappa$ satisfying
\[
\sup (\nacc(C_\alpha) \cap B) = \alpha.
\]

We shall prove that $A \subseteq T \restriction \alpha$,
from which it follows that $|A|\le|\alpha|<\kappa$.

To see that $A \subseteq T \restriction \alpha$, consider any $z \in T \restriction (\kappa \setminus \alpha)$,
and we will show that $z \notin A$ by finding some element of $A \cap (T \restriction \alpha)$ compatible with $z$.

Since $\dom(z) \geq \alpha$,
we can let $y := z \restriction \alpha$. Then $y \in T_\alpha$ and $y \subseteq z$.
By construction, $y = \mathbf b^\alpha_x = \bigcup_{ \beta \in C_\alpha \setminus \dom(x) } b^\alpha_x (\beta)$ for some $x \in T \restriction C_\alpha$.
Fix $\beta \in \nacc(C_\alpha) \cap B$ with $\height(x) < \beta < \alpha$.
Denote $\beta^-:=\sup(C_\alpha\cap\beta)$.
Since $\beta \in B$,
we know that $\Omega_\beta = A \cap (T \restriction \beta)$ is a maximal antichain in $T \restriction \beta$,
and hence there is some $s \in \Omega_\beta$ compatible with $b^\alpha_x(\beta^-)$,
so that by normality of the tree, $Q^{\alpha, \beta}_x \neq \emptyset$.
It follows that we chose $b^\alpha_x(\beta)$ to extend some $s \in \Omega_\beta$.
Altogether, $s \subseteq b^\alpha_x(\beta) \subset \mathbf b^\alpha_x = y \subseteq z$.
Since $s$ is an element of the antichain $A$, the fact that $z$ extends $s$ implies that $z \notin A$.
\end{proof}
\end{proof}

Let us briefly compare the above construction with Jensen's classical one \cite{MR309729}, as rendered in  \cite[Theorem IV.2.4]{MR750828}.
Both constructions consist of a recursive process of determining the levels $T_\alpha$ of the ultimate tree $T$, and both constructions face the same two challenges:
\begin{enumerate}
\item Maintaining the ability to climb up through the levels while keeping their width small.
\item Sealing antichains so that if $A \subseteq T$ is a maximal antichain,
then there would be some level $\alpha$,
where every node placed into $T_\alpha$ is compatible with some element of $A \cap (T \restriction \alpha)$.
\end{enumerate}

In both constructions, challenge~(1) is addressed at limit stage $\alpha<\kappa$ by attaching, for cofinally  many nodes $x\in T\restriction\alpha$,
a canonical branch $b^\alpha_x$ that goes through $x$ and climbs all the way up to level $\alpha$. In both constructions,
the coherence of the sequence $\langle C_\alpha\mid\alpha<\kappa\rangle$ entails that
 $b^{\bar\alpha}_x=b^\alpha_x\restriction\bar\alpha$ whenever $\bar\alpha\in\acc(C_\alpha)$,
thereby ensuring that the recursive process of constructing $b^\alpha_x$ never gets stuck.
So where is the difference?

The difference is in the way that we seal antichains. In the above construction,
we let $T_\alpha := \{ \mathbf b^\alpha_x \mid x \in T \restriction C_\alpha\}$ uniformly for all limit levels $\alpha<\kappa$,
whereas in the classical one, there is some stationary set $E$ of levels $\alpha<\kappa$, were maximal antichains are sealed
by letting $T_\alpha$ be some carefully chosen subset of $\{ \mathbf b^\alpha_x \mid x \in T \restriction \alpha\}$. But, of course, the latter puts the analogue of Claim~\ref{coherence} in danger!
To overcome this, Jensen ensured that the ordinals in $E$ never occur as accumulation points of $C_\alpha$ for any $\alpha<\kappa$.
While being a successful remedy, it means that the classical approach is based on \emph{sealing antichains at the levels of some nonreflecting stationary set}.\footnote{ As a matter of fact,
we are not aware of any previous $\diamondsuit$-based construction whose sealing process does not center on a nonreflecting stationary set.
The only candidate we could find in the literature that may involve sealing at a reflecting set is Theorem 4 from \cite{MR836425}, but that theorem must be false in light of Theorem 3.1 of \cite{MgSh:324}.}

In the microscopic approach, sealing of antichains can be done at the levels coming from any stationary set.
For example, in \cite{rinot20}, we constructed a free $\aleph_{\omega+1}$-Souslin tree in a model of Martin's Maximum (\mm),\footnote{Indeed, in the model of Corollary~\ref{c113} above.}
where the sealing of antichains was done at levels $\alpha<\aleph_{\omega+1}$ of countable cofinality,
even though $\mm$ implies that every stationary subset of $E^{\aleph_{\omega+1}}_\omega$ reflects.

\begin{prop}\label{prop24} If $\kappa$ is a regular uncountable cardinal and $\chi < \kappa$ satisfies $\lambda^{<\chi}<\kappa$ for all $\lambda<\kappa$, then $\boxtimes^-(E^\kappa_{\ge\chi})+\diamondsuit(\kappa)$ entails a normal, binary, splitting, $\chi$-complete $\kappa$-Souslin tree.
\end{prop}
\begin{proof} Let $\langle C_\alpha\mid \alpha<\kappa\rangle$ be a witness to $\boxtimes^-(E^\kappa_{\ge\chi})$.
Let $\langle R_i \mid i<\kappa\rangle$ and $\langle \Omega_\beta\mid\beta<\kappa\rangle$ together witness $\diamondsuit(H_\kappa)$. Let $\lhd$ be some well-ordering of $H_\kappa$.
As before, we shall recursively construct a sequence $\langle T_\alpha\mid \alpha<\kappa\rangle$ of levels
whose union will ultimately be the desired tree $T$.

Let $T_0:=\{\emptyset\}$, and for all $\alpha<\kappa$, let $T_{\alpha+1}:=\{ t{}^\smallfrown\langle0\rangle, t{}^\smallfrown\langle1\rangle\mid t\in T_\alpha\}$.
For each $x\in T\restriction C_\alpha$, define the sequence $b^\alpha_x\in\prod_{\beta\in C_\alpha\setminus\dom(x)}T_\beta$ exactly as in the proof of Proposition~\ref{prop23},
and denote its limit by $\mathbf b^\alpha_x$.
Now, there are two cases to consider:

$\br$ If $\cf(\alpha)<\chi$, let $T_\alpha$ consist of all branches through $(T\restriction\alpha,\subset)$.

$\br$ If $\cf(\alpha)\ge\chi$, let $T_\alpha:=\{ \mathbf b^\alpha_x\mid x\in T\restriction C_\alpha\}$.

Notice that in both cases, $\{ \mathbf b^\alpha_x\mid x\in T\restriction C_\alpha\}\s T_\alpha$, and $|T_\alpha|<\kappa$.

Having constructed all levels of the tree, we then let
\[
T := \bigcup_{\alpha < \kappa} T_\alpha.
\]

\begin{claim}
If $A \subseteq T$ is a maximal antichain, then $|A|<\kappa$.
\end{claim}
\begin{proof}
Let $A \subseteq T$ be a maximal antichain.
By Claim~\ref{c232},
$$B := \{ \beta \in R_0 \mid  A\cap(T\restriction\beta)= \Omega_\beta\text{ is a maximal antichain in }T\restriction\beta \}$$
is a cofinal subset of $\kappa$.
Thus we apply $\boxtimes^-(E^\kappa_{\ge\chi})$ to obtain
an ordinal $\alpha\in E^\kappa_{\ge\chi}$ satisfying
\[
\sup (\nacc(C_\alpha) \cap B) = \alpha.
\]

To see that $A \subseteq T \restriction \alpha$, consider any $z \in T \restriction (\kappa \setminus \alpha)$.
Let $y := z \restriction \alpha \in T_\alpha$.
By $\cf(\alpha)\ge\chi$, we have $y = \mathbf b^\alpha_x$ for some $x \in T \restriction C_\alpha$.
Then the same analysis of the proof of Claim~\ref{c233} entails the existence of $s\in A$ such that $s\subset\mathbf b^\alpha_x=y\s z$. In particular, $z\notin A$.
\end{proof}
So $(T,\subset)$ is a normal, binary, splitting, $\chi$-complete $\kappa$-Souslin tree.
\end{proof}

Note that the hypothesis of Proposition~\ref{prop24} involves the principle  $\diamondsuit(\kappa)$ rather than $\diamondsuit(E^\kappa_{\ge\chi})$,
and indeed, we did not advise with $\Omega_\beta$ when constructing the level $T_\beta$.
Rather, we advised with $\Omega_\beta$ at levels $\alpha > \beta$ for which $\beta \in \nacc(C_\alpha)$.
Specifically, if $\beta \in \nacc(C_\alpha)$,
$\Omega_\beta$ is a maximal antichain in $T \restriction \beta$,
and $\height(x) < \beta < \alpha$,
then the node $\mathbf b^\alpha_x$ placed into $T_\alpha$
must be compatible with some element of $\Omega_\beta$.
The proxy principle $\boxtimes^-(E^\kappa_{\ge\chi})$ then ensures that,
for any maximal antichain $A \subseteq T$, we can find some ordinal $\alpha\in E^\kappa_{\ge\chi}$ such that
$A$ was gradually sealed when building level $T_\alpha$, and therefore $A\s T\restriction\alpha$.

\medskip
A construction of a coherent Souslin tree at the level of a successor cardinal may be found in \cite{MR384542},\cite{MR1683897},\cite{MR830071}.
Here, we give a construction that applies also for inaccessible cardinals.

\begin{prop}\label{prop25} If $\kappa$ is a regular uncountable cardinal and $\boxtimes(\kappa)+\diamondsuit(\kappa)$ holds, then there exists a normal, slim, prolific, coherent $\kappa$-Souslin tree.
\end{prop}
\begin{proof} Let $\langle C_\alpha\mid \alpha<\kappa\rangle$ be a witness to $\boxtimes(\kappa)$.
Without loss of generality, we may assume that $0\in C_\alpha$ for all $\alpha<\kappa$.
Let $\langle R_i \mid i<\kappa\rangle$ and $\langle \Omega_\beta\mid\beta<\kappa\rangle$ together witness $\diamondsuit(H_\kappa)$.
Let $\pi:\kappa\rightarrow\kappa$ be such that $\alpha\in R_{\pi(\alpha)}$ for all $\alpha<\kappa$.
By $\diamondsuit(\kappa)$, we have $\left|H_\kappa\right| =\kappa$,
thus let $\lhd$ be some well-ordering of $H_\kappa$ of order-type $\kappa$,
and let $\phi:\kappa\leftrightarrow H_\kappa$ witness the isomorphism $(\kappa,\in)\cong(H_\kappa,\lhd)$.
Put $\psi:=\phi\circ\pi$.

For two elements of $\eta,\tau$ of $H_\kappa$, we define $\eta*\tau$ to be the emptyset,
unless $\eta,\tau\in{}^{<\kappa}\kappa$ with $\dom(\eta)<\dom(\tau)$, in which case $\eta*\tau:\dom(\tau)\rightarrow\kappa$ is defined by stipulating:
$$(\eta*\tau)(\beta):=\begin{cases}\eta(\beta),&\text{if }\beta\in\dom(\eta);\\
\tau(\beta),&\text{otherwise.}\end{cases}$$

We shall now recursively construct a sequence $\langle T_\alpha\mid \alpha<\kappa\rangle$ of levels
whose union will ultimately be the desired tree $T$.

Let $T_0:=\{\emptyset\}$, and for all $\alpha<\kappa$, let $T_{\alpha+1}:=\{ t{}^\smallfrown\langle i\rangle\mid t\in T_\alpha, i<\max\{\omega,\alpha\}\}$.

Next, suppose that $\alpha$ is a nonzero limit ordinal, and that $\langle T_\beta\mid \beta<\alpha\rangle$ has already been defined.
Similar to the proof of Proposition~\ref{prop23}, to each node $x \in T \restriction\alpha$ we shall associate some node $\mathbf b^\alpha_x:\alpha\rightarrow\kappa$  above $x$,
and then let $T_\alpha := \{ \mathbf b^\alpha_x \mid x \in T \restriction \alpha\}$.

Unlike the proof of Proposition~\ref{prop23}, we first define $\mathbf b^\alpha_x$ for $x=\emptyset$.

Define $b^\alpha_\emptyset\in\prod_{\beta\in C_\alpha}T_\beta$ by recursion.
Let $b^\alpha_\emptyset(0):=\emptyset$.
Next, suppose $\beta^-<\beta$ are successive points of $C_\alpha$, and $b^\alpha_\emptyset(\beta^-)$ has already been defined.
In order to decide $b^\alpha_\emptyset(\beta)$, we advise with the following set:
$$Q^{\alpha, \beta} := \{ t\in T_\beta\mid \exists s\in \Omega_{\beta}[ (s\cup (\psi(\beta)*b^\alpha_\emptyset(\beta^-)))\s t]\}.$$
Now, consider the two possibilities:
\begin{itemize}
\item If $Q^{\alpha,\beta} \neq \emptyset$, let $t$ denote its $\lhd$-least element, and put $b^\alpha_\emptyset(\beta):=b^\alpha_\emptyset(\beta^-)* t$;
\item Otherwise, let $b^\alpha_\emptyset(\beta)$ be the $\lhd$-least element of $T_\beta$ that extends $b^\alpha_\emptyset(\beta^-)$.
\end{itemize}

Note that $Q^{\alpha,\beta}$ depends only on $T_\beta,\Omega_\beta,\psi(\beta)$ and $b^\alpha_\emptyset(\beta^-)$,
and hence for every ordinal $\gamma<\kappa$, if $C_\alpha\cap(\beta+1)=C_\gamma\cap(\beta+1)$, then $b^\alpha_\emptyset\restriction(\beta+1)=b^\gamma_\emptyset\restriction(\beta+1)$.
It follows that for all $\beta \in \acc(C_\alpha)$ such that $b^\alpha_\emptyset\restriction\beta$ has already been defined,
we may let $b^\alpha_\emptyset(\beta):=\bigcup\rng(b^\alpha_\emptyset\restriction\beta)$ and infer that $b^\alpha_\emptyset(\beta)=\mathbf b^\beta_\emptyset\in T_\beta$.
This completes the definition of $b^\alpha_\emptyset$ and its limit $\mathbf b^\alpha_\emptyset=\bigcup\rng(b^\alpha_\emptyset)$.

Next, for each $x\in T\restriction\alpha$, let $\mathbf b^\alpha_x:=x*\mathbf b^\alpha_\emptyset$.
This completes the definition of the level $T_\alpha$.

Having constructed all levels of the tree, we then let
\[
T := \bigcup_{\alpha < \kappa} T_\alpha.
\]

\begin{claim} For every $\alpha<\kappa$, every two nodes of $T_\alpha$ differ on a finite set.
\end{claim}
\begin{proof} Suppose not, and let $\alpha$ be the least counterexample. Clearly, $\alpha$ must be a limit nonzero ordinal.
Pick $x,y\in T\restriction\alpha$ such that $\mathbf b^\alpha_x$ differs from $\mathbf b^\alpha_y$ on an infinite set.
As $\mathbf b^\alpha_x=x*\mathbf b^\alpha_\emptyset$ and $\mathbf b^\alpha_y=y*\mathbf b^\alpha_\emptyset$,
it follows that $x$ and $y$ differ on an infinite set, contradicting the minimality of $\alpha$.
\end{proof}
In particular, for every $\alpha < \kappa$,
$T_\alpha$ is a subset of $^\alpha\kappa$ of size $\le\max\{\aleph_0,|\alpha|\} < \kappa$.
Thus, we are left with verifying that $(T,\subset)$ is $\kappa$-Souslin. That is, establishing the following.

\begin{claim}
If $A \subseteq T$ is a maximal antichain, then $|A|<\kappa$.
\end{claim}
\begin{proof}
Let $A \subseteq T$ be a maximal antichain.
By the same proof of Claim~\ref{c232}, for every $i<\kappa$, the set $$B_i := \{ \beta \in R_i \mid A\cap(T\restriction\beta)= \Omega_\beta\text{ is a maximal antichain in }T\restriction\beta\}$$
is stationary.
Thus, we apply $\boxtimes(\kappa)$ to the sequence $\langle B_i\mid i<\kappa\rangle$, and the club $D:=\{ \alpha<\kappa\mid T\restriction\alpha\s\phi[\alpha]\}$
to obtain
an ordinal $\alpha\in D$ such that for all $i<\alpha$:
\begin{equation}\label{eq1}
\sup (\nacc(C_\alpha) \cap B_i) = \alpha.
\end{equation}

To see that $A \subseteq T \restriction \alpha$, consider any $z \in T \restriction (\kappa \setminus \alpha)$.
Let $y := z \restriction \alpha \in T_\alpha$.
By construction, $y = \mathbf b^\alpha_x = x* \mathbf b^\alpha_\emptyset$ for some $x \in T \restriction \alpha$.
As $\alpha\in D$ and $x\in T\restriction\alpha$, we can fix $i<\alpha$ such that $\phi(i)=x$.

Fix $\beta \in \nacc(C_\alpha) \cap B_i$ with $\height(x) < \beta < \alpha$. Clearly, $\psi(\beta)=\phi(\pi(\beta))=\phi(i)=x$.
Since $\beta \in B_i$,
we know that $\Omega_\beta = A \cap (T \restriction \beta)$ is a maximal antichain in $T \restriction \beta$,
and hence $Q^{\alpha, \beta}\neq\emptyset$. Let $t:=\min(Q^{\alpha,\beta},\lhd)$ and  $\beta^-:=\sup(C_\alpha\cap\beta)$.
Then $b^\alpha_\emptyset(\beta)=b^\alpha_\emptyset(\beta^-)*t$, and there exists some $s\in \Omega_{\beta}$ such that $(s\cup (x*b^\alpha_\emptyset(\beta^-)))\s t$.
In particular, $x*b^\alpha_\emptyset(\beta)$ extends an element of $\Omega_\beta$.
Altogether, there exists some $s\in A\cap (T\restriction\beta)$ such that $s\s x*b^\alpha_\emptyset(\beta)\s x*\mathbf b^\alpha_\emptyset=\mathbf b^\alpha_x=y\s z$,
and hence $z\notin S$.
\end{proof}
\end{proof}

\begin{remark} As Equation~(\ref{eq1}) above makes explicit, the preceding proof did not utilize the full force of the axiom $\boxtimes(\kappa)$.
For an application of the full force of $\boxtimes(\kappa)$, we refer the reader to Section 6 of \cite{rinot20}.
\end{remark}

Proposition~\ref{prop25} partially demonstrates the \emph{microscopic approach}: there is a stationary set $S\s\kappa$, such that for every $\alpha\in S$, every node of $T_\alpha$
is determined by an element $x\in T\restriction\alpha$, a club $C\in\mathcal C_\alpha$ and some increasing and continuous sequence $b^C_y\in\prod_{\beta\in C\setminus\dom(y)}T_\beta$ for some $y\in T\restriction C$.
The move from $b^C_y(\beta^-)$ to its successor $b^C_y(\beta)$ depends only on
 $b^C_y(\beta^-)$, and the restriction of some structure
$$\mathcal M=(H(\kappa),\in,\lhd,\psi,T,\langle R_i\mid i<\kappa\rangle,\langle\Omega_\eta\mid\eta<\kappa\rangle,\ldots)$$
to level $\beta+1$.
As we are only allowed to ``look down'', the coherence (\`{a} la Claim~\ref{coherence}) is guaranteed.

Underlying that, we have a predefined library of \emph{actions}, each labeled by a member of $H_\kappa$, and each, given a restricted structure $\mathcal M\restriction(\gamma+1)$ and an element $z\in T\restriction\gamma$,
will determine an extension of $z$ that belongs to the top level of the normal tree $T\restriction(\gamma+1)$.
Fix at the outset a subset \verb"h" of $H_\kappa$. Then, the microscopic steps from $b^C_y(\beta^-)$ to  $b^C_y(\beta)$ are the outcome of feeding $\mathcal M\restriction (\beta+1)$ and $b^C_y(\beta^-)$
to the action $\psi(\beta)$ provided that $\beta\in\verb"h"$,
or feeding them to the default action (labeled by $\emptyset\in H_\kappa$) otherwise.
Needless to say that these microscopic steps do not know where they are heading, let alone are aware of $\alpha$.

Nota bene that for different choices of $\verb"h"$ and $\langle\mathcal C_\alpha\mid\alpha<\kappa\rangle$, the above machine will produce different $\kappa$-trees.
Of interest is the analysis of \verb"h" files that include two actions of contradictory purpose (e.g., one for making the tree rigid, and the other for making the tree homogeneous).

\section{\texorpdfstring{The coherence relation $\sq$}{Full coherence}}

The relation $\sq$ is the strongest coherence relation one can expect in this context.
Note that in Section~\ref{section:constructions}, a coherent $\kappa$-Souslin tree was derived from $\p(\kappa,2,{\sq},\kappa)$.

\begin{lemma}\label{square-equiv}
For any infinite cardinal $\lambda$,
$\p^{-}(\lambda^+,2,{\sq},1,\{\lambda^+\},2,0,\mathcal E_\lambda)$ is equivalent to $\square_\lambda$.
\end{lemma}
\begin{proof} Straight-forward, but see also Lemma~\ref{l24} below.
\end{proof}

\begin{lemma}\label{l1023}
$\p^{-}(\kappa,2,{\sq})$ entails $\square(\kappa)$ for every regular uncountable cardinal $\kappa$.
\end{lemma}

\begin{proof}
Let $\langle C_\alpha\mid\alpha<\kappa\rangle$ witness $\p^{-}(\kappa,2,{\sq},1,\{\kappa\},2,1)$.
Towards a contradiction suppose that $\square(\kappa)$ fails. Then, there exists a club $C\s\kappa$ such that $C\cap\alpha=C_\alpha$ for all $\alpha\in\acc(C)$.
Since $\kappa$ is regular uncountable and $C$ is club in $\kappa$, it follows that also $\acc(C)$ is club.
Let $A_0:=\acc(C)$. Pick $\alpha\in\acc(C)$ such that $\sup(\nacc(C_\alpha)\cap A_0)=\alpha$. This is a contradiction to the fact that $\nacc(C_\alpha)\s\nacc(C), A_0=\acc(C)$,
and $\nacc(C)\cap\acc(C)=\emptyset$.
\end{proof}

\begin{lemma}\label{thm63}
Let $S\s\omega_1$ be stationary.  Then:
\begin{enumerate}
\item The following are equivalent:
\begin{enumerate}
\item  $\clubsuit_w(S)$;
\item $\p^-(\aleph_1,2,{\sq},1,\{S\},2,\omega_1,\mathcal E_\omega)$;
\item $\p^-(\aleph_1,2, \mathcal R, 1,\{S\},2,\omega_1)$ for some $\mathcal R$.
\end{enumerate}
\item The following are equivalent:
\begin{enumerate}
\item $\diamondsuit(S)$;
\item $\p(\aleph_1,2,{\sq},1,\{S\},2,\omega_1,\mathcal E_\omega)$;
\item $\p(\aleph_1,2, \mathcal R, 1,\{S\},2,\omega_1)$ for some $\mathcal R$.
\end{enumerate}
\end{enumerate}
\end{lemma}
\begin{proof} Recalling that $\sqaz$ is the same as $\sqsubseteq$, this is the case $\lambda = \aleph_0$ of Theorem~\ref{6.22} below.
\end{proof}

\begin{lemma}\label{lemma40}
If $\lambda$ is an uncountable cardinal and $\sd_\lambda$ holds,
then there exists a sequence $\langle (C_\alpha,X_\alpha)\mid \alpha<\lambda^+\rangle$
such that:
\begin{enumerate}
\item  for every limit $\alpha<\lambda^+$, $C_\alpha$ is a club in $\alpha$ of order-type $\le\lambda$, and $X_\alpha\s\alpha$;
\item if $\bar\alpha\in\acc(C_\alpha)$, then $C_\alpha\cap\bar\alpha=C_{\bar\alpha}$;
\item for every subset  $X\s\lambda^+$ and club $E\s\lambda^+$, there exists a limit $\alpha<\lambda^+$ with $\otp(C_\alpha)=\lambda$
such that $C_\alpha \s\{ \gamma\in E\mid X\cap\gamma=X_\gamma\}$.
\end{enumerate}
\end{lemma}

\begin{proof}
Pick a sequence
$\langle (D_\alpha,X_\alpha)\mid \alpha<\lambda^+\rangle$ as in Definition \ref{def-sd}.
For every $\alpha\in E^{\lambda^+}_\omega$, let $c_\alpha$ be a cofinal subset of $\alpha$ of order-type $\omega$.
Then, for every limit $\alpha<\lambda^+$, let:
\[
C_\alpha:=\begin{cases}
\acc(D_\alpha),&\text{if } \sup(\acc(D_\alpha))=\alpha; \\
c_\alpha,&\text{otherwise.}\end{cases}
\]
Let $C_{\alpha+1} :=\emptyset$ for all $\alpha<\omega_1$.

To see that $\langle (C_\alpha,X_\alpha)\mid \alpha<\lambda^+\rangle$ is as sought:
\begin{enumerate}
\item Immediate.
\item
If $\bar\alpha \in \acc(C_\alpha)$, then in particular $\acc(C_\alpha) \neq \emptyset$,
so that $\otp(C_\alpha) > \omega$, meaning that $C_\alpha = \acc(D_\alpha)$,
and $\bar \alpha \in \acc(\acc(D_\alpha)) \subseteq \acc(D_\alpha)$,
so that $D_\alpha \cap \bar \alpha = D_{\bar \alpha}$,
thus $\acc(D_\alpha) \cap \bar \alpha = \acc(D_{\bar\alpha})$.
Then $\bar\alpha = \sup (\acc(D_\alpha) \cap \bar \alpha) = \sup(\acc(D_{\bar\alpha}))$,
so that $C_{\bar\alpha} = \acc(D_{\bar\alpha}) = \acc(D_\alpha) \cap \bar \alpha = C_\alpha \cap \bar \alpha$.
\item
Given a subset  $X\s\lambda^+$ and club $E\s\lambda^+$, take $\alpha < \lambda^+$ as in clause~(3)
of Definition~\ref{def-sd}.
We will show that this $\alpha$ is as required:
Since $\otp(\acc(D_\alpha)) = \lambda$ while (by clause~(1)) $\otp(D_\alpha) = \omega \cdot \lambda$,
it follows that $\sup(\acc(D_\alpha)) = \alpha$,
so that $C_\alpha = \acc(D_\alpha)$, and
$\otp(C_\alpha) = \otp(\acc(D_\alpha)) = \lambda$.
Furthermore, $C_\alpha = \acc(D_\alpha) \subseteq E$.
Finally, for any $\gamma \in C_\alpha \subseteq \acc(D_\alpha)$,
the combination of clauses~(2) and~(3) of Definition~\ref{def-sd}
gives $X_\gamma = X_\alpha \cap \gamma = (X \cap \alpha) \cap \gamma = X \cap \gamma$,
as required.
\qedhere
\end{enumerate}
\end{proof}

\begin{lemma}\label{thm62}  Suppose that $\kappa$ is a regular uncountable cardinal, $S\s E^\kappa_\omega$ is stationary, and $\diamondsuit(S)$ holds.

Then there exists a sequence $\langle (C_\alpha,X_\alpha)\mid \alpha\in E^\kappa_\omega\rangle$ such that:
\begin{itemize}
\item  for every $\alpha\in E^\kappa_\omega$, $C_\alpha$ is a countable club in $\alpha$, and $X_\alpha\s\alpha$;
\item if $\bar\alpha\in\acc(C_\alpha)$, then $\bar\alpha\in S$ and $C_\alpha\cap\bar\alpha=C_{\bar\alpha}$;
\item for every subset  $X\s\kappa$, club $E\s\kappa$ and nonzero $\varepsilon<\omega_1$, there exists $\alpha\in S$ with $\otp(C_\alpha)=\omega\cdot\varepsilon$
such that $C_\alpha\s\{ \gamma\in E \cap S \mid X\cap\gamma=X_\gamma\}$.
\end{itemize}
\end{lemma}
\begin{proof}
By $\diamondsuit(S)$ and the implication $(1)\Rightarrow(3)$ of Fact \ref{lemma613},
let us fix a sequence $\langle (X_\delta,Y_\delta)\mid\delta<\kappa\rangle$,
and functions $\phi_0:\kappa\rightarrow\omega_1$, $\phi_1:\kappa\rightarrow\kappa$ such that for every $X,Y\in\mathcal P(\kappa)$, $\varepsilon<\omega_1$ and $\alpha<\kappa$, the following set is stationary:
\[
\{\delta\in S\mid X_\delta=X\cap\delta, Y_\delta=Y\cap\delta, \phi_0(\delta)=\varepsilon, \phi_1(\delta)=\alpha \}.
\]

In particular, we may assume that $X_\delta$ and $Y_\delta$ are subsets of $\delta$ for all $\delta<\kappa$.

We now tailor the arguments of \cite{rinot19}.

For every  $\delta\in E^\kappa_\omega$,
let $D_\delta$ be some cofinal subset of $\delta$ of order-type $\omega$,
with the additional constraint that if
$\{\gamma\in Y_\delta\setminus(\phi_1(\delta)+1)\mid X_\gamma=X_\delta\cap\gamma \text{ and } Y_\gamma=Y_\delta\cap\gamma \}$
is cofinal in $\delta$,
then ensure that $D_\delta$ is a subset of it.

We shall define a sequence $\langle C_\delta\mid\delta\in E^\kappa_\omega\cup\{0\}\rangle$ by recursion over $\delta$, as follows.

Let $C_0:=\emptyset$. Suppose that $\delta\in E^\kappa_\omega$, and $\langle C_\alpha\mid\alpha\in E_\omega^\delta\cup\{0\}\rangle$ has already been defined.
The definition of $C_\delta$ now splits into three cases:
\begin{itemize}
\item[$\br$] If $\phi_0(\delta)$ is a successor ordinal $>1$, $\phi_1(\delta)\in S\cap\delta$ and $\otp(C_{\phi_1(\delta)})=\omega\cdot(\phi_0(\delta)-1)$,
then let $$C_\delta:=C_{\phi_1(\delta)}\cup\{\phi_1(\delta)\}\cup D_\delta.$$

\item[$\br$]
If $\phi_0(\delta)$ is a limit ordinal $>0$, and there exists an increasing sequence of ordinals $\langle \alpha_n\mid n<\omega\rangle$ that converges to $\delta$,
such that $\alpha_0 = \phi_1(\delta)$, $\{ \alpha_n\mid 0<n<\omega\}\s S\cap\delta$, $\langle C_{\alpha_n}\mid n<\omega\rangle$ is $\sq$-increasing,
$X_\beta=X_\delta\cap\beta$ and $Y_\beta=Y_\delta\cap\beta$
for all $\beta\in(\bigcup_{n<\omega}C_{\alpha_n})$,
and $\otp(\bigcup_{n<\omega}C_{\alpha_n})=\omega\cdot \phi_0(\delta)$, then fix such a sequence, and let $C_\delta=\bigcup_{n<\omega}C_{\alpha_n}$.

\item[$\br$]
Otherwise, let $C_\delta:=D_\delta$.
\end{itemize}

Clearly, for any $\delta\in E^\kappa_\omega$, $C_\delta$ is a club subset of $\delta$,
and either $\otp(C_\delta)=\omega\cdot\phi_0(\delta)$ or $\otp(C_\delta)=\omega$.
In particular, for every $\delta\in E^\kappa_\omega\cup\{0\}$, there exists some $\varepsilon<\omega_1$ for which  $\otp(C_\delta) = \omega \cdot \varepsilon$.

\begin{claim} If $\delta\in E^\kappa_\omega$ and $\bar\delta\in\acc(C_\delta)$, then $\bar\delta\in S$ and $C_\delta\cap\bar\delta=C_{\bar\delta}$.
\end{claim}
\begin{proof}
Suppose not, and let $\delta$ be the least counterexample.
Clearly, $C_\delta$ was defined according to the first case. That is, $C_\delta=C_{\phi_1(\delta)}\cup\{\phi_1(\delta)\}\cup D_\delta$.
In particular, $\phi_1(\delta)\in S\cap\delta$ and $C_\delta \cap \phi_1 (\delta) = C_{\phi_1 (\delta)}$.

Let $\bar\delta\in\acc(C_\delta)$ be such that $C_\delta\cap\bar\delta\neq C_{\bar\delta}$.

$\br$ If $\bar\delta=\phi_1(\delta)$, then we get a contradiction to the fact that  $C_\delta\cap \phi_1(\delta)=C_{\phi_1(\delta)}$.

$\br$ If $\bar\delta<\phi_1(\delta)$, then already $\bar\delta\in\acc(C_{\phi_1(\delta)})$ with $C_{\phi_1(\delta)}\cap\bar\delta\neq C_{\bar\delta}$,
contradicting the minimality of $\delta$.
\end{proof}

For subsets $E,X$ of $\kappa$, denote:
\begin{align*}
G(E,X) &:= \{\gamma\in E\cap S\mid X\cap\gamma=X_\gamma\ \&\  E \cap S \cap\gamma=Y_\gamma\};  \\
F(E,X) &:= \{ \alpha \in G(E,X) \cup \{0\} \mid C_\alpha \subseteq G(E,X) \}.
\end{align*}

\begin{claim}\label{DinG}
For all subsets $E, X \subseteq \kappa$, if $\delta \in \acc(G(E,X))$ and $\phi_1(\delta)<\delta$,
then $D_\delta \subseteq G(E,X)$.
\end{claim}
\begin{proof}
Fix $E, X \subseteq \kappa$ and $\delta \in \acc(G(E,X))$.
By $\delta \in G(E,X)$, we have $E \cap S \cap \delta = Y_\delta$ and $X \cap \delta = X_\delta$.
Then,
\begin{align*}
\{ \gamma \in Y_\delta \mid X_\gamma = X_\delta \cap \gamma\ \&\ Y_\gamma = Y_\delta \cap \gamma \}
&= \{ \gamma \in E \cap S \cap \delta \mid
           X_\gamma = (X \cap \delta) \cap\gamma\ \&\ Y_\gamma = (E \cap S \cap \delta) \cap \gamma \} \\
&= \{ \gamma \in E \cap S \cap \delta \mid X_\gamma = X \cap \gamma\ \&\  Y_\gamma = E \cap S \cap \gamma \} \\
&= G(E, X) \cap \delta,
\end{align*}
which is cofinal in $\delta$, so that $D_\delta$ was chosen to be a subset of it.
\end{proof}

\begin{claim}
Suppose that $X\s\kappa$ is some set and $E\s\kappa$ is some club.

For every $\varepsilon<\omega_1$ and every $\alpha\in F(E,X)$
satisfying $\otp(C_\alpha)<\omega\cdot\varepsilon$, the set
\[
S_{\alpha,\varepsilon}^{E,X}:=
\{\delta\in F(E,X)\mid C_\alpha\sq C_\delta \text{ and } \otp(C_\delta)=\omega\cdot\varepsilon \}
\]
is stationary.
\end{claim}
\begin{proof}
Note that by our choice of the diamond sequence, the set $G(E, X)$ is a stationary subset of $\omega_1$,
being the intersection of the club set $E$ with a stationary set.
Thus, in particular, $\acc^+ (G(E,X))$ is club in $\kappa$.

We now prove the claim by induction over $\varepsilon$.
First, notice that when $\varepsilon = 0$ there is nothing to show,
as there is no $\alpha$ satisfying $\otp(C_\alpha) < 0$.
Thus the induction begins with $\varepsilon = 1$.

\begin{description}
\item[Base case, $\varepsilon = 1$]
We consider only $\alpha$ satisfying $\otp(C_\alpha) = 0$.

By our choice of the diamond sequence, the following set is stationary:
\[
Z := \{ \delta \in G(E,X)\mid \phi_0 (\delta) = 1, \phi_1(\delta)=\alpha \}\setminus(\alpha+1).
\]
Since $\acc^+ (G(E,X))$ is a club subset of $\kappa$,
it follows that $Z \cap \acc^+ (G(E,X))$ is stationary.
We shall show that $Z \cap \acc^+ (G(E,X)) \subseteq S_{\alpha,1}^{E,X}$.

Let $\delta \in Z\cap \acc^+(G(E,X))$ be arbitrary.
We have $\delta \in Z \subseteq G(E,X)$ and $\delta\in\acc^+(G(E,X))$,
so that $\delta \in \acc(G(E,X))$.
Thus Claim~\ref{DinG} gives $D_\delta \subseteq G(E,X)$.
By $\delta \in Z$, we have $\phi_0(\delta) = 1$, and
so by definition of $C_\delta$ in this case it follows that $C_\delta=D_\delta \subseteq G(E,X)$,
so that $\delta \in F(E,X)$.
Clearly $C_\alpha = \emptyset \sqsubseteq C_\delta$ and $\otp(C_\delta) = \omega$,
so that $\delta \in S^{E,X}_{\alpha,1}$.

\item[Successor case, $\varepsilon=\epsilon'+1$ for some nonzero $\epsilon' < \omega_1$]
We assume the claim  holds for $\epsilon'$.
Fix $\alpha\in F(E,X)$
satisfying $\otp(C_\alpha)<\omega\cdot\varepsilon$,
and we must show that the set $S_{\alpha,\varepsilon}^{E,X}$ is stationary.
We find $\alpha' \in F(E,X)$ satisfying $C_\alpha \sqsubseteq C_{\alpha'}$
and $\otp(C_{\alpha'}) = \omega \cdot \epsilon'$ by considering two cases:
\begin{itemize}
\item
If $\otp(C_\alpha) = \omega \cdot \epsilon'$, then let $\alpha' := \alpha$.
\item
If $\otp(C_\alpha) < \omega \cdot \epsilon'$,
then apply the induction hypothesis to obtain $\alpha' \in S^{E,X}_{\alpha, \epsilon'}$.
\end{itemize}

By our choice of the diamond sequence, the following set is stationary:
\[
Z := \{ \delta \in G(E,X)\mid \phi_0 (\delta) = \varepsilon, \phi_1(\delta)=\alpha' \}\setminus(\alpha'+1).
\]
Thus, it suffices to prove that $S_{\alpha,\varepsilon}^{E,X}$ covers the stationary set $Z \cap \acc^+ (G(E,X))$.

Let $\delta \in Z\cap \acc^+(G(E,X))$ be arbitrary.
Again, Claim~\ref{DinG} gives $D_\delta \subseteq G(E,X)$.
Since $\delta \in Z$, we have $\phi_0(\delta) = \varepsilon$
is a successor ordinal $ >1$, $\phi_1(\delta)<\delta$,
and $\otp(C_{\phi_1(\delta)}) = \otp(C_{\alpha'}) = \omega\cdot\epsilon' =
\omega \cdot (\varepsilon-1) = \omega \cdot (\phi_0(\delta)-1)$,
so that by definition of $C_\delta$ in this case, we have $C_\delta = C_{\alpha'} \cup \{\alpha'\} \cup D_\delta$,
where $D_\delta \subseteq \delta \setminus (\alpha'+1)$.
In particular, $C_\alpha \sqsubseteq C_{\alpha'} \sqsubseteq C_\delta$ and
$\otp(C_\delta)=\otp(C_{\alpha'})+\omega = \omega \cdot \epsilon' + \omega = \omega \cdot (\epsilon'+1) =
\omega \cdot \varepsilon$.
Since $\alpha' \in F(E,X)$,
we have $C_{\alpha'} \subseteq G(E,X)$ and $\alpha' \in G(E,X)$.\footnote{%
The case $\alpha' = 0$ is ruled out by the fact that $\otp(C_{\alpha'}) >0$.}
Thus, altogether, $C_\delta = C_{\alpha'} \cup \{ \alpha' \} \cup D_\delta \subseteq G(E,X)$,
so that $\delta \in F(E,X)$, and hence $\delta \in S^{E,X}_{\alpha,\varepsilon}$.

\item[Limit case]
Suppose that $\varepsilon < \omega_1$ is a nonzero limit ordinal,
and for every $\epsilon' < \varepsilon$ and $\alpha\in F(E,X)$
satisfying $\otp(C_\alpha)< \omega\cdot \epsilon'$,
the set $S_{\alpha,\epsilon'}^{E,X}$ is stationary.
Now fix $\alpha\in F(E,X)$
satisfying $\otp(C_\alpha)<\omega\cdot\varepsilon$,
and we must show that the set $S_{\alpha,\varepsilon}^{E,X}$ is stationary.
Let $D\s\omega_1$ be an arbitrary club.
Since the set
\[
Z:=\{\delta\in G(E,X)\mid \phi_0(\delta)=\varepsilon, \phi_1(\delta) = \alpha \}
\]
is stationary in $\omega_1$,
pick an elementary submodel $\mathcal M\prec H_{\kappa^+}$, with $\mathcal M\cap\kappa\in Z\cap D$,
such that $\alpha, \varepsilon\in\mathcal M$ and
$\langle S^{E,X}_{\alpha',\epsilon'}\mid \epsilon'<\omega_1,\alpha'<\kappa\rangle\in\mathcal M$.
Denote $\delta:=\mathcal M\cap\kappa$.
We shall show that $\delta \in S^{E,X}_{\alpha, \varepsilon}$.

Let $\langle \delta_n\mid n<\omega\rangle$ be a strictly increasing sequence of ordinals converging to $\delta$.
Let $\langle \varepsilon_n\mid n<\omega\rangle$
be a strictly increasing sequence of ordinals converging to $\varepsilon$,
with $\otp(C_\alpha) = \omega \cdot \varepsilon_0$.
Now, define a sequence $\langle \alpha_n\mid n<\omega\rangle$ by recursion as follows,
where we will ensure, for each $n < \omega$, that $\alpha_n \in F(E,X) \cap \delta$
and $\otp(C_{\alpha_n}) = \omega \cdot \varepsilon_n$.

Let $\alpha_0:=\alpha$.
Next, fix $n<\omega$, and suppose $\alpha_n$ has already been defined.
Since $\varepsilon_n < \varepsilon_{n+1} < \varepsilon$,
the induction hypothesis guarantees that $S^{E,X}_{\alpha_n,\varepsilon_{n+1}}$ is stationary,
and since $\alpha_n, \varepsilon_{n+1} \in\mathcal M$,
it follows that $S^{E,X}_{\alpha_n,\varepsilon_{n+1}} \in \mathcal M$.
Since also $\delta_n < \delta$, it follows by elementarity of $\mathcal M$ that we can
pick $\alpha_{n+1}\in S^{E,X}_{\alpha_n,\varepsilon_{n+1}}\cap\mathcal M$ with $\alpha_{n+1}>\delta_n$.

Since $\alpha_{n+1}\in S^{E,X}_{\alpha_n,\varepsilon_{n+1}}$ for all $n<\omega$,
it follows that $\langle C_{\alpha_n}\mid n<\omega\rangle$ is $\sq$-increasing,
and $\otp(C_{\alpha_n}) = \omega \cdot \varepsilon_n$ for all $n < \omega$.
Consequently, $\langle \alpha_n\mid n<\omega\rangle$ is increasing and converging to $\delta$,
$\{ \alpha_n\mid 0<n<\omega\}\s S\cap\delta$,
and
\[
\otp(\bigcup_{n<\omega} C_{\alpha_n})
= \sup_{n<\omega} (\otp(C_{\alpha_n}))
= \sup_{n<\omega} (\omega \cdot \varepsilon_n)
= \omega \cdot \sup_{n<\omega} \varepsilon_n
= \omega \cdot \varepsilon
= \omega \cdot \phi_0(\delta)
\]
Furthermore, since $\alpha_n \in F(E,X)$ for all $n < \omega$, it follows that
$\bigcup_{n<\omega}C_{\alpha_n}\s G(E,X)$.
Together with $\delta \in Z \subseteq G(E,X)$,
this implies that for every $\beta\in(\bigcup_{n<\omega}C_{\alpha_n})$
we have $X_\beta = X \cap \beta = (X \cap \delta) \cap \beta = X_\delta \cap \beta$
and $Y_\beta = E \cap S \cap \beta = (E \cap S \cap \delta) \cap \beta = Y_\delta \cap \beta$.
Finally, $\delta \in Z$ gives $\alpha_0 = \alpha = \phi_1(\delta)$.
Altogether, the conditions are satisfied for $C_\delta$ to be chosen according to second case of the definition,
so that $\otp(C_\delta) = \omega \cdot \varepsilon$, $C_\delta \subseteq G(E,X)$,
and $C_\alpha  = C_{\phi_1(\delta)} \sqsubseteq C_\delta$.
It follows that $\delta \in S^{E,X}_{\alpha, \varepsilon}$.
But $\delta$ is an element of the arbitrary club set $D$.
Thus $S^{E,X}_{\alpha, \varepsilon}$ is stationary, as required.
\qedhere
\end{description}
\end{proof}

Given a subset $X\subseteq\kappa$, club $E\subseteq\kappa$, and nonzero $\varepsilon < \omega_1$, apply the last claim with $\alpha = 0$ to obtain
$\delta \in S^{E,X}_{0, \varepsilon} \subseteq F(E,X) \setminus \{0\} \subseteq G(E,X) \subseteq S$,
so that $\otp(C_\delta) = \omega \cdot \varepsilon$,
and $C_\delta \subseteq G(E,X) \subseteq \{\gamma\in E \cap S \mid X\cap\gamma=X_\gamma\}$,
completing the proof of the lemma.
\end{proof}

In addition to its importance in the present context,
the next theorem also has applications to infinite graph theory \cite{rinot17}.

\begin{thm}\label{thm32} Assume any of the following:
\begin{itemize}
\item $\lambda=\aleph_0,  S\s\omega_1$ and $\diamondsuit(S)$ holds; or
\item $\lambda$ is an uncountable cardinal, $S=E^{\lambda^+}_{\cf(\lambda)}$ and $\sd_\lambda$ holds.
\end{itemize}

Denote $\chi := \omega \cdot \lambda$ (ordinal multiplication), and let $\sigma<\chi$ be any ordinal.
Then $\p(\lambda^+,2,{\sq},\lambda^+,\{S\},2,\sigma,\mathcal E_\chi)$ holds.
Moreover,
there exist a sequence $\langle C_\alpha\mid\alpha<\lambda^+\rangle$,
and a function $h:\lambda^+\rightarrow\lambda^+$ satisfying the following:
\begin{itemize}
\item if $\alpha<\lambda^+$ is a limit, then $C_\alpha$ is a club subset of $\alpha\setminus\{0\}$ of order-type $\le\chi$;
\item if $\bar\alpha\in\acc(C_\alpha)$,
then $C_{\bar\alpha}=C_\alpha\cap\bar\alpha$ and $h(\bar\alpha)=h(\alpha)$;
\item for every sequence $\langle A_\delta\mid \delta<\lambda^+\rangle$ of cofinal subsets of $\lambda^+$,
every club $D\s\lambda^+$, and every $\varsigma<\lambda^+$, there exists $\alpha\in S$ for which all of the following hold:
\begin{enumerate}
\item $h(\alpha)=\varsigma$;
\item $\nacc(C_\alpha)\s\bigcup_{\delta<\alpha}A_\delta$;
\item $\otp(\{\beta\in \acc(C_\alpha) \mid \suc_\sigma(C_\alpha\setminus\beta)\s A_\delta\})=\lambda$ for every $\delta<\alpha$;
\item for every $\beta<\gamma$ in $C_\alpha$, there exists $\eta\in D$, with $\beta<\eta<\gamma$.
\end{enumerate}
\end{itemize}
\end{thm}
\begin{proof}
First, notice that the given hypotheses (either $\diamondsuit(S)$ in case $\lambda = \aleph_0$
or $\sd_\lambda$ in case $\lambda$ is uncountable) imply $\diamondsuit(\lambda^+)$, which implies $\ch_\lambda$.
Thus, fix a function $\pi:\lambda^+\rightarrow{}^{\lambda^+}{\lambda^+}$ such that $\{ \alpha<\lambda^+\mid f\s \pi(\alpha)\}$
is cofinal in $\lambda^+$ for all $f\in{}^{<\lambda^+}\lambda^+$.

Using Lemma \ref{thm62} (in case $\lambda = \aleph_0$)  or Lemma \ref{lemma40} (in case $\lambda>\aleph_0$),
pick a sequence
$\langle (D_\alpha,X_\alpha)\mid \alpha<\lambda^+\rangle$
such that:
\begin{itemize}
\item  for every limit $\alpha<\lambda^+$, $D_\alpha$ is a club in $\alpha$ of order-type $\le\chi$, and $X_\alpha\s\alpha$;
\item if $\bar\alpha\in\acc(D_\alpha)$, then $D_\alpha\cap\bar\alpha=D_{\bar\alpha}$;
\item for every subset  $X\s\lambda^+$ and club $E\s\lambda^+$, there exists a limit $\alpha\in S$ with $\otp(D_\alpha)=\chi$
such that $\nacc(D_\alpha)\s\{ \gamma\in E\mid X\cap\gamma=X_\gamma\}$.\footnote{Lemmas~\ref{lemma40} and~\ref{thm62} give $D_\alpha\s\{\gamma\in E\mid X\cap\gamma=X_\gamma\}$,
but $\nacc(D_\alpha)\s\{\gamma\in E\mid\dots\}$ is all we need here.
Also, Lemma~\ref{thm62} gives us a sequence of local clubs of arbitrarily large order-type,
and we can simply replace any club of order-type $>\omega^2$ with a club of order-type $\omega$ (of the same sup).}
\end{itemize}

Define $h:\lambda^+\rightarrow\lambda^+$ by letting, for all $\alpha < \lambda^+$:
\[
h(\alpha) :=
\begin{cases}
\min(X_{\min(D_\alpha)}),  &\text{if $\alpha\in\acc(\lambda^+)$ and } X_{\min(D_\alpha)}\neq\emptyset; \\
0                                          &\text{otherwise.}
\end{cases}
\]
Notice that if $\bar\alpha\in\acc(D_\alpha)$,
then  $\min(D_{\bar\alpha})=\min(D_\alpha)$, and hence $h(\bar\alpha)=h(\alpha)$.

For every $\gamma<\lambda^+$, fix an injection $\psi_\gamma:\gamma+1\rightarrow\lambda$.
Using $\left| \chi \right| = \lambda$,
fix a function $\psi:\lambda\rightarrow \chi\times\lambda$
such that $\{ k<\lambda\mid (i,j)=\psi(k)\}$ has order-type $\lambda$ for all $(i,j)\in \chi\times\lambda$.

Set $C_0 := \emptyset$, and for every $\alpha < \lambda^+$, set $C_{\alpha+1} := \emptyset$.
Now, fix any nonzero limit $\alpha < \lambda^+$, and we will show how to construct $C_\alpha$.

Without loss of generality, assume $\sigma$ is an infinite limit ordinal.
By $\otp(D_\alpha)\leq\chi = \omega\cdot\lambda \le\sigma\cdot\lambda$,
we can let $o_\alpha: D_\alpha\rightarrow\lambda$ be the unique function satisfying
$\otp(D_\alpha\cap\beta)\in[\sigma\cdot o_\alpha(\beta),\sigma\cdot o_\alpha(\beta)+\sigma)$
for each $\beta \in D_\alpha$.
Then, define $\varphi_\alpha: D_\alpha\rightarrow\alpha$ by letting for all $\beta \in D_\alpha$:
\[
\varphi_\alpha(\beta):=\begin{cases}
\delta, &\text{if } \delta<\beta\ \&\
      \psi(o_\alpha(\beta))=(\otp(D_\alpha\cap\delta),\psi_{\min(D_\alpha\setminus \delta)}(\delta));\\
0,&\text{otherwise.}
\end{cases}
\]

Note that $\varphi_\alpha$ is well-defined, since for every $(i,j)\in \chi\times\lambda$, the set
 $\{\delta<\alpha\mid \otp(D_\alpha\cap\delta)=i\ \&\ \psi_{\min(D_\alpha\setminus\delta)} (\delta) =j\}$
contains at most a single element.

Next, define $d_\alpha : D_\alpha \to \lambda^+$ by letting for all $\beta \in D_\alpha$:
\[
d_\alpha(\beta) :=
\begin{cases}
\pi(\min(X_{\min(D_\alpha\setminus(\beta+1))}\setminus(\beta+1)))(\varphi_\alpha(\beta)),
    &\text{if } X_{\min(D_\alpha\setminus(\beta+1))}\nsubseteq\beta+1; \\
0, &\text{otherwise.}
\end{cases}
\]

Then, define $c_\alpha : D_\alpha \to \lambda^+$ by letting for all $\beta \in D_\alpha$:
\[
c_\alpha (\beta) :=
\begin{cases}
d_\alpha (\beta),                             &\text{if } \beta<d_\alpha(\beta)<\min(D_\alpha\setminus(\beta+1)); \\
\min(D_\alpha\setminus(\beta+1)), &\text{otherwise.}
\end{cases}
\]

Finally, let
\[
C_\alpha := \acc(D_\alpha) \cup \{ c_\alpha(\beta) \mid \beta\in D_\alpha\}.
\]

The very definition of $c_\alpha$ (regardless of the fact that it relies on $d_\alpha$) makes it clear that
$\beta < c_\alpha(\beta) \leq \min (D_\alpha \setminus (\beta+1))$ for all $\beta \in D_\alpha$,
so that
$\otp(C_\alpha)=\otp(D_\alpha) \leq \chi$,
$\acc^+(C_\alpha)=\acc(D_\alpha)$,
$\nacc(C_\alpha) = \{ c_\alpha(\beta) \mid \beta\in D_\alpha\}$,
and $C_\alpha$ is a club in $\alpha \setminus \{0\}$.

Having constructed $C_\alpha$ for all $\alpha < \lambda^+$, we will show that the sequence
$\left< C_\alpha \mid \alpha < \lambda^+ \right>$ and the function $h$
satisfy the requirements of the theorem.

Fix $\bar\alpha\in\acc(C_\alpha)$. Then $\bar\alpha\in\acc(D_\alpha)$, and hence
$D_{\bar\alpha}=D_\alpha\cap \bar\alpha$,
$o_{\bar\alpha} = o_\alpha \restriction \bar \alpha$,
$\varphi_{\bar\alpha}=\varphi_\alpha\restriction \bar\alpha$,
$d_{\bar\alpha} = d_\alpha \restriction \bar \alpha$,
$c_{\bar\alpha} = c_\alpha \restriction \bar \alpha$,
and $h(\bar\alpha)=h(\alpha)$,
and it follows that $C_{\bar\alpha}=C_\alpha\cap \bar\alpha$.

Thus, to see that $\langle C_\alpha\mid\alpha<\lambda^+\rangle$ meets our needs,
let us fix a sequence $\langle A_\delta\mid \delta<\lambda^+\rangle$ of cofinal subsets of $\lambda^+$,
together with a club $D\s\lambda^+$, and some ordinal $\varsigma<\lambda^+$.

Define an increasing function $f:\lambda^+\rightarrow\lambda^+$ recursively by letting $f(0):=\varsigma$, and for all nonzero $\beta<\lambda^+$:
\[
f(\beta):=\min\{\tau\in \lambda^+\setminus(\sup(f[\beta])+1)\mid  \bigwedge_{\delta<\beta}\pi(\tau)(\delta)=\min(A_\delta\setminus(\min(D\setminus(\beta+1))+1))\}.
\]
Consider the set $X:=f[\lambda^+]$, and the club
\[
E:=\{ \beta\in \acc(D)\cap \diagonal_{\delta<\lambda^+}\acc^+(A_\delta)\mid f[\beta]\s\beta\}\setminus (\varsigma+1).
\]
Pick a limit $\alpha\in S$ with $\otp(D_\alpha)=\chi$
such that $\nacc(D_\alpha)\s\{ \gamma\in E\mid X\cap\gamma=X_\gamma\}$.
In particular, $D_\alpha\s E$ and $h(\alpha)=\min(X_{\min(D_\alpha)})=\min(X)=f(0)=\varsigma$.

\begin{claim}
For every $\beta\in D_\alpha$, there exists $\eta\in D$
such that  $\beta\in\eta\in c_\alpha(\beta)\in A_{\varphi_{\alpha}(\beta)}$.
In particular:
\begin{itemize}
\item  $\nacc(C_\alpha)\s\bigcup_{\delta<\alpha}A_\delta$;
\item for every $\beta<\gamma$ in $C_\alpha$, there exists $\eta\in D$, with $\beta<\eta<\gamma$.
\end{itemize}
\end{claim}
\begin{proof}  Fix $\beta\in D_\alpha$. Denote $\beta^+:=\min(D_\alpha\setminus(\beta+1))$,
$\gamma := c_\alpha(\beta)$, $\delta:=\varphi_{\alpha}(\beta)$, and $\eta:=\min(D\setminus(\beta+1))$.
Then $\delta < \beta < \gamma \leq \beta^+ < \alpha$ and
$\gamma = c_\alpha(\beta) = \min(C_\alpha\setminus(\beta+1)) \in \nacc(C_\alpha)$.
As $\beta,\beta^+\in D_\alpha\s E$, $f$ is increasing, $X =f[\lambda^+]$, and $X_{\beta^+}=X\cap\beta^+$,
we have
$\sup(f[\beta])=\beta$,
$\sup(X_{\beta^+})=\beta^+$,
$\min(X\setminus(\beta+1))=f(\beta)$,
and
\[
d_\alpha(\beta) = \pi(f(\beta))(\delta) = \min(A_\delta\setminus(\min(D\setminus(\beta+1))+1))
= \min(A_\delta\setminus(\eta+1)).
\]
Since $\beta < \beta^+$ and $\beta^+ \in E \subseteq \acc(D)$, we have $\eta < \beta^+$.
Then, since $\beta^+\in E \setminus (\delta+1) \s\acc^+(A_\delta)$,
it follows that $\beta < \eta < d_\alpha(\beta) < \beta^+$,
and hence $c_\alpha(\beta)=d_\alpha(\beta)$.
Altogether, $\beta\in \eta\in c_\alpha(\beta)=\gamma$, where $\eta\in D$ and $\gamma\in A_\delta$.
\end{proof}

Fix $\delta<\alpha$.
Put $i:=\otp(D_\alpha\cap\delta)$ and $j:=\psi_{\min(D_\alpha\setminus\delta)}(\delta)$.
By the choice of $\psi$, the set $\{k<\lambda\mid \psi(k)=(i,j)\}$ has order-type $\lambda$,
and hence $B:=\{\beta\in D_\alpha\setminus(\delta+1)\mid \psi(o_\alpha(\beta))=(i,j)\}$
contains $\lambda$-many intervals (relativized to $D_\alpha)$ of length $\sigma$,
with each interval beginning at a point from $\acc(D_\alpha)$.
By definition of $\varphi_\alpha$, we have $\varphi_\alpha(\beta)=\delta$ for all $\beta\in B$.
Then the preceding claim shows that $\{ c_\alpha(\beta)\mid \beta\in B\}\s \nacc(C_\alpha)\cap A_\delta$.
In particular, $\otp(\{\beta\in \acc(C_\alpha) \mid \suc_\sigma(C_\alpha\setminus\beta)\s A_\delta\})=\lambda$.
\end{proof}

Of course, in the case of $\lambda = \aleph_0$ we cannot improve $\mathcal E_{\omega^2}$ to $\mathcal E_\omega$ while maintaining $\sigma = \omega$ in the preceding,
but note we can do the following.

\begin{thm}\label{thm309} Assume  $\diamondsuit(S)$ holds for a given $S\s\omega_1$.

Then $\p(\omega_1,2,{\sq},\omega_1,\{S\},2,n,\mathcal E_\omega)$ holds for every $n<\omega$.
\end{thm}
\begin{proof} Modify the construction  of Theorem~\ref{thm32}, as follows. Let $\chi:=\omega$ so that in particular, $\otp(D_\alpha)\le\omega$ for all $\alpha<\omega_1$.
Then, given a positive integer $n$, define for all $\alpha<\omega_1$, $o_\alpha:D_\alpha\rightarrow\omega$ by stipulating:
$$o_\alpha(\beta):=\left\lfloor{\frac{\otp(D_\alpha\cap\beta)}{n}}\right\rfloor.$$
The rest of the construction remains intact.
\end{proof}

\begin{lemma}\label{common}
Suppose that $\lambda\le\cf(\kappa)=\kappa$ are uncountable cardinals, and $\langle C_\alpha \mid \alpha < \kappa\rangle$ is a sequence satisfying the following:
\begin{enumerate}
\item[(i)] For every limit ordinal $\alpha<\kappa$, $C_\alpha$ is a club subset of $\alpha$;
\item[(ii)] For every limit ordinal $\Theta < \lambda$,
and every sequence $\left< B_\iota \mid \iota < \Theta \right>$ of cofinal subsets of $\kappa$,
there exists some limit ordinal $\alpha<\kappa$ such that:
\begin{enumerate}
\item $\otp(C_\alpha) = \Theta$; and
\item $C_\alpha (\iota+1) \in B_\iota$ for co-boundedly many $\iota < \Theta$.
\end{enumerate}
\end{enumerate}

Then
\begin{enumerate}
\item
For every infinite cardinal $\theta<\lambda$, every ordinal $\sigma<\lambda$,
every sequence $\langle A_i \mid i < \theta \rangle$ of cofinal subsets of $\kappa$,
and every infinite regular cardinal $\chi < \lambda$,
there exist stationarily many $\alpha \in E^{\kappa}_\chi$ satisfying,
for every $i < \theta$:
\[
\sup\{ \beta \in C_\alpha \mid \suc_\sigma (C_\alpha \setminus \beta) \subseteq A_i \} = \alpha.
\]
\item
For every cofinal subset $A \subseteq \kappa$,
every infinite regular cardinal $\chi < \lambda$, and every limit ordinal $\theta<\lambda$,
there exist stationarily many $\alpha \in E^{\kappa}_\chi$ satisfying $\otp(C_\alpha)\ge\theta$ for which there exist $\beta<\alpha$ such that
\[
\suc_{\kappa} (C_\alpha\setminus\beta) \subseteq A.
\]
\item $\kappa^{<\lambda}=\kappa$;
\end{enumerate}
\end{lemma}

\begin{proof}
\begin{enumerate}
\item
Suppose that we are given $\theta, \sigma < \lambda$, $\langle A_i\mid i<\theta\rangle$,
some infinite regular cardinal $\chi<\lambda$, and some club $D\s\kappa$. Fix a bijection $\psi:\kappa\leftrightarrow\theta\times\kappa$.
Let $\psi_0:\kappa\rightarrow\theta$ be the function such that for all $\alpha<\kappa$, if $\psi(\alpha)=(i,j)$, then $\psi_0(\alpha)=i$.

Define a function $f:\kappa\rightarrow\kappa$ by recursion:
\begin{itemize}
\item $f(0):=\min(A_{\psi_0(0)})$;
\item $f(\alpha):=\min(A_{\psi_0(\alpha)}\setminus(\min(D\setminus (\sup(f[\alpha])+1))+1))$ for all nonzero $\alpha<\kappa$.
\end{itemize}

Let $I:=\sigma\times\theta\times\chi$ be the Cartesian product, and let $\lhd$ denote the reverse-lexicographic ordering of $I$ induced from $\in$, so that $(I,\lhd)$ is isomorphic to $(\Theta,\in)$,
where $\Theta:=\sigma\cdot\theta\cdot\chi$ (ordinal multiplication). Then $\cf(\Theta)=\chi$,
and since $\sigma,\theta,\chi$ are all smaller than the cardinal $\lambda$, we have $\Theta<\lambda$. Fix a bijection $\pi:\Theta\leftrightarrow I$ such that $\alpha\in\beta\in\Theta$ implies $\pi(\alpha)\lhd\pi(\beta)$.
Define $\langle B_\iota\mid \iota<\Theta\rangle$ by letting $$B_\iota=\{ f(\alpha)\mid \psi_0(\alpha)=i\}$$ for the unique $i<\theta$ such that $\pi(\iota)$ is of the form $(\cdot,i,\cdot)$.
Evidently, $B_\iota$ is a cofinal subset of the corresponding $A_i$.
Now, fix a limit ordinal  $\alpha < \lambda^+$ satisfying satisfying properties (ii)(a) and (ii)(b) of the hypothesis.

By (i) and (ii)(a), we have $\cf(\alpha) = \cf(\otp(C_\alpha)) = \cf(\Theta)=\chi$.

By (ii)(b) and the fact that $B_\iota\s\rng(f)$ for all $\iota$,
we get that $\alpha\in\acc^+(\rng(f))$. But the definition of $f$ ensures that for all $\alpha<\beta<\kappa$,
there exists some $\eta\in D$ such that $f(\alpha)<\eta<f(\beta)$, and hence also $\alpha\in\acc^+(D)$.
As $D$ is a club, we altogether have $\alpha\in D\cap E^{\kappa}_{\chi}$.

By (ii)(b), fix $\iota'<\theta$ such that $C_\alpha (\iota+1) \in B_\iota$ whenever $\iota'<\iota < \Theta$.
Then by the definition of $\langle B_\iota\mid \iota<\Theta\rangle$,
for every $j < \sigma$, $i < \theta$, and $\eta < \chi$, either $(\sigma \cdot \theta \cdot \eta + \sigma \cdot i + j)\le\iota'$, or
\[
C_\alpha (\sigma \cdot \theta \cdot \eta + \sigma \cdot i + j + 1) = C_\alpha (\pi^{-1}(j, i, \eta) +1)
\in B_{\pi^{-1}(j, i, \eta)} \s A_i,
\]
so that for any large enough $\beta$ of the form $C_\alpha (\sigma \cdot \theta \cdot \eta + \sigma \cdot i)$, we have
$\suc_\sigma (C_\alpha \setminus \beta) \subseteq A_i$.
For any fixed $i < \theta$ the set $\{ \sigma \cdot \theta \cdot \eta + \sigma \cdot i \mid \eta < \chi \}$
is cofinal in $\Theta$,
so that $\{ C_\alpha (\sigma \cdot \theta \cdot \eta + \sigma \cdot i) \mid \eta < \chi \}$ is cofinal in $\alpha$,
and the required result follows.

\item Suppose that we are given $A,\chi,\theta$ as in the hypothesis.
Let $D\s\kappa$ be an arbitrary club. Let $A'$ be a cofinal subset of $A$ with the property that for all $\alpha<\beta$ in $A'$,
there exists some $\eta\in D$ with $\alpha'<\eta<\beta'$.
Let $\Theta := \theta+\chi$, and let $B_\iota = A'$ for all $\iota < \Theta$.
Now, fix a limit ordinal  $\alpha < \kappa$ satisfying properties (ii)(a) and (ii)(b) of the hypothesis. In particular, $\otp(C_\alpha)\ge\theta$.

By (i) and (ii)(a), we have $\cf(\alpha) = \cf(\otp(C_\alpha)) = \cf(\Theta) = \cf(\chi) = \chi$.
By (ii)(b), fix $\iota'<\theta$ such that $C_\alpha (\iota+1) \in B_\iota = A'$ whenever $\iota'<\iota < \Theta$.
Put $\beta:=C_\alpha(\iota')$.
Then $\suc_{\kappa}(C_\alpha\setminus\beta)  = \{ C_\alpha (\iota+1) \mid \iota'\le\iota < \Theta \} \subseteq A'\s A$,
and $\alpha\in\acc^+(A')\s D$. In particular, $\alpha\in D\cap E^{\kappa}_{\chi}$.

\item Let $\langle S_i\mid i<\kappa\rangle$ be some partition of $\kappa$ into stationary (or just, cofinal) sets. For every $\beta<\alpha<\kappa$,
let $X^\alpha_\beta:=\{ i<\kappa\mid C_\alpha\cap A_i\nsubseteq\beta\}$. It is easy to see that $[\kappa]^{<\lambda}\s \{X^\alpha_\beta\mid \beta<\alpha<\kappa\}$.
\qedhere
\end{enumerate}
\end{proof}
\begin{cor}\label{thm61} For any infinite cardinals $\theta<\lambda$ and any ordinal $\sigma<\lambda$, the following are equivalent:
\begin{enumerate}
\item $\square_\lambda+\ch_\lambda$;
\item $\p(\lambda^+,2,{\sq},\theta,\{E^{\lambda^+}_\chi\mid \aleph_0 \leq \cf(\chi) = \chi < \lambda\},2,\sigma,{\mathcal E_\lambda})$;
\item $\p(\lambda^+, 2, {\sq}, 1, \{E^{\lambda^+}_\chi\mid \aleph_0 \leq \cf(\chi) = \chi < \lambda \}, 2, \lambda^+,{\mathcal E_\lambda})$;
\item $\p(\lambda^+,2,{\sq},1,\{\lambda^+\},2, 0,{\mathcal E_\lambda})$.
\end{enumerate}
\end{cor}

\begin{proof}
\begin{description}
\item[(1) $\implies$ (2) \& (3)]
Since $\lambda$ is uncountable, we get from Fact \ref{fact922} that $\diamondsuit(\lambda^+)$ holds.
Next, by Fact \ref{fact_rinot11},
we can fix a $\sqc_\lambda$-sequence,
that is, a $\square_\lambda$-sequence $\overrightarrow C=\langle C_\alpha \mid \alpha < \lambda^+ \rangle$
that also satisfies the hypotheses of Lemma~\ref{common}.

Since $\overrightarrow C$ is a $\square_\lambda$-sequence,
for all $\alpha<\lambda^+$, $C_\alpha$ has order-type $\leq \lambda$,
and is a club in $\alpha$ if $\alpha$ is a limit ordinal,
and if $\bar\alpha\in\acc(C_\alpha)$, then $C_{\bar\alpha}=C_\alpha\cap\bar\alpha$,
and hence $C_{\bar\alpha}\sq C_\alpha$.

Thus, the fact that $\overrightarrow C$ witnesses
$\p^-(\lambda^+,2,{\sq},\theta,\{E^{\lambda^+}_\chi\mid \aleph_0 \leq \cf(\chi) = \chi < \lambda\},2,\sigma, \mathcal E_\lambda)$
and
$\p^-(\lambda^+, 2, {\sq}, 1, \{E^{\lambda^+}_{\chi}\mid \aleph_0 \leq \cf(\chi) = \chi < \lambda \}, 2, \lambda^+,\mathcal E_\lambda)$
follows from the two respective parts of Lemma~\ref{common}.

\item[(2) $\implies$ (4) and (3) $\implies$ (4)]
Immediate by monotonicity of the parameters.

\item[(4) $\implies$ (1)]
By $\p(\lambda^+,2,{\sq},1,\{\lambda^+\},2,0,{\mathcal E_\lambda})$,
we have $\diamondsuit(\lambda^+)$, and hence $\ch_\lambda$ holds.
$\square_\lambda$ follows using Lemma~\ref{square-equiv}.
\qedhere
\end{description}
\end{proof}
\begin{cor}\label{cor64} For every singular cardinal $\lambda$, the following are equivalent:
\begin{enumerate}
\item $\square_\lambda+\ch_\lambda$;
\item $\p(\lambda^+,2,{\sq},\lambda^+,\{E^{\lambda^+}_{\cf(\lambda)}\},2,\sigma,\mathcal E_\lambda)$ for every $\sigma<\lambda$.
\end{enumerate}
\end{cor}
\begin{proof} The forward implication follows from Fact \ref{fact_rinot19} and  Theorem \ref{thm32}.
The backward implication follows from Corollary \ref{thm61}.
\end{proof}

\begin{cor}\label{cor69}
For every uncountable cardinal $\lambda$, the following are equivalent:
\begin{enumerate}
\item $\sd_\lambda$;
\item $\p(\lambda^+,2,\sq,1,\{E^{\lambda^+}_{\cf(\lambda)}\},2,\lambda^+,\mathcal E_\lambda)$.
\end{enumerate}
\end{cor}
\begin{proof}
The forward implication is obtained by applying Theorem \ref{thm32} to the constant sequence $\langle A_0 \mid \delta < \lambda^+\rangle$,
yielding  stationarily many $\alpha\in E^{\lambda^+}_{\cf(\lambda)}$ such that
$\suc_{\lambda^+}(C_\alpha \setminus \beta) \subseteq \nacc(C_\alpha) \subseteq A_0$
whenever $\beta < \alpha$.

For the backward implication, we consider two cases.
If $\lambda$ is singular, then by Corollary \ref{thm61}, $\square_\lambda+\ch_\lambda$ holds,
and then by Fact~\ref{fact_rinot19}, so does $\sd_\lambda$.

Thus, from now on, suppose that $\lambda$ is a regular cardinal,
$\langle Z_\beta\mid \beta<\lambda^+\rangle$ is a witness to $\diamondsuit(\lambda^+)$, and
$\langle C_\alpha\mid\alpha<\lambda^+\rangle$ is a witness to $\p^-(\lambda^+,2,{\sq},1,\{E^{\lambda^+}_{\lambda}\},2,\lambda^+,\mathcal E_\lambda)$.
We shall prove that $\sd_\lambda$ holds.
For every limit $\alpha<\lambda^+$, if there exists some $\gamma<\alpha$ and a set $X_\alpha\s\alpha$ such that $X_\alpha\cap\beta=Z_\beta$ for all $\beta\in\nacc(C_\alpha\setminus\gamma)$,
then $X_\alpha$ is uniquely determined, and we may let $\gamma(\alpha)$ be the least $\gamma$ as in the preceding. Otherwise, let $X_\alpha:=\emptyset$ and $\gamma(\alpha):=0$.

Let $D_\alpha :=C_\alpha\setminus\gamma(\alpha)$.
We claim that $\langle (D_\alpha,X_\alpha)\mid \alpha<\lambda^+\rangle$ witnesses $\sd_\lambda$.

For every limit $\alpha < \lambda^+$, clearly $D_\alpha$ is club in $\alpha$,
and $\otp(D_\alpha) \leq \otp(C_\alpha) \leq \lambda$, and $X_\alpha \subseteq \alpha$.

Suppose that $\bar\alpha\in\acc(D_\alpha)$.
Then $\bar\alpha\in\acc(C_\alpha)$ and $\bar\alpha>\gamma(\alpha)$,
and hence $C_{\bar\alpha} = C_\alpha \cap \bar\alpha$,
and it is easy to see that $\gamma(\bar\alpha) = \gamma(\alpha)$ and $X_{\bar\alpha} = X_\alpha \cap \bar\alpha$,
so it follows that $D_{\bar\alpha} = D_\alpha \cap \bar\alpha$.

Consider any subset $X \subseteq \lambda^+$ and any club $E \subseteq \lambda^+$.
Since $\left< Z_\beta \mid \beta < \lambda^+ \right>$ is a $\diamondsuit(\lambda^+)$-sequence,
the set $A_0 := \{ \beta \in E \mid X \cap \beta = Z_\beta \}$ is stationary, and hence cofinal in $\lambda^+$.
Thus we can choose $\alpha \in E\cap E^{\lambda^+}_\lambda$ such that
$\sup \{ \beta \in C_\alpha \mid \suc_{\lambda^+} (C_\alpha \setminus \beta) \subseteq A_0 \} = \alpha$.
Pick $\beta_0 \in C_\alpha$ such that $\suc_{\lambda^+} (C_\alpha \setminus \beta_0) \subseteq A_0$.
Then for every $\beta \in \nacc (C_\alpha \setminus(\beta_0+1))$ we have $X \cap \beta = Z_\beta$,
so that $X_\alpha$ must have been defined to be equal to $X \cap \alpha$,
and $D_\alpha = C_\alpha \setminus \gamma$ for some $\gamma\le\beta_0+1$.
Furthermore, $\acc(D_\alpha) \subseteq \acc^+ (A_0) \subseteq \acc^+(E) \subseteq E$, since $E$ is club.
Finally, since $\lambda$ is uncountable, $\otp(\acc(D_\alpha)) = \otp(D_\alpha)$,
but $\lambda = \cf(\alpha) \leq \otp(D_\alpha) \leq \otp(C_\alpha) \leq \lambda$,
so that $\otp(\acc(D_\alpha)) = \lambda$,
as required.
\end{proof}

\begin{fact}\label{thm66}
Suppose that $V=L$, and that $\kappa$ is an inaccessible cardinal that is not weakly compact.

Then $\p(\kappa,2,{\sq},\kappa,\{E^\kappa_{\ge\chi}\mid \chi<\kappa\},2,\sigma)$ holds for all $\sigma<\kappa$.
\end{fact}

The proof will appear in an upcoming paper by Rinot and Schindler. Here, we only briefly explain how to derive
$\p(\kappa,2,{\sq},\theta,\{E^\kappa_{\ge\chi}\mid \chi<\kappa\},2,\sigma)$ for  all $\theta<\kappa$ (and all $\sigma<\kappa$).

\begin{thm}\label{thm39}
Suppose that $V=L$, and that $\kappa$ is an inaccessible cardinal that is not weakly compact.

Then $\p(\kappa,2,{\sq},1,\{E^\kappa_{\chi}\mid \aleph_0 \leq \cf(\chi) = \chi <\kappa\},2,\kappa)$ holds.
In particular, $\p(\kappa,2,{\sq},\theta,\{E^\kappa_{\ge\chi}\mid \chi<\kappa\},2,\sigma)$ holds for all $\theta,\sigma<\kappa$.
\end{thm}
\begin{proof}[Proof sketch] Work in $L$. As hinted in \cite[Theorem 3.2]{Sh:347}, the proof of \cite[$\S2$]{AShS:221} essentially shows
that for every inaccessible cardinal $\kappa$ that is not weakly compact,
there exists a sequence $\langle (D_\alpha,X_\alpha)\mid \alpha<\kappa\rangle$ such that for every limit $\alpha<\kappa$,
$D_\alpha$ is a club in $\alpha$, and if $\bar\alpha \in\acc(D_{\bar\alpha})$,
then $D_{\bar\alpha}=D_\alpha\cap \bar\alpha$ and $X_{\bar\alpha}=X_\alpha\cap \bar\alpha$. Moreover,
for every club $E\s\kappa$, subset $X\s\kappa$, and a limit ordinal $\Theta<\kappa$, there exists a singular limit ordinal $\alpha<\kappa$ with $\otp(D_\alpha)=\Theta$,
satisfying $X\cap\alpha=X_\alpha$ and $D_\alpha\s E$.
Thus, fix a sequence $\langle (D_\alpha,X_\alpha)\mid \alpha<\kappa\rangle$ as above.
For all $\alpha<\kappa$, define $f_\alpha:D_\alpha\rightarrow\alpha$ by stipulating:
$$f_\alpha(\beta):=\min((X_\beta\cup\{\beta\})\setminus\sup(D_\alpha\setminus\beta)).$$
Put $C_\alpha:=\rng(f_\alpha)$. It is not hard to verify that $\langle C_\alpha\mid\alpha<\kappa\rangle$
witnesses that  $\p^-(\kappa,2,{\sq},1,\{E^\kappa_{\chi}\mid \aleph_0 \leq \cf(\chi) = \chi <\kappa\},2,\kappa)$ holds. As $\langle X_\alpha\mid\alpha<\kappa\rangle$ witnesses that $\diamondsuit(\kappa)$ holds,
we altogether infer that $\p(\kappa,2,{\sq},1,\{E^\kappa_{\chi}\mid \aleph_0 \leq \cf(\chi) = \chi <\kappa\},2,\kappa)$ holds.

The fact that, modulo $\kappa^{<\kappa}=\kappa$, $\p^-(\kappa,2,{\sq},1,\{E^\kappa_{\chi}\mid \aleph_0 \leq \cf(\chi) = \chi <\kappa\},2,\kappa)$
entails a simultaneous witness to $\p^-(\kappa,2,{\sq},\theta,\{E^\kappa_{\geq\chi}\mid \chi<\kappa\},2,\sigma)$ for all $\theta,\sigma<\kappa$,
is proven using the coding+decoding techniques of the proof of Theorem~\ref{thm32} augmented by the ordinal arithmetic considerations of Lemma~\ref{common}.
\end{proof}

\begin{thm}\label{thm67a}\label{thm67} Suppose that $\sigma<\lambda=\lambda^{<\lambda}$ are infinite cardinals. If $\square_\lambda$ holds, then:
\begin{enumerate}
\item $V^{\add(\lambda,1)}\models \p^-(\lambda^+,2,{\sq},\lambda^+,\{S\s E^{\lambda^+}_\lambda\mid S\text{ is stationary}\},2,\sigma,\mathcal E_\lambda)$;
\item $V^{\add(\lambda,1)}\models \p(\lambda^+,2,{\sq},\lambda^+,\{S\s E^{\lambda^+}_\lambda\mid S\text{ is stationary}\},2,\sigma,\mathcal E_\lambda)$, provided that $\ch_\lambda$ holds in $V$.
\end{enumerate}
\end{thm}
\begin{proof} Work in $V$. Fix a $\square_\lambda$-sequence $\langle C_\alpha\mid\alpha<\lambda^+\rangle$.
For every $\alpha<\lambda^+$, let $\psi_\alpha:\lambda\setminus\{0\}\rightarrow\alpha$ be some surjection.\footnote{The case $\alpha=0$ is negligible.}
Given a function $g:\lambda\rightarrow\lambda$, we derive the following objects:
\begin{itemize}
\item $1_\alpha^g:=\{j < \otp(C_\alpha) \mid g(j)\neq 0\}$;
\item $g_\alpha:1^g_\alpha\rightarrow\alpha$ by stipulating $g_\alpha(j):=\psi_{C_\alpha(j)}(g(j))$;
\item $C_\alpha^g:=\{C_\alpha(j)\mid j\in\acc^+(1^g_\alpha)\}\cup\{\max\{g_\alpha(j),C_\alpha(\sup(1^g_\alpha\cap j))\}\mid j\in\nacc(1^g_\alpha)\}$;
\item $D_\alpha^g:=C_\alpha^g$ whenever $\sup(C_\alpha^g)=\alpha$, and $D_\alpha^g:=C_\alpha\setminus\sup(C^g_\alpha)$ otherwise.
\end{itemize}

\begin{claim} For every $g\in{}^\lambda\lambda$, $\langle D_\alpha^g\mid \alpha<\lambda^+\rangle$ is a $\square_\lambda$-sequence.
\end{claim}
\begin{proof} This is Claim 2.3.2 of \cite{rinot12}.
\end{proof}
\begin{enumerate}
\item Let $g:\lambda\rightarrow\lambda$ be $\add(\lambda,1)$-generic over $V$, and consider the $\square_\lambda$-sequence $\langle D^g_\alpha\mid \alpha<\lambda^+\rangle$ in $V[g]$.
As $\lambda^{<\lambda}=\lambda$, every cofinal subset of $\lambda^+$ from $V[g]$ covers a cofinal subset of $\lambda^+$ from $V$.
Thus, a simple density argument (cf.~Claims 2.3.1 and 2.3.3 of \cite{rinot12}) establishes that for every cofinal subset $A\s\lambda^+$,
there exists a club $D_A\s\lambda^+$ such that for every $\alpha\in E^{\lambda^+}_\lambda\cap D_A$, we have
$$\sup\{ \beta\in D^g_\alpha\mid \suc_\sigma(D^g_\alpha\setminus\beta)\s A\}=\alpha.$$

It follows that for every sequence $\langle A_i\mid i<\lambda^+\rangle$ of cofinal subsets of $\lambda^+$,
if we let $D:=\diagonal_{i<\lambda^+}D_{A_i}$, then for every $\alpha\in D\cap E^{\lambda^+}_\lambda$ and every $i<\alpha$:
$$\sup\{ \beta\in D^g_\alpha\mid \suc_\sigma(D^g_\alpha\setminus\beta)\s A_i\}=\alpha.$$
So $\langle D^g_\alpha\mid \alpha<\lambda^+\rangle$ witnesses the validity of $\p^-(\lambda^+,2,{\sq},\lambda^+,\{S\s E^{\lambda^+}_\lambda\mid S\text{ is stationary}\},2,\sigma,\mathcal E_\lambda)$.

\item By $\ch_\lambda + \lambda^{<\lambda}=\lambda$,
we have $V[g]\models\aleph_0<\lambda\ \&\ \ch_\lambda$.
So, by Fact \ref{fact922}, $V[g]\models\diamondsuit(\lambda^+)$.
Recalling the previous clause, we are done.
\qedhere
\end{enumerate}
\end{proof}

\section{\texorpdfstring{The coherence relation $\sq_\chi$}{Right-subscript coherence}}

Various constructions of Souslin-trees using the relation $\sq_\chi$ may be found in \cite{rinot20}.

\begin{lemma}\label{l24}
Suppose that $\lambda$ is an uncountable cardinal, and $\chi,\eta\le\lambda$ are infinite regular cardinals.

The following are equivalent:
\begin{enumerate}
\item $\boxminus_{\lambda,\ge\chi}$ holds.
\item
For every stationary $S\s\lambda^+$, there exist a stationary subset $S' \s S$
and a sequence $\langle C_\alpha\mid \alpha\in\Gamma\rangle$ satisfying:
\begin{itemize}
\item $E^{\lambda^+}_{\ge\chi}\s\Gamma\s \acc(\lambda^+)$;
\item if $\alpha\in\Gamma$, then $C_\alpha$ is a club subset of $\alpha$ of order-type $\le\lambda$;
\item if $\alpha\in\Gamma$ and $\bar\alpha\in\acc(C_\alpha)$, then $\bar\alpha\in\Gamma\setminus S'$ and $C_{\bar\alpha}=C_\alpha\cap \bar\alpha$;
\item for every club $D\s\lambda^+$, there exist stationarily many $\alpha\in \Gamma\cap E^{\lambda^+}_\eta$ such that $\min(C_\alpha)\in D$.
\end{itemize}
\item $\p^{-}(\lambda^+,2,{\sq_\chi},1,\{\lambda^+\},2,0,\mathcal E_\lambda)$.
\end{enumerate}
\end{lemma}

In particular, $\boxminus_{\lambda,\ge\aleph_0}$, $\square_\lambda$, and $\p^{-}(\lambda^+,2,{\sq},1,\{\lambda^+\},2,0,\mathcal E_\lambda)$  are all equivalent.

\begin{proof}
\begin{description}
\item[(1) $\implies$ (2)]
Let $\langle C_\alpha\mid \alpha\in E^{\lambda^+}_{\ge\chi}\rangle$ be a $\boxminus_{\lambda,\ge\chi}$-sequence.
First, we make the following adjustment.
If $\bar\alpha < \alpha$ are two elements of $E^{\lambda^+}_{\ge\chi}$ such that $\bar\alpha \in \acc(C_\alpha)$,
then replace $C_{\bar\alpha}$ with $C_\alpha \cap \bar\alpha$.
Notice that this adjustment is well-defined as a result of the second clause of Definition~\ref{def109}.
Then, let $\Gamma:=\bigcup\{ \acc(C_\alpha)\cup\{\alpha\}\mid \alpha\in E^{\lambda^+}_{\ge\chi}\}$,
and define for every $\bar\alpha \in \Gamma \cap E^{\lambda^+}_{<\chi}$,
$C_{\bar\alpha} = C_\alpha \cap \bar\alpha$ for some $\alpha \in E^{\lambda^+}_{\ge\chi}$ satisfying
$\bar\alpha \in \acc(C_\alpha)$. Again, this is well-defined.
The following is clear:
\begin{itemize}
\item $E^{\lambda^+}_{\ge\chi}\s\Gamma\s \acc(\lambda^+)$;
\item if $\alpha\in\Gamma$, then $C_\alpha$ is a club subset of $\alpha$ of order-type $\le\lambda$;
\item if $\alpha\in\Gamma$ and $\bar\alpha\in\acc(C_\alpha)$,
then $\bar\alpha\in\Gamma$ and $C_{\bar\alpha}=C_\alpha\cap \bar\alpha$.
\end{itemize}

If $S_0:=S\setminus\Gamma$ is stationary, let $\epsilon:=0$.
Otherwise, $S\cap \Gamma$ is stationary in $\lambda^+$, and since
$\{ \otp(C_\alpha) \mid \alpha \in S\cap\Gamma \}$ is a subset of $\acc(\lambda+1)$,
there must exist some nonzero limit ordinal $\epsilon \leq \lambda$ such that
$S_\epsilon:=\{ \alpha\in S\cap\Gamma\mid \otp(C_\alpha)=\epsilon\}$ is stationary,
so let $\epsilon$ denote the least such ordinal.

For all $\alpha\in\Gamma$, set:
\[
c_\alpha:=\begin{cases}C_\alpha,&\text{if } \otp(C_\alpha)\le\epsilon;\\
             C_\alpha\setminus C_\alpha(\epsilon),&\text{otherwise.}
\end{cases}
\]

Evidently:
\begin{itemize}
\item $S':=S_\epsilon$ is a stationary subset of $S$;
\item if $\alpha\in\Gamma$, then $c_\alpha$ is a club subset of $\alpha$ of order-type $\le\lambda$;
\item if $\alpha\in\Gamma$ and $\bar\alpha\in\acc(c_\alpha)$, then $\bar\alpha\in\Gamma\setminus S'$ and $c_{\bar\alpha}=c_\alpha\cap \bar\alpha$.
\end{itemize}

Now, for all $i<\lambda$ and $\alpha\in\Gamma$, define:
\[
c^i_\alpha:=\begin{cases}c_\alpha,&\text{if } \otp(c_\alpha)\le i;\\
             c_\alpha\setminus c_\alpha(i),&\text{otherwise.}
\end{cases}
\]

We claim that there exists a limit ordinal $i<\lambda$, such that for every club $D\s\lambda^+$,
there exist stationarily many $\alpha\in \Gamma\cap E^{\lambda^+}_\eta$ with $\min(c_\alpha^i)\in D$.
Of course, we then could simply fix such an $i$,
and conclude that $S'$ and $\langle c^i_\alpha\mid\alpha\in\Gamma\rangle$  are as sought.

Thus, suppose there is no such $i$. Then, we may find a sequence $\langle (D_i,E_i)\mid i<\lambda\rangle$
of pairs of club subsets of $\lambda^+$, such that for every limit $i<\lambda$ and every $\alpha\in \Gamma\cap E^{\lambda^+}_\eta \cap E_i$,
we have $\min(c^i_\alpha)\notin D_i$. Consider the club $D:=\bigcap_{i \in \acc(\lambda)}(D_i\cap E_i)$.
Pick $\alpha\in E^{\lambda^+}_{\max\{\aleph_1,\chi,\eta\}}\cap\acc(D)$.
Then $\alpha\in\Gamma$ and $c_\alpha\cap D$ is a club in $\alpha$.
Put $\beta:=\min(c_\alpha\cap D)$, and $i:=\otp(c_{\alpha}\cap\beta)$.
Pick $\bar\alpha\in(\acc(c_\alpha\cap D)\cup\{\alpha\})$ with $\cf(\bar\alpha)=\eta$. Then $\bar\alpha\in\Gamma\cap E^{\lambda^+}_\eta\cap E_i$,
and $\min(c_{\bar\alpha}^i) = c_{\bar\alpha}(i) = c_\alpha(i)=\beta\in D\s D_i$. This is a contradiction.

\item[(2) $\implies$ (3)]
Suppose $\left< C_\alpha \mid \alpha \in \Gamma \right>$ is given and satisfying the hypotheses.
We extend it to a sequence $\left< C_\alpha \mid \alpha < \lambda^+ \right>$ as follows:
\begin{itemize}
\item
Let $C_0 := \emptyset$.
\item
Let $C_{\alpha+1} := \{\alpha\}$ for every $\alpha < \lambda^+$.

\item
For every $\alpha \in \acc(\lambda^+) \setminus\Gamma$,
let $C_\alpha$ be a club subset of $\alpha$ of order-type $\cf(\alpha)$
with $\nacc(C_\alpha) \s \nacc(\alpha)$.
\end{itemize}

It is clear that $\left< C_\alpha \mid \alpha < \lambda^+ \right>$ witnesses
$\p^-(\lambda^+, 2, {\sq_\chi}, 1, \{\lambda^+\}, 2, 0, \mathcal E_\lambda)$.

\item[(3) $\implies$ (1)]
Let $\langle C_\alpha\mid\alpha<\lambda^+\rangle$ witness
$\p^-(\lambda^+, 2, {\sq_\chi}, 1, \{\lambda^+\}, 2, 0, \mathcal E_\lambda)$.
To see that its
restriction $\langle C_\alpha \mid \alpha \in E^{\lambda^+}_{\ge\chi} \rangle$ satisfies
$\boxminus_{\lambda, \geq\chi}$,
consider any $\alpha, \beta \in E^{\lambda^+}_{\ge\chi}$
and any $\gamma \in \acc(C_\alpha) \cap \acc(C_\beta)$.
We must have $C_\gamma \sq_\chi C_\alpha$ and $C_\gamma \sq_\chi C_\beta$.
But $\otp(C_\alpha) \geq \cf(\alpha) \geq \chi$,
so that by definition of $\sq_\chi$ we must have $C_\gamma \sq C_\alpha$,
and similarly $C_\gamma \sq C_\beta$.
Thus $C_\alpha \cap \gamma = C_\gamma = C_\beta \cap \gamma$, as required.
\qedhere
\end{description}
\end{proof}

\begin{thm}\label{thm38}
Suppose that $\boxminus_{\lambda,\ge\chi}+\ch_\lambda$ holds for a given limit cardinal $\lambda$
and some fixed infinite regular cardinal $\chi<\lambda$.
Then:
\begin{enumerate}
\item
$\p(\lambda^+,2,{\sq_\chi},\theta,\{E^{\lambda^+}_\eta\mid \aleph_0 \leq \cf(\eta) = \eta < \lambda\},2,\sigma,{\mathcal E_\lambda})$ holds
for every $\theta,\sigma<\lambda$.

\item $\p(\lambda^+, 2, {\sq}_\chi, 1, \{E^{\lambda^+}_\eta\mid \aleph_0 \leq \cf(\eta) = \eta < \lambda \}, 2, \lambda^+,{\mathcal E_\lambda})$ holds.

\item If $\lambda$ is singular, then  $\p(\lambda^+,2,{\sq_\chi},\lambda^+,\{E^{\lambda^+}_{\cf(\lambda)}\},2,\sigma,\mathcal E_\lambda)$ holds for every $\sigma<\lambda$.
\end{enumerate}
\end{thm}
\begin{proof}
As $\lambda$ is uncountable, Fact \ref{fact922} entails $\diamondsuit(\lambda^+)$, and so we only need to establish the corresponding  $\p^-(\dots)$ principles of Clauses (1)--(3).

The upcoming proof will invoke tools from \cite{rinot11} to  establish Clauses (1),(2).
Then, by going further and invoking tools from \cite{rinot19}, we shall establish Clause~(3).

\begin{claim}\label{claim6131}
There exist sequences $\langle C_\alpha\mid \alpha\in \Gamma\rangle$
and $\langle (S_i,\gamma_i)\mid i\le\cf(\lambda) \rangle$ such that:
\begin{itemize}
\item $E^{\lambda^+}_{\ge\chi}\s\Gamma\s \acc(\lambda^+)$, and $\Gamma=\biguplus_{i\le\cf(\lambda)}S_i$;
\item if $\alpha\in\Gamma$, then $C_\alpha$ is a club subset of $\alpha$ of order-type $\le\lambda$;
\item if $\alpha\in S_i$ and $\bar\alpha\in \acc(C_\alpha)$, then $\bar\alpha\in S_i$ and $C_{\bar\alpha} = C_\alpha\cap \bar\alpha$;
\item $\{ \alpha\in S_i\mid \otp(C_\alpha)=\gamma_i, C_\alpha\s E\}$ is stationary for every $i<\cf(\lambda)$ and every club $E\s\lambda^+$;
\item $\{ \gamma_i\mid i<\cf(\lambda)\}$ is a cofinal subset of $\lambda$.
\end{itemize}
\end{claim}
\begin{proof} By Lemma \ref{l24} and the proof of \cite[Lemma 2.3]{rinot11}.
That lemma builds on \cite[Lemma 2.1]{rinot11} in case that $\lambda$ is regular,
and \cite[Lemma 2.2]{rinot11} in case that $\lambda$ is singular.
The proof of the latter goes through as soon as one replaces there ``$\lambda_0=\cf(\lambda)$'' with ``$\lambda_0=\max\{\cf(\lambda),\chi\}$'';
the proof of the former goes through verbatim.
\end{proof}
Let $\langle C_\alpha\mid \alpha\in\Gamma\rangle$ and $\langle (S_i,\gamma_i)\mid i\le\cf(\lambda) \rangle$ be given by the preceding claim.
Note that given any club $E \subseteq \lambda^+$,
any $i < \cf(\lambda)$ and any nonzero limit ordinal $\Theta < \gamma_i$,
we can choose $\alpha \in S_i$ with $\otp(C_\alpha) = \gamma_i$ and $C_\alpha \subseteq E$,
so that letting $\bar\alpha = C_\alpha (\Theta)$ we have $\bar\alpha \in \acc(C_\alpha)$,
and it follows that $\bar\alpha \in S_i \cap E$, $\otp(C_{\bar\alpha}) = \Theta$,
and $C_{\bar\alpha} \subseteq C_\alpha \subseteq E$.
Therefore, we can fix a sequence $\left< \Theta_i \mid i < \cf(\lambda) \right>$ such that:
\begin{itemize}
\item $\{\Theta_i \mid i < \cf(\lambda)\}$ is a set of regular cardinals, cofinal in the limit cardinal $\lambda$;
\item $\Theta_i \leq \gamma_i$ for all $i < \cf(\lambda)$;
\item $\{ \alpha\in S_i\mid \otp(C_\alpha)=\cf(\alpha)=\Theta_i, C_\alpha\s E\}$ is stationary for every $i<\cf(\lambda)$ and every club $E\s\lambda^+$.
\end{itemize}

By removing elements of $\{\Theta_i \mid i < \cf(\lambda)\}$ if necessary (and merging the corresponding sets $S_i$ into $S_{\cf(\lambda)}$),
and re-indexing,
we may assume that  $\Theta_i \geq \chi$ for all $i < \cf(\lambda)$.
If $\lambda$ is singular, we may moreover assume that $\Theta_i >\cf(\lambda)$ for all $i < \cf(\lambda)$.

For every $i < \cf(\lambda)$,
denote $T_i:=\{\alpha\in S_i\mid \otp(C_\alpha)=\Theta_i\}$.
Fix a sequence of injections $\langle \psi_\gamma : \gamma+1\rightarrow\lambda\mid \gamma<\lambda^+\rangle$.
For every $\alpha\in\Gamma$,  define an injection $\varrho_\alpha:\alpha\rightarrow\lambda\times\lambda$ by
stipulating $\varrho_\alpha(\delta):=(\otp(C_\alpha\cap\delta),\psi_{\min(C_\alpha\setminus\delta)}(\delta))$.
Now, put $H^j_\alpha:=(\varrho_\alpha^{-1}[\Theta_j\times\Theta_j])^2$ for all $j<\cf(\lambda)$.
Then $\{ H_\alpha^j\mid j<\cf(\lambda)\}\s[\alpha\times\alpha]^{<\lambda}$
is an increasing chain,  converging to $\alpha\times\alpha$,
and if $\bar\alpha \in\acc(C_\alpha)$,
then $\varrho_{\bar\alpha} = \varrho_\alpha \restriction \bar\alpha$,
so that $H^j_{\bar\alpha} = H^j_\alpha\cap(\bar\alpha \times \bar\alpha)$ for all $j<\cf(\lambda)$.

By $\ch_\lambda$, let $\{X_\gamma\mid \gamma<\lambda^+\}$ be an enumeration of $[\lambda\times\lambda\times\lambda^+]^{\le\lambda}$.
For all $(j,\tau)\in\lambda\times\lambda$ and $X\s\lambda\times\lambda\times\lambda^+$,
let $\pi_{j,\tau}(X):=\{\varsigma<\lambda^+\mid (j,\tau,\varsigma)\in X\}$.
\begin{claim}\label{437} Suppose that $i<\cf(\lambda)$.

There exist $(j,\tau)\in\cf(\lambda)\times\lambda$ and $Y\s\lambda^+\times\lambda^+$
such that for every club $D\s\lambda^+$ and every subset $Z\s\lambda^+$, there exists some  $\alpha\in T_i$ such that:
\begin{enumerate}
\item $C_\alpha\s D$;
\item $H^j_\alpha\bks Y\s \{(\eta,\gamma)\mid Z\cap\eta=\pi_{j,\tau}(X_\gamma)\}$;
\item $\sup(\acc^+(\{\eta < \alpha \mid (\eta,\gamma)\in H^j_\alpha\bks Y\text{ for some }\gamma<\min(C_\alpha\setminus(\eta+1))\})\cap\acc(C_\alpha))=\alpha$.
\end{enumerate}
\end{claim}
\begin{proof} This is Claim 2.5.2 of \cite{rinot11}, and the proof is identical.
\end{proof}
Let $\langle (j_i,\tau_i,Y_i)\mid i<\cf(\lambda)\rangle$ be given by the previous claim.

For every $i<\cf(\lambda)$ and $\alpha\in S_i$, let:
\[
f_\alpha^i:=\{(\eta,\gamma)\in H^{j_i}_\alpha\bks Y_i\mid \gamma=\min\{\gamma' < \min(C_\alpha\setminus(\eta+1))\mid (\eta,\gamma')\in H^{j_i}_\alpha\bks Y_i\}\}.
\]

Then, let  $C^i_\alpha$ be the set of all $\delta$ such that all of the following properties hold:
\begin{enumerate}
\item $\delta\in C_\alpha$;
\item $\sup(\dom(f^i_\alpha)\cap\delta)\ge\sup(C_\alpha\cap\delta)$;
\item if $\eta\in\dom(f^i_\alpha)\cap\delta$, then $\pi_{j_i,\tau_i}(X_{f^i_\alpha(\eta)})\s\eta$;
\item if $\eta'<\eta<\delta$ satisfy $\eta', \eta \in \dom(f^i_\alpha)$,
then $\pi_{j_i,\tau_i}(X_{f^i_\alpha(\eta)})\bks \pi_{j_i,\tau_i}(X_{f^i_\alpha(\eta')})\s[\eta',\eta)$.
\end{enumerate}

For every $\alpha\in S_{\cf(\lambda)}$, write $C^{\cf(\lambda)}_\alpha:=\emptyset$.

Finally, for all $\alpha\in\Gamma$, put:
\[
C_\alpha^\bullet:=\begin{cases}
C_\alpha^i,&\text{if } \alpha\in S_{i},\sup(C_\alpha^i)=\alpha;\\
C_\alpha\setminus\sup(C_\alpha^i),&\text{if } \alpha\in S_i,\sup(C_\alpha^i)<\alpha.
\end{cases}
\]

Also,  for all $\alpha<\lambda^+$, let
\[
Z_\alpha:=\begin{cases}
\bigcup\{\pi_{j_i,\tau_i}(X_{f^i_\alpha(\eta)})\mid \eta\in \dom(f^i_\alpha)\},
&\text{if } \alpha\in S_i,\sup(C_\alpha^i)=\alpha;\\
\emptyset,&\text{otherwise.}
\end{cases}
\]

\begin{claim}\label{c6112}
All of the following properties hold for $\langle (C_\alpha^\bullet,Z_\alpha)\mid\alpha\in\Gamma\rangle$:
\begin{enumerate}
\item $C^\bullet_\alpha$ is a club subset of $\alpha$ (in fact a subclub of $C_\alpha$)
of order-type $\le\lambda$ for all $\alpha\in\Gamma$;
\item if $\alpha\in\Gamma $ and $\bar\alpha\in \acc(C^\bullet_\alpha)$, then $\bar\alpha\in \Gamma$,
$C^\bullet_{\bar\alpha} =C^\bullet_\alpha\cap \bar\alpha$, and $Z_{\bar\alpha} =Z_\alpha\cap \bar\alpha$;
\item for every club $D\s\lambda^+$, every subset $A\s\lambda^+$, and every $i<\cf(\lambda)$,
there exists some $\alpha\in\Gamma$ such that:
\begin{enumerate}
\item $C^\bullet_\alpha\s D$;
\item $Z_\alpha=A\cap\alpha$;
\item $\cf(\alpha)=\Theta_i$;
\item $\sup(\acc(C^\bullet_\alpha))=\alpha$.
\end{enumerate}
\end{enumerate}
\end{claim}
\begin{proof} This is the content of Claim 2.5.4 of \cite{rinot11}.
\end{proof}

Notice that $Z_\alpha \subseteq \alpha$ for all $\alpha < \lambda^+$,
using property~(3) of the definition of $C^i_\alpha$.
It then follows from the last claim that
$\left< Z_\alpha \mid \alpha < \lambda^+ \right>$ is a $\diamondsuit(\lambda^+)$-sequence.

Fix a bijection $\psi:\lambda\times\lambda^+\leftrightarrow\lambda^+$.

We define $\left< D_\alpha \mid \alpha < \lambda^+ \right>$ as follows:

\begin{itemize}
\item
Let $D_0 := \emptyset$, and for every $\alpha < \lambda^+$, let $D_{\alpha+1} := \{ \alpha \}$.

\item
For every $\alpha \in \acc(\lambda^+) \setminus\Gamma$,
let $D_\alpha$ be a club subset of $\alpha$ of order-type $\cf(\alpha)$ with $\nacc(D_\alpha)\s\nacc(\alpha)$.
\item
Let $\alpha \in \Gamma$ be arbitrary.
Put $C_\alpha':=\acc(C^\bullet_\alpha)$ in case that
$\sup(\acc(C^\bullet_\alpha))=\alpha$, and let $C_\alpha'$ be some cofinal subset of $\alpha$ of order-type $\omega$ otherwise.
Thus $C_\alpha'$ is a club subset of $\alpha$ of order-type $\leq \lambda$.
Next, for $\beta\in\nacc(C_\alpha')$, let:
\begin{itemize}
\item $X^\beta_\alpha:=\{\gamma < \lambda^+ \mid  \psi(\otp(\nacc(C_\alpha')\cap\beta),\gamma)\in  Z_\beta\}$;
\item $Y^\beta_\alpha:=X^\beta_\alpha\cap(\min\left(C^\bullet_\alpha\bks(\sup(C_\alpha'\cap\beta)+1)\right),\beta)$;
\item $\beta_\alpha:=\min(Y^\beta_\alpha\cup\{\beta\})$;
\item $D_\alpha:=\acc(C_\alpha')\cup\{\beta_\alpha\mid \beta\in\nacc(C_\alpha')\}$.
\end{itemize}
\end{itemize}

For all $\alpha \in \Gamma$ and all $\beta \in \nacc(C'_\alpha)$,
we have $\sup(C'_\alpha \cap \beta) < \beta_\alpha \leq \beta$.
Thus, for all $\alpha \in \Gamma$,
$\acc(D_\alpha)=\acc(C_\alpha')\s\acc(C_\alpha^\bullet)\s\Gamma$,
so that $\otp(D_\alpha)=\otp(C_\alpha')\le\otp(C_\alpha^\bullet)\le\lambda$,
and $D_\alpha$ is a club in $\alpha$.

Then, just as in the proof of of~\cite[Claim~3.2.1]{rinot11},
$\langle D_\alpha\mid\alpha<\lambda^+\rangle$ is a sequence of local clubs, each of order-type $\le\lambda$,
and if $\alpha\in\Gamma$ and $\bar\alpha \in\acc(D_\alpha)$,
then $\bar\alpha \in\Gamma$ and $D_{\bar\alpha}=D_\alpha\cap \bar\alpha$.
It then follows from the definition of $D_\alpha$ in case $\alpha \notin \Gamma$ that
$D_{\bar\alpha} \sqsubseteq_\chi D_\alpha$
for all $\alpha < \lambda^+$ and all $\bar\alpha \in \acc(D_\alpha)$.

\begin{claim}\label{claim383}
For every nonzero limit ordinal $\Theta<\lambda$ and every sequence $\langle A_i\mid i<\Theta\rangle$
of cofinal subsets of $\lambda^+$,
there exists some $\delta\in\Gamma$ such that:
\begin{itemize}
\item $\otp(D_\delta) = \Theta$; and
\item $D_\delta (i+1) \in A_i$ for all $i<\Theta$.
\end{itemize}
\end{claim}
\begin{proof} This is the content of Claim 3.2.2 from \cite{rinot11}.
\end{proof}

Then, the fact that $\left< D_\alpha \mid \alpha < \lambda^+ \right>$ witnesses
$\p^-(\lambda^+,2,{\sq}_\chi,\theta,\{E^{\lambda^+}_\eta\mid \aleph_0 \leq \cf(\eta) = \eta < \lambda\},2,\sigma, \mathcal E_\lambda)$
and
$\p^-(\lambda^+, 2, {\sq}_\chi, 1, \{E^{\lambda^+}_\eta\mid \aleph_0 \leq \cf(\eta) = \eta < \lambda \}, 2, \lambda^+,{\mathcal E_\lambda})$
follows from Lemma~\ref{common},
so that we have proven Clauses~(1) and~(2) of this theorem.

Next, let us work towards establishing Clause~(3).
Thus, we assume that $\lambda$ is a singular cardinal.

By removing the minimal element of $D_\alpha$, and putting $0$ instead, we may assume that $D_\alpha(0) = 0$ for all $\alpha\in\Gamma$.
Next, fix an increasing and continuous sequence $\langle \lambda_j\mid j\le\cf(\lambda)\rangle$
with $\lambda_0=\cf(\lambda)$, $\cf(\lambda_{j+1})=\lambda_{j+1}$
for all $j<\cf(\lambda)$, and $\lambda_{\cf(\lambda)}=\lambda$.
Denote $\Lambda:=\{\lambda_j\mid j<\cf(\lambda)\}$. For every limit $\epsilon\le\lambda$, put
\[
E_\epsilon:=\begin{cases}
\epsilon,&\text{if } \epsilon\le\lambda_0; \\
\epsilon\setminus\lambda_j,&\text{if $\epsilon\in(\lambda_j,\lambda_{j+1}]$ for } j < \cf(\lambda); \\
\Lambda\cap\epsilon,&\text{otherwise.}
\end{cases}
\]
Then $E_\epsilon$ is a club subset of $\epsilon$ for all limit $\epsilon \leq \lambda$.
In particular, $E_{\otp(D_\delta)}$ is a club subset of $\otp(D_\delta)$ for all limit $\delta < \lambda^+$.

As in the proof of \cite{rinot19}, we let $\pi_\delta:\otp(D_\delta)\rightarrow D_\delta$ denote the order-preserving bijection,
and then put $D_\delta':= \pi_\delta[E_{\otp(D_\delta)}]$ for every $\delta\in\Gamma$.
Thus for every $\delta \in \Gamma$, $D'_\delta \subseteq D_\delta$ is a club subset of $\delta$,
and $\otp(D'_\delta) \leq \otp(D_\delta) \leq \lambda$.

\medskip

Let $\varphi:\lambda^+\rightarrow\lambda^+$ be a surjection such that for all $\alpha<\lambda^+$,
$\varphi(\alpha)\le\alpha$ and $\varphi^{-1}\{\alpha\}$ is stationary.
Split $\Gamma$ into three sets:
\begin{itemize}
\item $\Gamma_0:=\{\delta\in\Gamma\mid \otp(D_\delta)\le\lambda_0\}$;
\item $\Gamma_1:=\{\delta\in\Gamma\mid \otp(D_\delta)\in(\lambda_j,\lambda_{j+1}]\text{ for some }j<\cf(\lambda)\}$;
\item $\Gamma_2:=\{\delta\in\Gamma\mid \otp(D_\delta)=\lambda_j\text{ for some nonzero limit }j\le\cf(\lambda)\}$.
\end{itemize}

We shall define a sequence $\langle G_\delta\mid \delta<\lambda^+\rangle$ by recursion over $\delta<\lambda^+$.
Let $G_{\delta}=\emptyset$ for all $\delta\in\lambda^+\setminus\Gamma$.
Now, suppose that $\delta\in\Gamma$, and $\langle G_\alpha\mid\alpha<\delta\rangle$ has already been defined.
The definition of $G_\delta$ splits into cases:

\begin{itemize}
\item
If $\delta\in \Gamma_2$, then let $G_\delta:=D_\delta'$.

\item
If $\delta\in \Gamma_1$, then consider the ordinal $\phi_\delta:=\varphi({\pi_\delta(1)})$:
\begin{itemize}
\item
If $\phi_\delta\in\Gamma$ and $|G_{\phi_\delta}|<\lambda$, then let $G_\delta:=G_{\phi_\delta}\cup\{\phi_\delta\}\cup D_\delta'$.

\item
Otherwise,  let $G_\delta:=D_\delta'$.
\end{itemize}

\item
If $\delta\in  \Gamma_0$, then we shall try to define an increasing and continuous sequence of ordinals $\langle \delta^i\mid i<\otp(D_\delta)\rangle$, by recursion over $i<\otp(D_\delta)$.
Let $\delta^0:=0$.
Suppose that $i<\otp(D_\delta)$ and $\delta^i$ has already been defined. If there exists an ordinal $\beta$
such that $\pi_\delta(i)<\beta<\pi_\delta(i+1)$, $G_{\delta^i}\sqsubseteq G_\beta$, $\nacc(G_\beta)\s Z_{\pi_\delta(i+1)}$, and $\otp(G_\beta)=\lambda_{i+1}$,
then put $\delta^{i+1}:=\beta$ for the least such $\beta$. If not, then we shall terminate the recursion and say that ``the $\delta$-process identified a failure at stage $i+1$''.

\begin{itemize}
\item
If the $\delta$-process identified a failure at stage $i+1$, then let $G_\delta:=D_\delta\setminus \pi_\delta(i)$.

\item
Otherwise, let $G_\delta:=\bigcup\{G_{\delta^i}\mid i<\otp(D_\delta)\}$.
\end{itemize}
\end{itemize}

This concludes the definition of $\langle G_\delta\mid\delta<\lambda^+\rangle$.

\begin{claim}\label{claim6135} The sequence $\langle (G_\delta,Z_\delta)\mid \delta\in\Gamma\rangle$ satisfies:
\begin{enumerate}
\item  for every $\delta\in\Gamma$, $G_\delta$ is a club in $\delta$ of order-type $\le\lambda$, and $Z_\delta\s\delta$;
\item if $\delta\in\Gamma$ and $\bar\delta\in\acc(G_\delta)$, then $\bar\delta\in\Gamma$ and $G_\delta\cap\bar\delta=G_{\bar\delta}$;
\item for every subset  $Z\s\lambda^+$ and club $E\s\lambda^+$,
there exists $\delta \in \Gamma$ with $\otp(G_\delta)=\lambda$
such that $\nacc(G_\delta) \s\{ \gamma\in E\mid Z\cap\gamma=Z_\gamma\}$.
\end{enumerate}
\end{claim}
\begin{proof} (1) and (2) are just like the proof of Claim 1 of \cite{rinot19}.

(3) Given $Z$ and $E$ as above, let $X:=\{ \gamma\in E\mid Z\cap\gamma=Z_\gamma\}$.
By the fact that $\left< Z_\gamma \mid \gamma < \lambda^+ \right>$ is a $\diamondsuit(\lambda^+)$-sequence,
we have $X\in[\lambda^+]^{\lambda^+}$.
Then, by the proofs of Claims 2 and 3 of \cite{rinot19}, there exists some $\delta\in\Gamma$ with $\otp(G_\delta)=\lambda$ such that $\nacc(G_\delta)\s X$.
\end{proof}

Let $\sigma<\lambda$ be an arbitrary infinite cardinal.
Using $\ch_\lambda$,
fix a function $\pi:\lambda^+\rightarrow{}^{\lambda^+}{\lambda^+}$ such that $\{ \alpha<\lambda^+\mid f\s \pi(\alpha)\}$
is cofinal in $\lambda^+$ for all $f\in{}^{<\lambda^+}\lambda^+$.
Also fix a function $\psi':\lambda\rightarrow \lambda\times\lambda$
such that $\{ k<\lambda\mid (i,j)=\psi'(k)\}$ has order-type $\lambda$ for all $(i,j)\in \lambda\times\lambda$.
We now relativize the proof of Theorem \ref{thm32} to the sequence $\langle (G_\delta,Z_\delta)\mid \delta\in\Gamma\rangle$, as follows.

Let $\alpha\in\Gamma$ be arbitrary.
Let $o_\alpha: G_\alpha\rightarrow\lambda$ be the unique function satisfying
$\otp(G_\alpha\cap\beta)\in[\sigma\cdot o_\alpha(\beta),\sigma\cdot o_\alpha(\beta)+\omega)$
for each $\beta \in G_\alpha$.
Define $\varphi_\alpha: G_\alpha\rightarrow\alpha$ by letting for all $\beta \in G_\alpha$:
\[
\varphi_\alpha(\beta):=\begin{cases}
\delta, &\text{if } \delta<\beta\ \&\
      \psi'(o_\alpha(\beta))=(\otp(G_\alpha\cap\delta),\psi_{\min(G_\alpha\setminus \delta)}(\delta));\\
0,&\text{otherwise.}
\end{cases}
\]
Define $d_\alpha : G_\alpha \to \lambda^+$ by letting for all $\beta \in G_\alpha$:
\[
d_\alpha(\beta) :=
\begin{cases}
\pi(\min(Z_{\min(G_\alpha\setminus(\beta+1))}\setminus(\beta+1)))(\varphi_\alpha(\beta)),
    &\text{if } Z_{\min(G_\alpha\setminus(\beta+1))}\nsubseteq\beta+1; \\
0, &\text{otherwise.}
\end{cases}
\]
Define $c_\alpha : G_\alpha \to \lambda^+$ by letting for all $\beta \in G_\alpha$:
\[
c_\alpha (\beta) :=
\begin{cases}
d_\alpha (\beta),                             &\text{if } \beta<d_\alpha(\beta)<\min(G_\alpha\setminus(\beta+1)); \\
\min(G_\alpha\setminus(\beta+1)), &\text{otherwise.}
\end{cases}
\]
Finally, let:
\[
G^\bullet_\alpha := \acc(G_\alpha) \cup \{ c_\alpha(\beta) \mid \beta\in G_\alpha\}.
\]

For all $\alpha\in\acc(\lambda^+)\setminus\Gamma$, let $G_\alpha^\bullet$ be a club in $\alpha$ with $\otp(G_\alpha^\bullet)=\cf(\alpha)$ and $\nacc(G_\alpha^\bullet)\s\nacc(\alpha)$.
Let $G_0^\bullet := \emptyset$, and let $G_{\alpha+1}^\bullet:=\{\alpha\}$ for all $\alpha<\lambda^+$.
Then $\langle G^\bullet_\alpha\mid\alpha<\lambda^+\rangle$ witnesses
$\p^-(\lambda^+,2,{\sq_\chi},\lambda^+,\{E^{\lambda^+}_{\cf(\lambda)}\},2,\sigma,\mathcal E_\lambda)$.
\end{proof}

The preceding theorem was focused on limit cardinals. We now establish the same result for $\lambda$ successor.

\begin{thm}\label{boxminus-succ}
Suppose that $\boxminus_{\lambda,\ge\chi}+\ch_\lambda$ holds for a given successor cardinal $\lambda$,
and for some fixed infinite regular cardinal $\chi<\lambda$.
Then:
\begin{enumerate}
\item
$\p(\lambda^+,2,{\sq_\chi},\theta,\{E^{\lambda^+}_\eta\mid \aleph_0 \leq \cf(\eta) = \eta < \lambda\},2,\sigma,{\mathcal E_\lambda})$ holds
for every cardinal $\theta<\lambda$ and every ordinal $\sigma<\lambda$;

\item $\p(\lambda^+, 2, {\sq}_\chi, 1, \{E^{\lambda^+}_\eta\mid \aleph_0 \leq \cf(\eta) = \eta < \lambda \}, 2, \lambda^+,{\mathcal E_\lambda})$ holds.
\end{enumerate}
\end{thm}

\begin{proof}
As $\lambda$ is uncountable, Fact \ref{fact922} entails $\diamondsuit(\lambda^+)$,
so that we only need to establish the corresponding  $\p^-(\dots)$ principles of Clauses~(1) and~(2).

As in Claim~\ref{claim6131},
we find sequences $\langle C_\alpha\mid \alpha\in \Gamma\rangle$
and $\langle (S_i,\gamma_i)\mid i\le \lambda \rangle$ such that:
\begin{itemize}
\item $E^{\lambda^+}_{\ge\chi}\s\Gamma\s \acc(\lambda^+)$, and $\Gamma=\biguplus_{i\le \lambda}S_i$;
\item if $\alpha\in\Gamma$, then $C_\alpha$ is a club subset of $\alpha$ of order-type $\le\lambda$;
\item if $\alpha\in S_i$ and $\bar\alpha\in \acc(C_\alpha)$,
then $\bar\alpha\in S_i$ and $C_{\bar\alpha} = C_\alpha\cap \bar\alpha$;
\item $\{ \alpha\in S_i\mid \otp(C_\alpha)=\gamma_i, C_\alpha\s E\}$ is stationary for every $i< \lambda$ and every club $E\s\lambda^+$;
\item $\{ \gamma_i\mid i< \lambda\}$ is a cofinal subset of $\lambda$.
\end{itemize}

For every $i < \lambda$,
write $T_i:=\{ \delta\in S_i \mid \otp(C_\delta)=\gamma_i\}$.
We now go along the lines of the proof of Theorem 3.3 from \cite{rinot11}.
By $\ch_\lambda$, let $\{X_\gamma\mid \gamma<\lambda^+\}$ be an enumeration of $[\lambda\times\lambda\times\lambda^+]^{\le\lambda}$.
For all $(j,\tau)\in\lambda\times\lambda$ and $X\s\lambda\times\lambda\times\lambda^+$,
write $\pi_{j,\tau}(X):=\{\varsigma<\lambda^+\mid (j,\tau,\varsigma)\in X\}$.
Fix a sequence of surjections $\langle \psi_\xi : \lambda\rightarrow \xi \mid \xi <\lambda^+\rangle$.

For all $\delta\in\Gamma$ and $j<\lambda$, denote
\[
H^j_\delta:= \left\{ (\eta,\psi_{\min(C_\delta\setminus (\eta+1))}(\iota))\mid \eta\in C_\delta, \iota<j \right\}.
\]
Notice that if $\bar\delta \in \acc(C_\delta)$,
then $H^j_{\bar\delta} = \{(\eta,\gamma)\in H^j_\delta \mid \eta<\bar\delta\}$ for all $j<\lambda$.

\begin{claim}\label{22437} Suppose that $i<\lambda$.

Then there exist $(j,\tau)\in\lambda\times\lambda$ and $Y\s\lambda^+\times\lambda^+$
such that for every club $D\s\lambda^+$ and every
subset $Z\s\lambda^+$, there exists some  $\delta\in T_i$ such that:
\begin{enumerate}
\item $\dom(H^j_\delta\bks Y)=C_\delta\s D$;
\item $H^j_\delta\bks Y\s \{(\eta,\gamma)\mid Z\cap\eta=\pi_{j,\tau}(X_\gamma)\}$.
\end{enumerate}
\end{claim}
\begin{proof} This is Claim 3.3.1 of \cite{rinot11}.
\end{proof}

Let $\langle (j_i,\tau_i,Y_i)\mid i<\lambda\rangle$ be given by the previous claim.
Let $(j_\lambda,\tau_\lambda,Y_\lambda)$ be an arbitrary element of $\lambda\times\lambda\times\mathcal P(\lambda^+\times\lambda^+)$.
Then, for all $i\le\lambda$, $\delta\in S_i$ and $\eta \in C_\delta$, put
\[
X_{\eta,\delta}:=\bigcup\{\pi_{j_i,\tau_i}(X_\gamma)\mid (\eta,\gamma)\in H^{j_i}_\delta\setminus Y_i\}\cap\eta.
\]

Next, for $\delta \in \Gamma$, define
$h_\delta : C_\delta \to \delta$ by setting, for all $\eta \in C_\delta$:
\[
h_\delta(\eta):=\begin{cases}
\min(X_{\eta,\delta}\setminus(\sup(C_\delta\cap\eta)+1)),
        &\text{if } X_{\eta,\delta} \nsubseteq \sup(C_\delta\cap\eta)+1; \\
\eta,&\text{otherwise.}
\end{cases}
\]

Then, for all $\delta\in\Gamma$, put $G_\delta := \rng(h_\delta)$.
For consecutive points $\eta_1 < \eta_2$ in $C_\delta$,
notice that $\eta_1 < h_\delta(\eta_2) \leq \eta_2$.
Also, $h_\delta(\eta) = \eta$ for any $\eta \in \acc(C_\delta)$.
Thus $\acc(G_\delta) = \acc(C_\delta)$ and $\nacc(G_\delta) = h_\delta[\nacc(C_\delta)]$.

Next, we shall use $\diamondsuit(\lambda^+)$ to guess subsets of $\lambda \times \lambda^+$ (rather than subsets of $\lambda^+$).\footnote{See Exercise II.51 of \cite{MR597342}}
More specifically, we fix a matrix  $\langle S^\iota_\gamma\mid\iota<\lambda,\gamma<\lambda^+\rangle$
 with the property that for every sequence $\langle Z_\iota\mid \iota<\lambda\rangle$ of subsets of $\lambda^+$,
the following set is stationary:
\[
\{\gamma<\lambda^+\mid \forall \iota<\lambda(Z_\iota\cap\gamma=S^\iota_{\gamma})\}.
\]

Of course, we may assume that $S^\iota_\gamma\s\gamma$ for all $\iota,\gamma$.

The next claim is analogous to Claim~\ref{claim6135}.
\begin{claim}\label{claim6142}
\begin{enumerate}
\item  for every $\delta\in\Gamma$, $G_\delta$ is a club in $\delta$ of order-type $\le\lambda$;
\item if $\delta\in\Gamma$ and $\bar\delta\in\acc(G_\delta)$, then $\bar\delta\in\Gamma$ and $G_\delta\cap\bar\delta=G_{\bar\delta}$;
\item for every sequence $\langle Z_\iota\mid\iota<\lambda\rangle$ of subsets of $\lambda^+$, every club $E\s\lambda^+$,
and every nonzero limit $\Theta<\lambda$, there exists $\alpha\in\Gamma$ with $\otp(G_\alpha)=\Theta$
such that $\nacc(G_{\alpha}) \s\{ \gamma\in E\mid \forall\iota<\lambda(Z_\iota\cap\gamma=S^\iota_\gamma)\}$.
\end{enumerate}
\end{claim}
\begin{proof}
\begin{description}
\item[(1) \& (2)]
Just like the proof of Claim 3.3.2 from \cite{rinot11}.
\item[(3)]
Given $\langle Z_\iota\mid\iota<\lambda^+\rangle$ and $E$ as above, consider the stationary set $Z:=\{\gamma\in E\mid \forall\iota<\lambda(Z_\iota\cap\gamma=S^\iota_\gamma)\}$.
Denote $D:=\acc^+(Z)$, which is club in $\lambda^+$.
Fix a large enough $i<\lambda$ so that $\gamma_i > \Theta$.
Recalling that the triple $(j_i,\tau_i,Y_i)$ was given by Claim \ref{22437},
we may now fix some $\delta\in T_i$ such that:
\begin{enumerate}
\item $\dom(H^{j_i}_\delta\bks Y_i)=C_\delta\s D$;
\item $H^{j_i}_\delta\bks Y_i\s \{(\eta,\gamma)\mid Z\cap\eta=\pi_{j_i,\tau_i}(X_\gamma)\}$.
\end{enumerate}

Consider any $\eta \in C_\delta$.
Since $\eta \in C_\delta \subseteq D = \acc^+(Z)$,
we have $\sup(Z \cap \eta) = \eta$.
Then, using Clause~(2) and the fact that $\eta \in C_\delta = \dom(H^{j_i}_\delta \setminus Y_i)$,
it follows that
\[
X_{\eta, \delta} = \bigcup \left\{ Z \cap \eta \mid (\eta, \gamma) \in H^{j_i}_\delta \setminus Y_i \right\} \cap \eta
= Z \cap \eta.
\]
In particular, if $\eta \in \nacc(C_\delta)$,
then $\sup(C_\delta \cap \eta) < \eta = \sup(Z \cap \eta) = \sup(X_{\eta,\delta})$,
so that $h_\delta(\eta) \in X_{\eta, \delta} \subseteq Z$.
Altogether, it follows that $\nacc(G_\delta)=h_\delta[\nacc(C_\delta)]\s Z$.

Let $\alpha:= G_\delta(\Theta)$.
Then $\alpha\in\acc(G_\delta)$, $\otp(G_{\alpha})=\Theta$ and $\nacc(G_{\alpha})\s\nacc(G_\delta)\s Z$.
\qedhere
\end{description}
\end{proof}

Let $\alpha\in\Gamma$. Let $\pi_\alpha:\otp(G_\alpha)\rightarrow G_\alpha$ denote the order-preserving bijection.
Define $g_\alpha:\otp(G_\alpha)\rightarrow\alpha$ by stipulating:
\[
g_\alpha(j):=\begin{cases}
\min(S^\iota_{\pi_\alpha(j)}\bks (\pi_\alpha(\iota)+1)),
                     &\text{if } j=\iota+1\ \&\ S_{\pi_\alpha(j)}^{\iota} \nsubseteq \pi_\alpha(\iota)+1; \\
\pi_\alpha(j),&\text{otherwise.}
\end{cases}
\]

Let $G_\alpha^\bullet:=\rng(g_\alpha)$.
For every $i < \otp(G_\alpha)$, $\pi_\alpha(i) < g_\alpha(i+1) \leq \pi_\alpha(i+1)$,
and for every limit $i$, $g_\alpha(i) = \pi_\alpha(i)$.
Thus, for every $\alpha \in \Gamma$, $G_\alpha^\bullet$ is club in $\alpha$,
$\acc(G_\alpha^\bullet) = \acc(G_\alpha)$, and $\otp(G_\alpha^\bullet) =\otp(G_\alpha) \leq \lambda$.
Furthermore, $G_{\bar\alpha}^\bullet = G_\alpha^\bullet \cap \bar\alpha$
for every $\alpha \in \Gamma$ and every $\bar\alpha \in \acc(G^\bullet_\alpha)$.

Let $G_0^\bullet := \emptyset$, and let $G_{\alpha+1}^\bullet:=\{\alpha\}$ for all $\alpha<\lambda^+$.
For all $\alpha\in\acc(\lambda^+)\setminus\Gamma$, let $G_\alpha^\bullet$ be a club subset of  $\alpha$ of order-type $\cf(\alpha)$ with $\nacc(G_\alpha^\bullet)\s\nacc(\alpha)$.

\begin{claim}
For every nonzero limit ordinal $\Theta < \lambda$ and
every sequence $\left< B_\iota \mid \iota < \Theta \right>$ of cofinal subsets of $\lambda^+$,
there exists some $\alpha \in \Gamma$ such that:
\begin{enumerate}
\item $\otp(G^\bullet_\alpha) = \Theta$; and
\item $G^\bullet_\alpha (\iota+1) \in B_\iota$ for all $\iota < \Theta$.
\end{enumerate}
\end{claim}
\begin{proof} Given a sequence $\langle B_\iota\mid\iota<\Theta\rangle$ as in the hypothesis,
let $E:=\bigcap_{\iota<\Theta}\acc^+(B_\iota)$, which is club in $\lambda^+$.
By Claim~\ref{claim6142} (letting $Z_\iota = B_\iota$ for all $\iota < \Theta$),
we now fix $\alpha\in\Gamma$ with $\otp(G_\alpha)=\Theta$
such that $\nacc(G_\alpha) \s\{ \gamma\in E\mid \forall\iota<\Theta(B_\iota\cap\gamma=S^\iota_\gamma)\}$.
In particular, $\otp(G^\bullet_\alpha)=\otp(G_\alpha)=\Theta$. Now, let $\iota<\Theta$ be arbitrary.
Denote $\gamma:=\pi_\alpha(\iota+1)$, and $\gamma^-:=\pi_\alpha(\iota)$.
By definition, $g_\alpha(\iota+1)$
is equal to $\min(S^\iota_{\gamma}\bks (\gamma^-+1))$, provided that the latter is nonempty.
As $\gamma\in\nacc(G_\alpha)$, we know that $\gamma\in E\s\acc^+(B_\iota)$ and $B_\iota\cap\gamma=S^\iota_{\gamma}$.
Consequently,  $G^\bullet_\alpha(\iota+1)=g_\alpha(\iota+1)\in B_\iota$.
\end{proof}

Then, the fact that $\left< G^\bullet_\alpha \mid \alpha < \lambda^+ \right>$ witnesses
$\p^-(\lambda^+,2,{\sq}_\chi,\theta,\{E^{\lambda^+}_\eta\mid \aleph_0 \leq \cf(\eta) = \eta < \lambda\},2,\sigma, \mathcal E_\lambda)$
and
$\p^-(\lambda^+, 2, {\sq}_\chi, 1, \{E^{\lambda^+}_\eta\mid \aleph_0 \leq \cf(\eta) = \eta < \lambda \}, 2, \lambda^+,{\mathcal E_\lambda})$
follows from Lemma~\ref{common}.
\end{proof}

The following two corollaries are generalizations of Corollaries \ref{thm61} and~\ref{cor64}, respectively,
which are the special cases of the following when
$\chi = \aleph_0$.

\begin{cor}\label{boxminus-equiv}
For any infinite cardinals $\theta <\lambda$, any infinite regular cardinal $\chi < \lambda$,
and any ordinal $\sigma<\lambda$, the following are equivalent:
\begin{enumerate}
\item $\boxminus_{\lambda, \geq\chi} + \ch_\lambda$;
\item $\p(\lambda^+,2,{\sq}_\chi,\theta,\{E^{\lambda^+}_\eta\mid \aleph_0 \leq \cf(\eta) = \eta < \lambda\},2,\sigma,{\mathcal E_\lambda})$;
\item $\p(\lambda^+, 2, {\sq}_\chi, 1, \{E^{\lambda^+}_\eta\mid \aleph_0 \leq \cf(\eta) = \eta < \lambda \}, 2, \lambda^+,{\mathcal E_\lambda})$;
\item $\p(\lambda^+,2,{\sq}_\chi,1,\{\lambda^+\},2, 0,{\mathcal E_\lambda})$.
\end{enumerate}
\end{cor}

\begin{proof}
\begin{description}
\item[(1) $\implies$ (2) \& (3)]
By Theorem~\ref{thm38} if $\lambda$ is a limit cardinal,
and Theorem~\ref{boxminus-succ} if $\lambda$ is a successor.

\item[(2) $\implies$ (4) and (3) $\implies$ (4)]
Immediate by monotonicity of the parameters.

\item[(4) $\implies$ (1)]
By $\p(\lambda^+,2,{\sq}_\chi,1,\{\lambda^+\},2,0,{\mathcal E_\lambda})$,
we have $\diamondsuit(\lambda^+)$, and hence $\ch_\lambda$ holds.
$\boxminus_{\lambda, \geq\chi}$ follows using Lemma~\ref{l24}.
\qedhere
\end{description}
\end{proof}

\begin{cor}
For every singular cardinal $\lambda$ and any infinite regular cardinal $\chi < \lambda$, the following are equivalent:
\begin{enumerate}
\item $\boxminus_{\lambda, \geq\chi} + \ch_\lambda$;
\item $\p(\lambda^+,2,{\sq}_\chi,\lambda^+,\{E^{\lambda^+}_{\cf(\lambda)}\},2,\sigma,\mathcal E_\lambda)$ for all $\sigma<\lambda$.
\end{enumerate}
\end{cor}
\begin{proof} The forward implication follows from Theorem \ref{thm38}(3).
The backward implication follows from Corollary~\ref{boxminus-equiv}.
\end{proof}

By \cite[Claim 27]{Sh:108}, if $\chi$ is a supercompact cardinal, then for every singular cardinal $\lambda$ of cofinality $<\chi$, very weak forms of $\square_\lambda$ (such as $\square^*_\lambda$ and even $\ap_\lambda$) fail.
In contrast, we have the following.

\begin{cor}\label{cor116} Relative to the existence of a supercompact cardinal, it is consistent that there exists a supercompact cardinal $\chi$ such that
$\p(\lambda^+,2,{\sq_\chi},\lambda^+,\{E^{\lambda^+}_\omega\},2,\sigma,\mathcal E_\lambda)$ holds for every $\sigma < \lambda$, where $\lambda:=\chi^{+\omega}$.
\end{cor}
\begin{proof} Start with a model where $\chi$ is a Laver-indestructible supercompact cardinal \cite{MR0472529},
and $\ch_\lambda$ holds for $\lambda=\chi^{+\omega}$.
 Denote by $Q(\chi, \lambda)$ the collection of all partial functions $p\colonnn\lambda^+\rightarrow[\lambda^+]^{\leq\lambda}$ satisfying:
\begin{itemize}
\item $\dom(p)$ is a bounded subset of $\lambda^+$ with some maximal element, which we denote by $m(p)$;
\item $\dom(p)\supseteq E^{m(p)}_{\ge\chi}$;
\item for all $\alpha\in\dom(p)$:
\begin{itemize}
\item $p(\alpha)$ is a club subset of $\alpha$ of order-type $\le\lambda$;
\item if $\bar\alpha\in\acc(p(\alpha))$, then $\bar\alpha\in\dom(p)$, and $p(\bar\alpha)=p(\alpha)\cap\bar\alpha$.
\end{itemize}
\end{itemize}
We consider $Q(\chi, \lambda)$ as a notion of forcing where for $p, q \in Q(\chi, \lambda)$, $q$ extends $p$ iff $q\sqsupseteq p$.

By $\ch_\lambda$, we have $|Q(\chi, \lambda)|=\lambda^+$.
By virtually the same proof of \cite[Lemma 6.1]{MR1838355}, $Q(\chi,\lambda)$ is $(\lambda+1)$-strategically closed.
Thus, altogether $Q(\chi, \lambda)$ preserves cofinalities.

\begin{claim} $V^{Q(\chi,\lambda)}\models \p(\lambda^+,2,{\sq_\chi},\lambda^+,\{E^{\lambda^+}_\omega\},2,\sigma,\mathcal E_\lambda)$ holds for every $\sigma < \lambda$.
\end{claim}
\begin{proof}
We have already noticed that $V^{Q(\chi,\lambda)}$ is a $\lambda$-distributive forcing extension of $V$,
and so $V^{Q(\chi,\lambda)}\models\ch_\lambda$.
Thus, in light of Theorem~\ref{thm38}, it suffices to prove that $\forces_{Q(\chi,\lambda)}\boxminus_{\lambda,\geq\chi}$.
That is, it suffices to prove that $D_\alpha:=\{p\in Q(\chi,\lambda)\mid m(p)\ge\alpha\}$ is dense for all $\alpha<\lambda^+$.
We do so by induction:

$\br$ $D_0 = Q(\chi, \lambda)$, which is clearly dense.

$\br$ Suppose that $\alpha<\lambda^+$ and $D_\alpha$ is dense. We shall show that $D_{\alpha+1}$ is dense. Given $p\in Q(\chi, \lambda)$,
we may assume that $p\in D_\alpha$. Now, let $p':=p\cup\{(m(p)+\omega,(m(p),m(p)+\omega))\}$. Then $p'\in D_{\alpha+1}$.

$\br$ Suppose that $\alpha<\lambda^+$ is a nonzero limit ordinal and $D_\beta$ is dense for all $\beta<\alpha$.
Let $p\in Q(\chi,\lambda)$ be arbitrary.
Fix a function $f:\cf(\alpha)\rightarrow\alpha$ whose image is cofinal in $\alpha$.
Clearly, $\dom(f)\le\lambda$.
Now, since $Q(\chi,\lambda)$ is $(\lambda+1)$-strategically closed, use the winning strategy of player \textrm{II} to play a game of length $\dom(f)+1$, producing an increasing sequence of conditions $\langle p_j\mid j<\dom(f)+1\rangle$
so that $p_1\ge p$ and $p_{2j+1}\in D_{f(j)}$ for all $j<\dom(f)$. Then $p_{\dom(f)}$ is an extension of $p$ that belongs to $D_\alpha$,
showing that $D_\alpha$ is dense.
\end{proof}

\begin{claim}
 $Q(\chi, \lambda)$ is $(<\chi)$-directed-closed.
 \end{claim}
 \begin{proof} Suppose that $D\s Q(\chi,\lambda)$ is a directed family of size $<\chi$. So, for all $p,q\in D$, we know that $p\cup q$ is a condition.
Let $d:=\bigcup D$. If $\dom(d)$ has a maximal element, then $d\in Q(\chi,\lambda)$, and we are done.
Otherwise, for all $\alpha<\delta:=\sup(\dom(d))$, we may pick $p_\alpha\in D$ such that $m(p_\alpha)>\alpha$,
which must mean that $\cf(\delta)\le|D|<\chi$. So $d\cup\{(\delta+\omega,(\delta,\delta+\omega))\}$ is a legitimate condition that serves as a bound to all elements in $D$.
\end{proof}

So $\chi$ remains supercompact in the extension.
\end{proof}

\begin{lemma}\label{prop118} Suppose that $\chi<\cf(\kappa)=\kappa$ are infinite cardinals, and $\p^-(\kappa,2,{\sq_\chi})$ holds.

Then every stationary subset of $E^\kappa_{\ge\chi}$ may be partitioned into $\kappa$-many pairwise disjoint stationary sets such that no two of them reflect simultaneously.
\end{lemma}
\begin{proof} Let $\langle C_\alpha\mid \alpha<\kappa\rangle$ be a witness to $\p^-(\kappa,2,{\sq_\chi})$.
Suppose that $\Gamma$ is some stationary subset of $E^\kappa_{\ge\chi}$.
Following the proof of \cite[Lemma 3.2]{rinot18},
write $G_i^\tau:=\{\beta\in \Gamma\mid \otp(C_\beta)>i\ \&\ C_\beta(i)\ge\tau\}$ for all $i, \tau < \kappa$.
\begin{claim} There exists $i<\kappa$ such that $G_i^\tau$ is stationary for all $\tau<\kappa$.
\end{claim}
\begin{proof} Suppose not.
Then there exists a function $f:\kappa\rightarrow\kappa$ such that $G^{f(i)}_i$ is nonstationary for all  $i<\kappa$.
For each $i<\kappa$, let $D_i$ be a club subset of $\kappa\setminus G^{f(i)}_i$.
Let $D:=\{\delta\in\diagonal_{i<\kappa}D_i\mid f[\delta]\s\delta\}$, which is club in $\kappa$,
and put $S:=\{\beta\in \Gamma\mid  \otp(C_\beta)=\beta\}$.
By the first and third triangular bullets in the proof of \cite[Claim 3.2.1]{rinot18}, the set $B:=\{ \beta\in\acc(D)\cap {S}\mid D\cap\beta\nsubseteq C_\beta\}$ must be empty,
and $\acc(D)\cap S$ must be cofinal in $\kappa$.

By $B=\emptyset$, for every $\alpha<\beta$ both in $\acc(D)\cap{S}$, we have $D \cap \beta \subseteq C_\beta$, so that $\alpha\in \acc(D) \cap \beta \subseteq \acc(C_\beta)$ and hence $C_\alpha\sq_\chi C_\beta$. As $\otp(C_\beta)=\beta\ge\chi$ for all $\beta\in S$,
we infer that $\{ C_\delta\mid\delta\in\acc(D)\cap{S}\}$ is a $\sq$-chain, converging to the club $C:=\bigcup\{ C_\delta\mid \delta\in\acc(D)\cap{S}\}$.
Write $A:=\acc(C)$, which is club in $\kappa$. As $\langle C_\alpha\mid\alpha<\kappa\rangle$ witnesses $\p^-(\kappa,2,{\sq_\chi})$,
let us pick some $\beta\in A$ such that $\sup(A\cap\nacc(C_\beta))=\beta$.

By $\beta\in\acc(C)$, we know that $\beta\in \acc(C_\delta)$ for some $\delta\in\acc(D)\cap{S}$,
and then $C\cap\beta=(C\cap\delta)\cap\beta=C_\delta\cap\beta=C_\beta$.
So $A\cap\nacc(C_\beta)=\acc(C)\cap\nacc(C\cap\beta)=\emptyset$, contradicting the choice of $\beta$.
\end{proof}

Let $i<\kappa$ be given by the preceding claim. Denote $H^\tau:=\{\beta\in \Gamma\mid \otp(C_\beta)>i\ \&\ C_\beta(i)=\tau\}$,
and $\Theta:=\{ \tau<\kappa\mid H^\tau\text{ is stationary}\}$. Note that if $\sup(\Theta)<\kappa$, then a pressing down argument would contradict the choice of $i$.
Thus,  $\{H^\tau \mid \tau \in \Theta\}$ is a partition of a subset of $\Gamma$
into $\kappa$-many pairwise disjoint stationary sets. Let $\{ S_j\mid j<\kappa\}$ be a partition of $\Gamma$ such that $|S_j \mathbin{\Delta} H^{\Theta(j)}|\le1$ for all $j<\kappa$.

Let $j_0<j_1<\kappa$ be arbitrary.
Towards a contradiction, suppose that there exists some $\delta<\kappa$ such that $S_{j_0}\cap\delta$ and $S_{j_1}\cap\delta$ are both stationary.
Let $\ell<2$.
Write $\tau_\ell:=\Theta(j_\ell)$.
Then $H^{\tau_\ell}\cap\delta$ is a stationary subset of $E^\delta_{\ge\chi}$. In particular, $\cf(\delta)>\chi$ and $\acc(C_\delta)$ is a club in $\delta$.
Pick $\beta_\ell\in H^{\tau_\ell}\cap\acc(C_\delta)$. Then $C_{\beta_\ell}\sq_\chi C_\delta$. By $\otp(C_\delta)\ge\cf(\delta)>\chi$, we have $C_{\beta_\ell}\sq C_\delta$,
and hence $C_\delta(i)=C_{\beta_\ell}(i)=\tau_\ell$.

Altogether, $\tau_0=C_\delta(i)=\tau_1$, contradicting the fact that $\tau_0<\tau_1$.
\end{proof}

\section{\texorpdfstring{The coherence relation $\sql$}{Left-subscript coherence}}

In this section, we deal with the coherence relation $\sql$. Of course, this relation is of particular interest where used in the context of $\p(\lambda^+,\mu,\sql,\ldots,\mathcal E_\lambda)$,
because $\mathcal E_\lambda$ ensures that all accumulation points of all involved clubs would then have cofinality $<\lambda$,
therefore yielding $\sql$-coherence for free. But this relation is also useful in the absence of $\mathcal E_\lambda$.
For instance, in \cite{rinot20}, a $\lambda$-complete $\lambda$-free $\lambda^+$-Souslin tree
was constructed assuming $\lambda^{<\lambda}=\lambda$ and
$\p(\lambda^+,\lambda^+,\sql,\lambda^+,\{E^{\lambda^+}_\lambda\},2)$.

\begin{thm}\label{6.22}
Suppose that $\lambda$ is any infinite cardinal, and $S\s\lambda^+$ is stationary.
\begin{enumerate}
\item The following are equivalent:
\begin{enumerate}
\item  $\clubsuit_w(S)$;
\item $\p^-(\lambda^+,2,{\sql},1,\{S\},2,\lambda,\mathcal E_\lambda)$;
\item $\p^-(\lambda^+,2,{\sql},1,\{S\},2,\lambda^+,\mathcal E_\lambda)$;
\item $\p^-(\lambda^+,2, \mathcal R, 1,\{S\},2,\lambda^+)$ for some $\mathcal R$;
\item $\p^-(\lambda^+,2, {\sqleft{\lambda^+}}, 1,\{S\},2,\lambda^+)$.
\end{enumerate}
\item The following are equivalent:
\begin{enumerate}
\item $\diamondsuit(S)$;
\item $\p(\lambda^+,2,{\sql},1,\{S\},2,\lambda,\mathcal E_\lambda)$;
\item $\p(\lambda^+,2,{\sql},1,\{S\},2,\lambda^+,\mathcal E_\lambda)$;
\item $\p(\lambda^+,2, \mathcal R, 1,\{S\},2,\lambda^+)$ for some $\mathcal R$;
\item $\p(\lambda^+,2, {\sqleft{\lambda^+}}, 1,\{S\},2,\lambda^+)$.
\end{enumerate}
\end{enumerate}
\end{thm}
\begin{proof}
Clause~(2) follows from Clause~(1) together with the equivalence (1) $\iff$ (2) of Fact~\ref{lemma613}. Thus, let us prove Clause~(1).

\begin{description}
\item[(a) $\implies$ (b)]
Let $\langle X_\alpha\mid\alpha\in S\rangle$ be as in Definition \ref{def_clubsuitw_revised}.
Let $C_\alpha:= X_\alpha \cup \acc^+(X_\alpha)$ for all limit $\alpha\in S$.
Let $C_\alpha$ be a club subset of $\alpha$ of order-type $\cf(\alpha)$
for all limit $\alpha\in \lambda^+ \setminus S$.
Let $C_{\alpha+1} := \{\alpha\}$ for all $\alpha<\lambda^+$.
For all limit $\alpha < \lambda^+$,
it is clear that $C_\alpha$ is a club subset of $\alpha$ of order-type $\cf(\alpha) \leq \lambda$.
If $\alpha < \lambda^+$ and $\bar\alpha \in \acc(C_\alpha)$,
then $\cf(\sup(C_{\bar\alpha})) = \cf(\bar\alpha) \leq \otp(C_\alpha \cap \bar\alpha) < \otp(C_\alpha) = \cf(\alpha)
\leq \lambda$,
so that automatically $C_{\bar\alpha} \sql C_\alpha$.

To see that $\langle C_\alpha\mid\alpha<\lambda^+\rangle$ witnesses
$\p^-(\lambda^+,2,{\sql},1,\{S\},2,\lambda,\mathcal E_\lambda)$,
fix an arbitrary cofinal subset $A_0$ of $\lambda^+$ and some club $D\s\lambda^+$.
Define $f:\lambda^+\rightarrow\lambda^+$ by recursion over $\alpha<\lambda^+$:
\begin{itemize}
\item $f(0):=\min(A_0)$; and
\item for nonzero $\alpha<\lambda^+$, $f(\alpha):=\min(A_0\setminus(\min(D\setminus(\sup(f[\alpha])+1))+1))$.
\end{itemize}
Write $X:=\rng(f)$.
Then $X$ is a cofinal subset of $A_0$ (thus also of $\lambda^+$)
and has the property that for all $\beta<\gamma$ in $X$,
the ordinal-interval $(\beta,\gamma)$ contains an element from $D$.
Pick a limit $\alpha\in S$ such that $\sup(X_\alpha\setminus X) < \alpha$.
In particular, $\alpha\in\acc^+(X)\s\acc^+(D)\s D$.
Let $\gamma = \sup(X_\alpha \setminus X) +1$.
Then $\nacc(C_\alpha) \setminus\gamma = \nacc(X_\alpha) \setminus\gamma \subseteq X_\alpha\setminus\gamma \subseteq X \s A_0$.
Thus, $\suc_{\lambda}(C_\alpha\setminus\beta)\s A_0$ whenever $\gamma\le\beta<\alpha$.

\item[(b) $\implies$ (c)]
Any witness to $\p^-(\ldots,\lambda,\mathcal E_\lambda)$ forms a witness $\p^-(\ldots,\lambda^+,\mathcal E_\lambda)$,
by virtue of the eighth parameter.

\item[(c) $\implies$ (d)]
Immediate, by taking $\mathcal R:=\sql$.
\item[(d) $\implies$ (e)]
The relation $\sqleft{\lambda^+}$ provides no coherence information whatsoever.
\item[(e) $\implies$ (a)]
Let $\langle C_\alpha\mid\alpha<\lambda^+\rangle$ be a witness to
$\p^-(\lambda^+,2, {\sqleft{\lambda^+}}, 1,\{S\},2,\lambda^+)$.
For all limit $\alpha \in S$,
let $X_\alpha$ be a cofinal subset of $\nacc(C_\alpha)$ of order-type $\cf(\alpha)$.
For successor $\alpha \in S$, choose $X_\alpha$ arbitrarily.
Then $\langle X_\alpha\mid\alpha\in S\rangle$ witnesses $\clubsuit_w(S)$.
\qedhere
\end{description}
\end{proof}

Notice that likewise, if $S$ is a stationary subset of an inaccessible cardinal $\kappa$,
then $\clubsuit_w(S)\iff\p^-(\kappa,2,\sqleft{\kappa},1,\{S\},2,\kappa)$,
and $\diamondsuit(S)\iff\p(\kappa,2,\sqleft{\kappa},1,\{S\},2,\kappa)$.

\begin{lemma}\label{reflected-diamond-equivalence}
Suppose that $\lambda$ is an uncountable cardinal, and $S$ is a stationary subset of $E^{\lambda^+}_{\cf(\lambda)}$.
Then the following are equivalent:
\begin{enumerate}
\item $\langle\lambda\rangle^-_S$ holds;
\item There exist  sequences $\langle C_\delta\mid \delta\in S\rangle$ and $\langle A_\beta\mid \beta<\lambda^+\rangle$
such that:
\begin{enumerate}
\item for all $\delta\in S$, $C_\delta$ is  a club subset of $\delta$ of order-type $\lambda$;
\item for every club $D\s\lambda^+$, every subset $A\s\lambda^+$, and every infinite regular $\sigma<\lambda$ with $\sigma\neq\cf(\lambda)$,
there exist stationarily many $\delta\in S$ for which:
$$\otp(\{\beta\in \nacc(C_\delta)\cap D\cap E^{\lambda^+}_\sigma\mid A\cap\beta=A_\beta\})=\lambda.$$
\end{enumerate}
\end{enumerate}

In particular, $\langle\lambda\rangle^-_S$ entails $\diamondsuit(\lambda^+)$.
\end{lemma}
\begin{proof}
\begin{description}
\item[(2) $\implies$ (1)]
Let $\langle C_\delta\mid \delta\in S\rangle$ and $\langle A_\beta\mid \beta<\lambda^+\rangle$ be as in (2).
Let $\delta\in S$ be arbitrary.
Put  $A^\delta_i := A_{C_\delta(i)}$ for all $i<\lambda$.
Evidently, $\langle C_\delta\mid\delta\in S\rangle$ and $\langle A^\delta_i \mid \delta \in S, i < \lambda \rangle$
together witness $\langle\lambda\rangle^-_S$.

\item[(1) $\implies$ (2)]
Fix $\langle C_\delta \mid \delta \in S\rangle$ and $\langle A^\delta_i \mid \delta \in S, i < \lambda\rangle$
witnessing $\left<\lambda\right>^-_S$.
We commence by proving that $\mathcal P(\lambda)\s\{ A^\delta_{i+1}\mid \delta\in S, i<\lambda\}$,
thus establishing that $\ch_\lambda$ holds.

Let $A \subseteq \lambda$ be arbitrary.
In particular, $A \subseteq \lambda^+$,
so we can fix $\delta \in S$ above $\lambda$
such that $B:= \{ i<\lambda \mid A \cap (C_\delta(i+1)) = A^\delta_{i+1} \}$ is cofinal in $\lambda$.
Then $\{ C_\delta(i+1) \mid i \in B \}$ is cofinal in $C_\delta$ and therefore in $\delta$,
so we can choose $i \in B$ such that $\lambda < C_\delta(i+1) < \delta$.
Since $A \subseteq \lambda$ and $i \in B$, it follows that $A = A \cap (C_\delta(i+1)) = A^\delta_{i+1}$.

By $\ch_\lambda$, $\lambda>\aleph_0$ and Fact~\ref{fact922}, let us fix a sequence $\langle A_\beta\mid \beta<\lambda^+\rangle$
such that for every $A\s\lambda^+$ and every infinite regular $\sigma<\lambda$ with $\sigma\neq\cf(\lambda)$,
the set $\{ \beta\in E^{\lambda^+}_\sigma\mid A\cap\beta=A_\beta\}$ is stationary.

Let $\delta \in S$ be arbitrary.
Define $d_\delta : C_\delta \to \delta$ by setting, for every $\beta \in C_\delta$:
\[
d_\delta(\beta) := \min(\{\beta\}\cup(A^\delta_{\otp(C_\delta\cap\beta)}\setminus (\sup(C_\delta\cap\beta)+1))).
\]
Then define
\[
C^\bullet_\delta := \{ d_\delta(\beta) \mid \beta \in C_\delta \}.
\]
Clearly $C^\bullet_\delta$ is a club subset of $\delta$ of order-type $\lambda$,
and $\acc(C^\bullet_\delta) = \acc(C_\delta)$.
Furthermore:

\begin{claim}
For every cofinal subset $Z\s\lambda^+$, there exist stationarily many $\delta\in S$ such that
$$\otp(\nacc(C^\bullet_\delta)\cap Z)=\lambda.$$
\end{claim}

\begin{proof}
Let $Z \subseteq \lambda^+$ be an arbitrary cofinal set.
Let $D := \acc^+(Z)$, which is club in $\lambda^+$, and let $A := Z$.
From the fact that $\langle C_\delta \mid \delta \in S\rangle$ and
$\langle A^\delta_i \mid \delta \in S, i < \lambda\rangle$ witness $\left<\lambda\right>^-_S$,
we obtain stationarily many $\delta \in S$ such that $\otp(H) = \lambda$, where
\[
H := \{ \beta \in \nacc(C_\delta) \cap D \mid A \cap \beta = A^\delta_{\otp(C_\delta \cap \beta)} \}.
\]
Let $\beta \in H$ be arbitrary.
Then $Z \cap \beta = A^\delta_{\otp(C_\delta \cap \beta)}$.
Since $\beta \in \nacc(C_\delta)$ and $\beta \in D = \acc^+(Z)$,
the set
$A^\delta_{\otp(C_\delta\cap\beta)} \setminus (\sup(C_\delta\cap\beta) +1)
= Z \cap (\sup(C_\delta \cap \beta), \beta)$
is nonempty, and thus it contains $d_\delta(\beta)$.
It follows that $d_\delta[H] \subseteq \nacc(C^\bullet_\delta) \cap Z$,
showing that $\otp(\nacc(C^\bullet_\delta) \cap Z) = \lambda$, as required.
\end{proof}

Given a club $D\s\lambda^+$, a subset $A\s\lambda^+$, and an infinite regular $\sigma<\lambda$ with $\sigma\neq\cf(\lambda)$,
the set $Z:=\{\beta\in D\cap E^{\lambda^+}_\sigma\mid A\cap\beta=A_\beta\}$ is cofinal in $\lambda^+$ (indeed, even stationary), and hence by the previous claim
there exist stationarily many $\delta\in S$ such that $\otp(\nacc(C^\bullet_\delta)\cap Z)=\lambda$, as sought.
\qedhere
\end{description}
\end{proof}

\begin{thm}\label{thm312} Suppose that $\ch_\lambda$ holds for  a given regular uncountable cardinal $\lambda$,
and $S\s E^{\lambda^+}_\lambda$ is stationary. Then for any uncountable cardinal $\chi\le\lambda$,
$\curlywedge^-(\chi,S)$ entails $\langle\lambda\rangle^-_S$.
\end{thm}

\begin{remark}
Recall that Fact~\ref{factRinot09} already dealt with the case where  $\chi < \lambda$ or $\lambda$ is a successor cardinal.
Here, we give a different proof that covers also the hardest case where $\chi = \lambda$ is inaccessible.
\end{remark}

\begin{proof}
Let $\langle \mathcal C_\delta^i\mid \delta\in S, i<\lambda\rangle$ be a witness to $\curlywedge^-(\chi,S)$.
Let $\delta\in S$ and $i<\lambda$ be arbitrary.
Fix $f_\delta^i:\delta\rightarrow[\delta]^{<\chi}$ with the property that
$\{ \alpha, \alpha+1 \} \subseteq f_\delta^i(\alpha)\in\mathcal C_\delta^i$ for all $\alpha<\delta$.
Let $\langle e^j_\delta \mid j < \lambda \rangle$ be the increasing enumeration of some club subset of $\delta$.
Define a function $c^i_\delta:\lambda\rightarrow[\delta]^{<\chi}$ by recursion:
\begin{itemize}
\item $c^i_\delta(0):= \emptyset$;
\item $c^i_\delta(j+1):=f^i_\delta(\sup(c^i_\delta(j)))\setminus\sup(c^i_\delta(j))$;
\item $c^i_\delta(j):= \{ \max\{e_\delta^j,\sup(\bigcup c^i_\delta[j])\} \}$ for nonzero limit $j<\lambda$.
\end{itemize}
Finally, let $C^i_\delta$ denote the closure in $\delta$ of $\bigcup \rng(c^i_\delta)$. Then $C^i_\delta$ is a club in $\delta$ of order-type $\lambda$.

\begin{claim}\label{c6241} For every function $f:\lambda^+\rightarrow\lambda^+$
and every club $D\s\lambda^+$, there exist some $\delta\in S$ and $i<\lambda$ with
$$\sup\{\alpha\in C^i_\delta\cap D\cap E^{\lambda^+}_\lambda\mid f(\alpha)\in f^i_\delta(\alpha)\}=\delta.$$
\end{claim}
\begin{proof} Given $f$ and $D$ as in the hypothesis, define $g:[\lambda^+]^{<\omega}\rightarrow\lambda^+$ by stipulating
\[
g(\sigma) :=
\begin{cases}
0,                                              &\text{if } \sigma = \emptyset;  \\
\min (D\cap E^{\lambda^+}_\lambda \setminus (\beta+1)), &\text{if } \sigma=\{\beta\}; \\
f(\min(\sigma)), &\text{otherwise}.
\end{cases}
\]

Since $\chi$ is uncountable, $\mathcal D:=\{ x\in[\lambda^+]^{<\chi}\mid g``[x]^{<\omega}\s x\}$ is a club subset of $[\lambda^+]^{<\chi}$,
so (using the fact that $\langle \mathcal C_\delta^i\mid \delta\in S, i<\lambda\rangle$ witnesses $\curlywedge^-(\chi,S)$)
let us pick $\delta\in S$ and $i<\lambda$ with $\mathcal C^i_\delta\s\mathcal D$.
Let $\beta<\delta$ be arbitrary. We shall find $\alpha\in C^i_\delta\cap D\cap E^{\lambda^+}_\lambda$ above $\beta$ such that $f(\alpha)\in f^i_\delta(\alpha)$.
Fix some nonzero limit $j<\lambda$ such that $e_\delta^j>\beta$. Then $\beta':=\sup(c^i_\delta(j)) \geq e^j_\delta>\beta$,
and $f^i_\delta(\beta')\setminus \beta' = c^i_\delta(j+1)\s C^i_\delta$.

By $\beta'\in f^i_\delta(\beta')\in \mathcal C^i_\delta\s\mathcal D$,
we have that $\alpha:=g(\{\beta'\})=\min(D\cap E^{\lambda^+}_\lambda\setminus(\beta'+1))$ is in $f^i_\delta(\beta')\setminus\beta'$,
and therefore in $C^i_\delta$.
Thus, we have found an $\alpha\in D\cap E^{\lambda^+}_\lambda\cap C^i_\delta$ above $\beta$.
Finally, since $\{ \alpha, \alpha+1 \} \subseteq f_\delta^i(\alpha)\in\mathcal C_\delta^i\s\mathcal D$,
we get $f(\alpha) =  g(\{\alpha, \alpha+1\}) \in g``[f^i_\delta(\alpha)]^{<\omega} \subseteq f^i_\delta(\alpha)$, as sought.
\end{proof}

Invoking $\ch_\lambda$,
let $\{h_\beta\mid\beta<\lambda^+\}$ be some enumeration of ${}^{\lambda}\lambda^+$. For all $\delta\in S$ and $i<\lambda$,
define $g^i_\delta:\delta\rightarrow[\delta]^{<\chi}$ by stipulating:
$$g^i_\delta(\alpha):=\{ h_\beta(i)\mid \beta\in f^i_\delta(\alpha)\}\cap\delta.$$

\begin{claim} There exists $i<\lambda$ such that for every function $f:\lambda^+\rightarrow\lambda^+$
and every club $D\s\lambda^+$, there exist some $\delta\in S$ with
$$\sup\{\alpha\in C^i_\delta\cap D\cap E^{\lambda^+}_\lambda\mid f(\alpha)\in g^i_\delta(\alpha)\}=\delta.$$
\end{claim}
\begin{proof} Suppose not, and pick, for every $i<\lambda$, a counterexample $(f_i,D_i)$.
Define $f:\lambda^+\rightarrow\lambda^+$ by letting for all $\alpha<\lambda^+$: $$f(\alpha):=\min\{\beta<\lambda^+\mid h_\beta=\langle f_i(\alpha)\mid i<\lambda\rangle\}.$$
Let $D:=\bigcap_{i<\lambda}\{\delta\in D_i\mid f_i[\delta]\s\delta\}$, which is club in $\lambda^+$.
Using Claim~\ref{c6241}, pick $\delta\in S$ and $i<\lambda$
such that $$\Delta:=\{\alpha\in C^i_\delta\cap D\cap E^{\lambda^+}_\lambda\mid f(\alpha)\in f^i_\delta(\alpha)\}$$
is cofinal in $\delta$.
In particular, $\delta = \sup(\Delta) \in \acc^+(D) \subseteq D$, so that $f_i[\Delta] \subseteq f_i[\delta] \subseteq \delta$.
Consider an arbitrary $\alpha\in \Delta$, and let $\beta:=f(\alpha)$.
Then $\alpha\in D \subseteq D_i$ and $\beta\in f^i_\delta(\alpha)$, so that
$f_i(\alpha) = h_{\beta}(i) \in g^i_\delta(\alpha)$.
Altogether,
$$\{\alpha\in C^i_\delta\cap D_i\cap E^{\lambda^+}_\lambda\mid f_i(\alpha)\in g^i_\delta(\alpha)\}$$
contains the cofinal subset $\Delta$ of $\delta$, contradicting the choice of the pair $(f_i,D_i)$.
\end{proof}
Let $i<\lambda$ be given by the previous claim. For notational simplicity, denote $C^i_\delta$ by $C_\delta$.
For every $\alpha < \lambda^+$,
again invoking $\ch_\lambda$,
let $\{ X_\alpha^\beta\mid \beta<\lambda^+\}$ be some enumeration (possibly with repetition) of all subsets of $\alpha$.

Consider an arbitrary $\delta \in S$.
Let $g_\delta:\delta\rightarrow {}^{<\chi}\delta$ be such that for every $\alpha < \delta$,
$g_\delta(\alpha)$ is a surjection from some cardinal $<\chi$ to the set $\{\beta\in g^{i}_\delta(\alpha)\mid \sup(X^\beta_\alpha)=\alpha\}$.
As $\otp(C_\delta) = \lambda$, we have $C_\delta \cap E^{\lambda^+}_\lambda \subseteq \nacc(C_\delta)$,
so,  for all $\alpha\in C_\delta\cap E^{\lambda^+}_\lambda$, let us define a strictly increasing and continuous function $\varphi_{\delta,\alpha}: (\dom(g_\delta(\alpha)) +1) \rightarrow\alpha$
such that:
\begin{itemize}
\item  $\varphi_{\delta,\alpha}(0):=\sup(C_\delta\cap\alpha)$, and
\item $\varphi_{\delta,\alpha}(\xi+1):=\min(X_\alpha^{g_\delta(\alpha)(\xi)}\setminus (\varphi_{\delta,\alpha}(\xi) +1))$ for all $\xi \in \dom (g_\delta(\alpha))$.
\end{itemize}

Finally, put $D_\delta:=C_\delta\cup \bigcup \{\rng(\varphi_{\delta,\alpha})\mid \alpha\in C_\delta\cap E^{\lambda^+}_\lambda\}$. Clearly, $D_\delta$ is a club in $\delta$
of order-type $\lambda$.

By $\ch_\lambda$, $\lambda>\aleph_0$ and Fact~\ref{fact922}, let us fix a sequence $\langle A_\gamma\mid \gamma<\lambda^+\rangle$
such that for every $A\s\lambda^+$ and every infinite regular $\sigma<\lambda$,
the set $\{ \gamma\in E^{\lambda^+}_\sigma\mid A\cap\gamma=A_\gamma\}$ is stationary.

\begin{claim} For every club $D\s\lambda^+$, every subset $A\s\lambda^+$, and every infinite regular $\sigma<\lambda$,
there exist stationarily many $\delta\in S$ for which:
$$\sup\{\gamma\in \nacc(D_\delta)\cap D\cap E^{\lambda^+}_\sigma\mid A\cap\gamma=A_\gamma\}=\delta.$$
\end{claim}
\begin{proof}
Given arbitrary $A \subseteq \lambda^+$, clubs $D, E \subseteq \lambda^+$ and infinite regular $\sigma<\lambda$,
we shall find $\delta\in E\cap S$ such that $\sup\{ \gamma\in\nacc(D_\delta)\cap D\cap E^{\lambda^+}_\sigma\mid A\cap\gamma=A_\gamma\} = \delta$.
Let $G:=\{\gamma\in D\cap E^{\lambda^+}_\sigma\mid A\cap\gamma=A_\gamma\}$.
By our choice of $\langle A_\gamma\mid \gamma<\lambda^+\rangle$, $G$ is stationary.
Choose $f:\lambda^+\rightarrow\lambda^+$
so that for all $\alpha < \lambda^+$, $f(\alpha)$ is some $\beta<\lambda^+$ such that $G\cap\alpha=X_\alpha^\beta$.
Put $D':=\acc^+(G)\cap E$, which is club. By the choice of $i$, we may find some $\delta\in S$ such that
$$\Delta:=\{\alpha\in C_\delta\cap D'\cap E^{\lambda^+}_\lambda\mid f(\alpha)\in g^i_\delta(\alpha)\}$$
is cofinal in $\delta$.
In particular, $\delta = \sup(\Delta) \in \acc^+(D') \subseteq D' \subseteq E$.
Consider arbitrary $\alpha\in\Delta$.
By $\alpha\in D' \subseteq \acc^+(G)$,
$\alpha = \sup(G \cap \alpha) = \sup (X^{f(\alpha)}_\alpha)$,
and since $\alpha \in \Delta$, we have $f(\alpha) \in g^i_\delta(\alpha)$.
Altogether, $f(\alpha) \in \rng(g_\delta(\alpha))$.
Thus,
there exists some $\xi\in\dom(g_\delta(\alpha))$ such that $g_\delta(\alpha)(\xi)=f(\alpha)$.
Then $\varphi_{\delta,\alpha}(\xi+1)\in\nacc(D_\delta)\cap G\setminus\sup(C_\delta\cap\alpha)$.
Consequently, $\sup(\nacc(D_\delta)\cap G)=\delta$.
\end{proof}

Since $\cf(\delta) = \lambda$,
it now follows from Lemma~\ref{reflected-diamond-equivalence} that  $\langle\lambda\rangle^-_S$ holds.
\end{proof}

\begin{thm}\label{t39} Suppose that $\lambda$ is an uncountable cardinal, and $S\s E^{\lambda^+}_{\cf(\lambda)}$ is stationary.

Then $\langle\lambda\rangle^-_S$ entails $\p(\lambda^+,2,{\sql},\lambda^+,\{S\},2,\sigma,\mathcal E_\lambda)$ for any regular $\sigma<\lambda$.
\end{thm}
\begin{proof}
By Lemma~\ref{reflected-diamond-equivalence}, $\diamondsuit(\lambda^+)$ holds, and so it suffices to establish
$\p^-(\lambda^+,2,{\sql},\lambda^+,\{S\},2,\sigma,\mathcal E_\lambda)$.
Let $\langle C_\delta\mid\delta\in S\rangle$ and $\langle A_\beta\mid\beta<\lambda^+\rangle$ be given by Lemma~\ref{reflected-diamond-equivalence}(2).
Let $\sigma<\lambda$ be an arbitrary infinite regular cardinal.
By replacing $\sigma$ with $\sigma^+$ if necessary,
we may assume that $\sigma\neq\cf(\lambda)$ (and still $\sigma < \lambda$).

Fix a bijection $\pi:\lambda^+\times\lambda^+\leftrightarrow\lambda^+$,
and let $E:=\{\gamma<\lambda^+\mid \pi[\gamma\times\gamma]=\gamma\}$ denote its club of closure points.

Let $D_0 := \emptyset$, and for every $\delta < \lambda^+$, let $D_{\delta+1} := \{\delta\}$.
For every $\delta\in \acc(\lambda^+) \setminus S$, let $D_\delta$ be an arbitrary club subset of $\delta$ of order-type $\cf(\delta)$.

Next, let $\delta\in S$ be arbitrary.
Define
\[
N_\delta := \left\{ \beta\in\nacc(C_\delta)\cap E \mid  \text{for all } i,\gamma<\beta,
\text{ there exists } \tau\in\beta\setminus\gamma \text{ with }
\pi(i,\tau)\in A_\beta\right\}.
\]

Define $o_\delta:\nacc(C_\delta)\rightarrow\lambda$ by letting for all $\gamma\in\nacc(C_\delta)$:
   $$o_\delta(\gamma):=\otp(\{\beta\in N_\delta\cap E^\gamma_\sigma\mid A_\beta=A_\gamma\cap\beta\}).$$

Fix a surjection $f_\delta:\lambda\rightarrow\delta$ such that the preimage of any singleton is cofinal in $\lambda$,
and then define $g_\delta:\nacc(C_\delta)\rightarrow\delta$ by letting for all $\beta \in \nacc(C_\delta)$:
$$g_\delta(\beta):=\begin{cases}f_\delta(o_\delta(\beta)), &\text{if } f_\delta(o_\delta(\beta))<\beta;\\
0,&\text{otherwise.}\end{cases}$$

For all $\beta\in\nacc(C_\delta)$, let $H^\delta_\beta:=\left\{\tau\mid \sup(C_\delta\cap\beta)<\tau<\beta\ \&\ \pi(g_\delta(\beta),\tau)\in A_\beta\right\}$.
Then let $F^\delta_\beta$ be the closure of $\suc_\sigma(H^\delta_\beta)$. Clearly,
 $F^\delta_\beta$ is a closed subset of $(\sup(C_\delta\cap\beta),\beta]$ of order-type $\le\sigma+1$.
Finally, let
$$D_\delta:=C_\delta\cup\bigcup\{ F^\delta_\beta \mid \beta\in\nacc(C_\delta)\}.$$
As $\sigma<\lambda$, $(\sigma+1)\cdot\lambda=\lambda$, and so $D_\delta$ is a club subset of $\delta$ of order-type $\lambda$.

\begin{claim}\label{c6231} $\langle D_\delta \mid\delta<\lambda^+\rangle$ witnesses $\p^-(\lambda^+,2,{\sql},\lambda^+,\{S\},2,\sigma, \mathcal E_\lambda)$.
 \end{claim}
 \begin{proof} As $\otp(D_\delta)\le\lambda$ for all $\delta<\lambda^+$, the verification of $\sql$ becomes trivial.
Now, given a sequence $\langle X_i\mid i < \lambda^+ \rangle$ of cofinal subsets of $\lambda^+$,
we shall seek stationarily many $\delta\in S$ such that for every $i<\delta$,
$$\sup\{ \beta \in D_\delta\mid \suc_\sigma(D_\delta \setminus \beta) \subseteq X_i \} = \delta.$$

Consider the club $D:=E\cap \diagonal_{i<\lambda^+}(\acc^+(X_i))$, and the set $A:=\{ \pi(i,\tau)\mid i<\lambda^+, \tau\in X_i\}$.
Then, the set $G$ of all $\delta\in S$ such that
$$M_\delta:=\{ \beta\in \nacc(C_\delta) \cap D\cap E^{\lambda^+}_\sigma\mid  A\cap\beta=A_\beta\}\text{ has order-type }\lambda,$$
is stationary.

Let $\delta$ be an arbitrary element of the stationary set $G$.
Let us first show that
\[
M_\delta = \{ \beta \in N_\delta \cap E^{\lambda^+}_\sigma \mid A \cap \beta = A_\beta\}:
\]

Let $\beta \in M_\delta$ be arbitrary. Then $\beta \in D \subseteq E$
and $A_\beta = A\cap \beta$.
By $\beta \in D$, we also have  $\beta \in \bigcap_{i<\beta}\acc^+(X_i)$.
Thus, for all $i, \gamma < \beta$, there is some $\tau \in X_i \cap (\beta \setminus \gamma)$,
so that $\pi(i, \tau) \in A$, and (since $\beta \in E$) $\pi(i, \tau) < \beta$,
giving $\pi(i, \tau) \in A \cap \beta = A_\beta$.
Thus $\beta \in N_\delta$.

Conversely, suppose $\beta \in N_\delta \cap E^{\lambda^+}_\sigma$ satisfies $A \cap \beta = A_\beta$.
We have $\beta \in N_\delta \subseteq \nacc(C_\delta) \cap E$,
so it remains to show that $\beta \in \diagonal_{i<\lambda^+}(\acc^+(X_i))$.
Consider any $i, \gamma < \beta$.
Since $\beta \in N_\delta$,
we can fix $\tau\in\beta\setminus\gamma$ such that
$\pi(i,\tau)\in A_\beta= A \cap \beta$.
Then $\tau \in X_i \cap (\beta \setminus \gamma)$, as required.
Thus $\beta \in M_\delta$.

For any $\gamma \in M_\delta$,
\begin{align*}
o_\delta(\gamma)
    &= \otp \left(\left\{ \beta\in N_\delta\cap E^\gamma_\sigma \mid
        A_\beta= A_\gamma\cap \beta \right\} \right)    \\
    &= \otp \left(\left\{ \beta\in N_\delta\cap E^\gamma_\sigma \mid
        A_\beta= (A \cap \gamma) \cap \beta \right\} \right)    \\
    &= \otp (M_\delta \cap \gamma).
\end{align*}
Since $\otp(M_\delta) = \lambda$, it follows that $o_\delta[M_\delta] = \lambda$.
Also, $\sup(M_\delta) = \sup(C_\delta) = \delta$.

Finally, let $i,\eta<\delta$ be arbitrary. We shall find $\beta'\in D_\delta\setminus\eta$ such that $\suc_\sigma(D_\delta\setminus\beta')\s X_i$.
Fix a large enough $\beta\in M_\delta$ such that $i,\eta<\sup(C_\delta\cap\beta)$ and $f_\delta(o_\delta(\beta))=i$. Then $g_\delta(\beta)=i$.
since $\beta\in M_\delta$,
we get that $H^\delta_\beta=\{\tau\mid \sup(C_\delta\cap\beta)<\tau<\beta, \tau\in X_i\cap\beta\}$ is cofinal in $\beta$,
so that $\otp(D_\delta\cap(\sup(C_\delta\cap\beta),\beta))\ge\cf(\beta)=\sigma$.
Set $\beta':=\sup(C_\delta\cap\beta)$.
Then $\beta'\in D_\delta \setminus \eta$ and $\suc_\sigma(D_\delta\setminus\beta') \subseteq H^\delta_\beta\s X_i$, as sought.
\end{proof}
\end{proof}

\begin{cor}\label{c55} For every regular uncountable cardinal $\lambda$ and stationary $S\s E^{\lambda^+}_\lambda$, the following are equivalent:
\begin{enumerate}
\item $\langle\lambda\rangle^-_S$;
\item $\p(\lambda^+,2,{\sqleft{\lambda}},\lambda^+,\{S\},2,\sigma,\mathcal E_\lambda)$ for every regular cardinal $\sigma<\lambda$;
\item $\p(\lambda^+,2,{\sqleft{\lambda}},1,\{S\},2,1,\mathcal E_\lambda)$.
\end{enumerate}
\end{cor}
\begin{proof} $(1)\Rightarrow(2)$ By Theorem~\ref{t39}.

$(3)\Rightarrow(1)$
By the hypothesis $\p(\dots)$, $\diamondsuit(\lambda^+)$ holds, so by Fact~\ref{fact922},
let us fix a sequence $\langle A_\beta\mid \beta<\lambda^+\rangle$
such that for every $A\s\lambda^+$ and every infinite regular $\sigma<\lambda$,
the set $\{ \beta\in E^{\lambda^+}_\sigma\mid A\cap\beta=A_\beta\}$ is stationary.
Let  $\langle C_\delta\mid \delta<\lambda^+\rangle$ be a witness to
$\p^-(\lambda^+,2,{\sqleft{\lambda}}, 1,\{S\},2,1,\mathcal E_\lambda)$.
We verify that
$\langle C_\delta\mid \delta\in S\rangle$ and $\langle A_\beta\mid \beta<\lambda^+\rangle$ satisfy Clause~(2) of Lemma~\ref{reflected-diamond-equivalence}:

Given a club $D\s\lambda^+$, a subset $A\s\lambda^+$, and an infinite regular $\sigma<\lambda$,
the set $Z:=\{\beta\in D\cap E^{\lambda^+}_\sigma\mid A\cap\beta=A_\beta\}$ is cofinal in $\lambda^+$
(indeed, even stationary).
Hence,
there exist stationarily many $\delta\in S$ such that $\otp(\nacc(C_\delta)\cap Z)=\lambda$, as sought.
\end{proof}

\begin{thm}\label{thm616a}
Suppose that $\lambda$ is an infinite cardinal, and $S\s E^{\lambda^+}_{\cf(\lambda)}$ is stationary.

If $\diamondsuit(S)$ holds, then so does $\p(\lambda^+,2,{\sql},\lambda^+,\{S\},2,\sigma,\mathcal E_\lambda)$ for every ordinal $\sigma<\lambda$.
\end{thm}
\begin{proof} Recalling Theorem~\ref{thm309}, we may assume that $\lambda$ is uncountable.

For regular cardinals $\sigma<\lambda$, we could have simply used Fact~\ref{factRinot09} together with Theorem~\ref{t39},
but let us a give a proof that works for all cases.

Let $\lambda$ be an arbitrary uncountable cardinal, and let $\sigma<\lambda$ be some nonzero ordinal. By $\diamondsuit(S)$, we have $\diamondsuit(\lambda^+)$, and it remains to establish
$\p^-(\lambda^+,2,{\sql},\lambda^+,\{S\},2, \sigma,\mathcal E_\lambda)$.

Using $\diamondsuit(S)$,
fix a sequence $\langle (X_\alpha, Y_\alpha) \mid \alpha <\lambda^+ \rangle$
such that for all $X, Y \subseteq \lambda^+$ the set
$\{ \alpha \in S \mid X \cap \alpha = X_\alpha \ \&\  Y \cap \alpha = Y_\alpha \}$ is stationary.
Of course, we may assume that $X_\alpha$ and $Y_\alpha$ are subsets of $\alpha$. Denote $$H_\alpha:=\{ \gamma\in Y_\alpha \mid X_\alpha\cap\gamma=X_\gamma\ \&\ Y_\alpha\cap\gamma=Y_\gamma \}.$$

Let $D_0:=\emptyset$ and $D_{\alpha+1}:=\{\alpha\}$ for all $\alpha<\lambda^+$.
Next, for every nonzero limit $\alpha<\lambda^+$, we do the following.
If there exists a club $D_\alpha$ in $\alpha$ of order-type $\lambda$ such that $\nacc(D_\alpha)\s H_\alpha$,
we let $D_\alpha$ be such a club. Otherwise,
we let $D_\alpha$ be any club subset of $\alpha$ of order-type $\cf(\alpha)$.
Just as in Theorem~\ref{thm32}, write $\chi:=\omega\cdot\lambda$.
Of course, since we have assumed $\lambda$ is uncountable, it follows that $\chi = \lambda$.

\begin{claim}\begin{enumerate}
\item for every limit $\alpha<\lambda^+$, $D_\alpha$ is a club in $\alpha$ of order-type $\le\chi$;
\item if $\bar\alpha\in\acc(D_\alpha)$, then $D_{\bar\alpha}\sql D_\alpha$;
\item for every subset $X\s \lambda^+$ and club $E\s \lambda^+$, there exists a limit $\alpha\in S$ such that $\otp(D_\alpha)=\chi$ and $\nacc(D_\alpha)\s\{\gamma\in E\cap S\mid X\cap\gamma=X_\gamma\}$.
\end{enumerate}
\end{claim}

\begin{proof} (1) is trivial.

(2) follows from (1) and the fact that $\chi=\lambda$.

(3)
Fix $X \subseteq \lambda^+$ and club $E \subseteq \lambda^+$.
Define
\[
G := \{ \gamma \in E \cap S \mid X \cap \gamma = X_\gamma\ \&\  E \cap S \cap \gamma = Y_\gamma \}.
\]
By applying our diamond sequence to $X$ and $Y := E \cap S$,
we find that $G$ is a stationary subset of $\lambda^+$,
being the intersection of the club set $E$ with a stationary set.
Thus, in particular, $Z:=\{\alpha<\lambda^+\mid \otp(G\cap\alpha)=\alpha\text{ is divisible by }\lambda\}$ is club in $\lambda^+$,
and it follows that $G \cap Z$ is stationary in $\lambda^+$.
Choose $\alpha \in G\cap Z$.
Clearly, $\alpha \in G \subseteq S \subseteq E^{\lambda^+}_{\cf(\lambda)}$ and $G\cap\alpha=H_\alpha$.
Since $\cf(\alpha) = \cf(\lambda)$ and $\alpha = \otp(H_\alpha)$ is divisible by $\lambda$,
it follows that $D_\alpha$ is a club of order-type $\lambda$ such that
$\nacc(D_\alpha) \subseteq H_\alpha \subseteq G$.
Then, as in Claim~\ref{DinG}, it follows that $\nacc(D_\alpha) \subseteq G$, as well as $\otp(D_\alpha)=\lambda$,
giving the required result.
\end{proof}

Now, continue with the very same construction of Theorem~\ref{thm32}
to get a sequence $\langle C_\alpha\mid\alpha<\lambda^+\rangle$
such that $C_\alpha$ is a club in $\alpha$ of order-type $\le\lambda$, and
for every sequence $\langle A_\delta\mid \delta<\lambda^+\rangle$ of cofinal subsets of $\lambda^+$
and every club $D \subseteq \lambda^+$,
there exists $\alpha\in S \cap  D$ such that
\[
\otp(\{\beta\in \acc(C_\alpha) \mid \suc_\sigma(C_\alpha\setminus\beta)\s A_\delta\})=\lambda
\text{ for every }\delta<\alpha.
\]

Then $\langle C_\alpha\mid\alpha<\lambda^+\rangle$
witnesses $\p^-(\lambda^+,2,{\sql},\lambda^+,\{S\},2,\sigma,\mathcal E_\lambda)$.
\end{proof}
\begin{thm}\label{thm313} Suppose that $\lambda$ is a given uncountable cardinal.

If $\ch_\lambda + \lambda^{<\lambda}=\lambda$,
then $V^{\add(\lambda,1)}\models\p(\lambda^+,2,{\sql},\lambda^+,\{S\s E^{\lambda^+}_\lambda\mid S\text{ is stationary}\},2,\sigma, \mathcal E_\lambda)$ for every cardinal $\sigma<\lambda$.
\end{thm}
\begin{proof} This is the same proof as of Theorem \ref{thm67}. Only this time, we do not have to care about coherence.
\end{proof}

\section{\texorpdfstring{The coherence relation $\sqx^*$}{Left-subscript star-coherence}}

A construction of a Souslin tree using the relation $\sqx^*$ may be found in \cite{rinot23}.

\begin{lemma}\label{l11}
Suppose that $\boxminus_{\lambda,\ge\chi}+\ch_\lambda$ holds
for given infinite regular cardinals $\chi < \lambda$.

Then, there exist $S\s E^{\lambda^+}_\chi\s E^{\lambda^+}_{\ge\chi}\s\Gamma\s\acc(\lambda^+)$, and sequences $\langle c_\gamma\mid\gamma\in\Gamma\rangle$,
$\langle C_\gamma\mid\gamma<\lambda^+\rangle$, and $\langle X_{\beta}\mid\beta<\lambda^+\rangle$
that satisfy the following:
\begin{enumerate}
\item if $\gamma\in\acc(\lambda^+)$, then $C_\gamma$ is a club in $\gamma$;
\item if $\gamma\in\acc(\lambda^+)$ and ${\bar\gamma}\in\acc(C_\gamma)$,
then ${\bar\gamma}\notin S$ and $C_{\bar\gamma} \sqx^* C_\gamma$;
\item if $\gamma\in\Gamma$, then $c_\gamma$ is a club in $\gamma$ with $\otp(c_\gamma)\le\lambda$;
\item if $\gamma\in\Gamma$ and ${\bar\gamma}\in\acc(c_\gamma)$, then  ${\bar\gamma}\in\Gamma\setminus S$, $c_{\bar\gamma}=c_\gamma\cap{\bar\gamma}$, and $C_{\bar\gamma}=C_\gamma\cap{\bar\gamma}$;
\item for every subset $X\s\lambda^+$ and every club $D\s\lambda^+$, there exists some $\gamma\in E^{\lambda^+}_\lambda$
such that $\min(C_\gamma)\in D$, and
$$\sup(\{\beta\in\nacc(C_\gamma)\cap D\cap S\mid X_\beta=X\cap\beta\})=\gamma.$$
\end{enumerate}
\end{lemma}

\begin{proof} This is the argument of \cite[Theorem~3]{KjSh:449}, modulo various adjustments.

By applying Lemma \ref{l24} with $S = E^{\lambda^+}_\chi$ and $\eta = \lambda$,
let us fix a sequence $\langle c_\gamma\mid\gamma\in \Gamma\rangle$ and a stationary subset $S'\s E^{\lambda^+}_\chi$ such that:
\begin{itemize}
\item $E^{\lambda^+}_{\ge\chi}\s\Gamma\s \acc(\lambda^+)$;
\item if $\gamma\in\Gamma$, then $c_\gamma$ is a club subset of $\gamma$ of order-type $\le\lambda$;
\item if $\gamma\in\Gamma$ and  $\bar\gamma\in\acc(c_\gamma)$, then $\bar\gamma\in\Gamma\setminus S'$ and $c_{\bar\gamma}=c_\gamma\cap \bar\gamma$;
\item for every club $D\s\lambda^+$, there exist stationarily many $\gamma\in E^{\lambda^+}_\lambda$ such that $\min(c_\gamma)\in D$.
\end{itemize}

Next, by $\ch_\lambda$
and the fact that $S'$ is a stationary subset of $E^{\lambda^+}_{\neq\cf(\lambda)}$, we get from Fact \ref{fact922} that $\diamondsuit(S')$ holds.
We shall use $\diamondsuit(S')$ to guess subsets of $\omega \times \lambda^+$ (rather than subsets of $\lambda^+$).
More specifically, we fix a matrix
$\mathbb X=\langle X_\beta^n\mid n<\omega, \beta<\lambda^+ \rangle$
such that for every sequence $\langle X^n\mid n<\omega\rangle$ of subsets of $\lambda^+$, there exist stationarily many $\beta\in S'$
such that $\bigwedge_{n<\omega}X^n_\beta=X^n\cap\beta$.

We now attempt to construct, recursively, sequences $\langle S^n \mid n < \omega \rangle$ and
$\langle \langle C^n_\gamma \mid \gamma<\lambda^+ \rangle \mid n<\omega \rangle$
satisfying the following properties for all $n < \omega$ and all nonzero limit ordinals $\gamma < \lambda^+$:
\begin{enumerate}
\item[(a)] $S^{n+1} \subseteq S^n \subseteq S'$;
\item[(b)] $C^{n}_\gamma \s C^{n+1}_\gamma\s\gamma$;
\item[(c)] $C_\gamma^n$ is a club in $\gamma$, and if $\gamma \in \Gamma$ then $\min(C^n_\gamma)=\min(c_\gamma)$;
\item[(d)] if $\gamma\in\Gamma$ and  ${\bar\gamma}\in\acc(c_\gamma)$, then $C^n_{\bar\gamma}=C^n_\gamma\cap{\bar\gamma}$;
\item[(e)] if ${\bar\gamma}\in\acc(C^n_\gamma)$, then ${\bar\gamma}\notin S^n$ and $C^n_{\bar\gamma} \sqx^* C^n_\gamma$.
\end{enumerate}

Let $S^0:=S'$, and
let $C^0_\gamma:=c_\gamma$ for all $\gamma \in \Gamma$.
Let $C^0_0 = \emptyset$ and let $C^0_{\gamma+1}:=\{\gamma\}$ for all $\gamma<\lambda^+$.
For all $\gamma\in\acc(\lambda^+)\setminus\Gamma$,
let $C^0_\gamma$ be some club subset of $\gamma$ of order-type $\cf(\gamma)$.
Then $\acc(C^0_\gamma)\cap E^{\lambda^+}_{\ge\chi}=\emptyset$ for all $\gamma\in \lambda^+ \setminus \Gamma$.
In particular,  $S'\cap \acc(C^0_\gamma)=\emptyset$ for all $\gamma<\lambda^+$.
Thus,
it is clear that $S^0$ and $\langle C^0_\gamma \mid \gamma < \lambda^+ \rangle$ satisfy the above properties.

Now, fix $n<\omega$, and suppose that $S^n$ and $\langle C^n_\gamma\mid\gamma<\lambda^+\rangle$
have been constructed to satisfy the above properties.
It is clear that $S^n$, $\langle c_\gamma\mid\gamma\in \Gamma\rangle$,
$\langle C^n_\gamma\mid\gamma<\lambda^+\rangle$,
and $\langle X^n_\beta \mid\beta<\lambda^+\rangle$
satisfy properties (1)--(4) of the conclusion of the lemma.
If property~(5) is satisfied as well, then the lemma is proven and we abandon the recursive construction at this point.

Otherwise, there must exist a subset $X^n\s\lambda^+$ and a club $D^n\s\lambda^+$ such that
for all $\gamma\in E^{\lambda^+}_\lambda$, either $\min(C^n_\gamma)\notin D^n$, or  $\sup(\nacc(C^n_\gamma)\cap S^{n+1})<\gamma$, where we define:
$$S^{n+1}:=\{\beta\in S^n\cap D^n\mid X^n_\beta=X^n\cap\beta\}.$$

Define $C^{n+1}_\gamma$ by recursion over $\gamma<\lambda^+$, as follows. Let $C^{n+1}_0=\emptyset$, and $C^{n+1}_{\gamma+1}:=\{\gamma\}$.
Now, if $\gamma<\lambda^+$ is a nonzero limit ordinal, and $\langle C^{n+1}_{\bar\gamma}\mid {\bar\gamma}<\gamma\rangle$ has already been defined, we let
$$C^{n+1}_\gamma:=C^n_\gamma\cup \bigcup \{C^{n+1}_\beta\setminus\sup(C^n_\gamma\cap\beta)\mid \beta\in\nacc(C^n_\gamma)\setminus S^{n+1}, \beta>\min(C^n_\gamma)\}.$$

\begin{claim}\label{c251}
$S^{n+1}$ and $\langle C^{n+1}_\gamma \mid \gamma<\lambda^+ \rangle$
satisfy properties (a)--(e) of the recursion.
\end{claim}
\begin{proof}
(a)--(d) are trivial.

(e) Suppose not, and let $\gamma$ denote the least counterexample.
Let $\bar\gamma\in\acc(C^{n+1}_\gamma)$ be arbitrary. We address the two possible cases.

$\br$ Suppose that  $\bar\gamma\in S^{n+1}$.
Since $S^{n+1}\s S^n$ and (by the previous step of the recursion) $\acc(C_\gamma^n)\cap S^n=\emptyset$,
we get from ${\bar\gamma}\in\acc(C^{n+1}_\gamma)$ and the definition of $C_\gamma^{n+1}$, the existence of $\beta\in\nacc(C^n_\gamma)\setminus S^{n+1}$
such that either ${\bar\gamma}\in\acc(C^{n+1}_\beta\setminus\sup(C^n_\gamma\cap\beta))$ or $\beta={\bar\gamma}$.
By minimality of $\gamma$, we have $\acc(C^{n+1}_\beta)\cap S^{n+1}=\emptyset$ for all $\beta<\gamma$,
and hence it must be the case that ${\bar\gamma}=\beta$. So ${\bar\gamma}\in\nacc(C^n_\gamma)\setminus S^{n+1}$, contradicting the fact that ${\bar\gamma}\in S^{n+1}$.

$\br$ Suppose that $C^{n+1}_{\bar\gamma} \nsqx^* C^{n+1}_\gamma$.
Then $\cf(\bar\alpha) \geq\chi$ and $C^{n+1}_{\bar\gamma} \not\sqsubseteq^* C^{n+1}_\gamma$,
so that
for no $\alpha<\bar\gamma$ do we have $C^{n+1}_{\bar\gamma}\cap(\alpha,{\bar\gamma})=C^{n+1}_\gamma\cap(\alpha,{\bar\gamma})$.
Clearly, $\bar\gamma$ cannot be in $\acc(C^n_{\gamma})$.
Let $\beta\in\nacc(C^n_\gamma)\setminus S^{n+1}$ be such that either ${\bar\gamma}\in\acc(C^{n+1}_\beta\setminus\sup(C^n_\gamma\cap\beta))$ or $\beta={\bar\gamma}$.
If $\beta=\bar\gamma$, then $C^{n+1}_{\bar\gamma}\cap(\alpha,{\bar\gamma})=C^{n+1}_\gamma\cap(\alpha,{\bar\gamma})$ for $\alpha:=\sup(C^n_\gamma\cap\beta)$.
So $\bar\gamma<\beta$, and $\beta$ contradicts the minimality of $\gamma$.
\end{proof}

If we reach the end of the above recursive process, then we are equipped with a sequence
$\langle (\langle C^n_\gamma\mid\gamma<\lambda^+\rangle,S^n, D^n,X^n)\mid n<\omega\rangle$,
from which we shall derive a contradiction.
The set $\bigcap_{n<\omega}D^n$ is club in $\lambda^+$.
Thus, by the choice of $\mathbb X$, the following set must be stationary:
$$S'' :=\bigcap_{n<\omega} S^n=\left\{ \beta\in S'\cap\bigcap_{n<\omega}D^n\mid \bigwedge_{n<\omega}X^n_\beta=X^n\cap\beta\right\}.$$
Thus, $\acc^+(S'')$ is a club in $\lambda^+$, and we may
pick $\gamma\in E^{\lambda^+}_\lambda\cap\acc^+(S'')$ such that $\min(c_\gamma)\in\bigcap_{n<\omega}D^n$.
For all $n<\omega$, put $\gamma_n:=\sup(\nacc(C^n_\gamma)\cap S^{n+1})$.
By $\min(C_\gamma^n)\in D^n$, we have $\gamma_n<\gamma$.
By $\cf(\gamma)=\lambda>\aleph_0$, we get $\gamma^*<\gamma$,
where $\gamma^*:=\sup_{n<\omega}\gamma_n$.

\begin{claim}\label{empty-int}
$C^n_\gamma \cap S^{n+1} \cap (\gamma^*, \gamma) =  \emptyset$ for all $n<\omega$.
\end{claim}
\begin{proof}
Fix any $n < \omega$, and
suppose $\beta \in S^{n+1} \cap (\gamma^*, \gamma)$.
Then $\beta \in S^{n+1} \subseteq S^n$.
By Clause~(e) of the recursion, it follows that $\beta \notin \acc(C^n_\gamma)$.
But also $\beta > \gamma^* \geq \gamma_n$,
so that $\beta \notin \nacc(C^n_\gamma) \cap S^{n+1}$,
and it follows that $\beta \notin \nacc(C^n_\gamma)$.
Thus, altogether, $\beta \notin C^n_\gamma$.
\end{proof}

Since $S'' \cap \gamma$ is cofinal in $\gamma$,
let us pick $\beta\in S'' \cap(\gamma^*,\gamma)$ above $\min(c_\gamma)$.

For all $n<\omega$, let $\beta_n:=\min(C^n_\gamma\setminus\beta)$.
As $\{ C^n_\gamma\mid n<\omega\}$ is an increasing chain, $\langle \beta_n\mid n<\omega\rangle$
is weakly decreasing, and hence stabilizes. Fix $n<\omega$ such that $\beta_n=\beta_{n+1}$.
Since $\beta \in S'' \subseteq S^{n+1}$,
Claim~\ref{empty-int} gives $\beta \notin C^n_\gamma$, so that $\beta_n > \beta$.
In particular, $\beta_n = \min(C^n_\gamma \setminus (\beta+1))$,
so that $\beta_n\in\nacc(C^n_\gamma)$.
By Claim~\ref{empty-int} again and $\gamma > \beta_n > \beta > \gamma^*$,
it follows that $\beta_n\in\nacc(C^n_\gamma)\setminus S^{n+1}$.
Recalling that also $\beta_n >\beta >\min(c_\gamma) = \min(C^n_\gamma)$,
we infer from the definition of $C^{n+1}_\gamma$ that
$$C^{n+1}_\gamma\cap[\sup(C^n_\gamma\cap\beta_n),\beta_n)=C^{n+1}_{\beta_n}\cap[\sup(C^n_\gamma\cap\beta_n),\beta_n),$$
and by $\beta_n=\min(C^n_\gamma\setminus\beta)$, we have $\sup(C^n_\gamma\cap\beta_n)\le\beta$, and hence
$$C^{n+1}_\gamma\cap[\beta,\beta_n)=C^{n+1}_{\beta_n}\cap[\beta,\beta_n).$$

Since $\beta < \beta_n$, it follows that
$\beta_{n+1} = \min (C^{n+1}_\gamma \setminus \beta) = \min(C^{n+1}_{\beta_n} \setminus \beta) < \beta_n$, contradicting the choice of $n$.
\end{proof}

\begin{cor}\label{c93}
Suppose that $\boxminus_{\lambda,\ge\chi}+\ch_\lambda$ holds
for given infinite regular cardinals $\chi < \lambda$.
Then both of the following hold:
\begin{enumerate}
\item $\p(\lambda^+,2,{\sqx^*},1,\{E^{\lambda^+}_{\lambda}\},2,\chi)$;
\item $\p(\lambda^+,2,{\sqx^*},\chi,\{E^{\lambda^+}_{\lambda}\},2,1)$.
\end{enumerate}
\end{cor}
\begin{proof} As $\lambda$ is uncountable, Fact~\ref{fact922} entails $\diamondsuit(\lambda^+)$,
so that we only need to establish the corresponding  $\p^-(\dots)$ principles of Clauses~(1) and~(2).

Let $S\s E^{\lambda^+}_\chi$ along with sequences $\langle C_\beta\mid\beta<\lambda^+\rangle$ and $\langle X_\beta\mid\beta<\lambda^+\rangle$
be given by Lemma \ref{l11}.
We may assume that $X_\beta \subseteq \beta$ for all $\beta < \lambda^+$.

(1) We shall define a sequence $\langle C_\alpha^\bullet\mid \alpha<\lambda^+\rangle$ that witnesses
$\p^-(\lambda^+,2,{\sqx^*},1,\{E^{\lambda^+}_{\lambda}\},2,\chi)$.
First, for all limit $\beta<\lambda^+$
let $D_\beta$ be some club in $\beta$ of order-type $\cf(\beta)$,
with the additional constraint that if $\sup(X_\beta)=\beta$, then $\nacc(D_\beta)\s X_\beta$.
Then, for all limit $\alpha<\lambda^+$, let
\[
C^\bullet_\alpha:=\begin{cases}
C_\alpha\cup \bigcup \{ D_\beta\setminus\sup(C_\alpha\cap\beta)\mid \beta\in\nacc(C_\alpha)\cap S\},
                &\text{if }\alpha\in E^{\lambda^+}_{\ge\chi}\setminus S;\\
D_\alpha,&\text{otherwise.}
\end{cases}
\]

It is clear that $C_\alpha^\bullet$ is a club subset of $\alpha$.
Define $C_{\alpha+1}^\bullet := \{\alpha\}$ for all $\alpha < \lambda^+$.

Fix $\alpha < \lambda^+$ and $\bar\alpha \in \acc(C_\alpha^\bullet)$.
To show that $C_{\bar\alpha}^\bullet \sqx^* C_\alpha^\bullet$, we consider two possibilities:
\begin{itemize}
\item
If $\alpha \notin E^{\lambda^+}_{\ge\chi}\setminus S$,
then $C_\alpha^\bullet = D_\alpha$, and in particular (since $S \subseteq E^{\lambda^+}_\chi$)
$\alpha \in E^{\lambda^+}_{\leq\chi}$,
so that
$\cf(\sup(C_{\bar\alpha}^\bullet)) = \cf(\bar\alpha) \leq \otp(C_\alpha^\bullet \cap \bar\alpha)
< \otp(C_\alpha^\bullet) = \otp(D_\alpha) = \cf(\alpha) \leq \chi$,
and it follows that $C_{\bar\alpha}^\bullet \sqx^* C_\alpha^\bullet$.
\item
If $\alpha \in E^{\lambda^+}_{\ge\chi}\setminus S$,
then there are three cases to consider:
\begin{itemize}
\item
If $\bar\alpha \in \acc(C_\alpha)$,
then by Clause~(2) of Lemma~\ref{l11} we have $\bar\alpha \notin S$ and
$C_{\bar\alpha} \sqx^* C_\alpha$.
If $\cf(\bar\alpha) < \chi$ then we automatically have $C_{\bar\alpha}^\bullet \sqx^* C_\alpha^\bullet$.
Thus we assume $\cf(\bar\alpha) \geq\chi$, so that $C_{\bar\alpha} \sqsubseteq^* C_\alpha$.
Fix $\gamma < \bar\alpha$ such that $C_{\bar\alpha} \setminus \gamma \sqsubseteq C_\alpha \setminus \gamma$.
Since $\alpha, \bar\alpha \in E^{\lambda^+}_{\geq\chi} \setminus S$,
it is clear from the definition of $C_\alpha^\bullet$ in this case that
$C_{\bar\alpha}^\bullet \setminus \gamma \sqsubseteq C_\alpha^\bullet \setminus \gamma$,
so that $C_{\bar\alpha}^\bullet \sqsubseteq^* C_\alpha^\bullet$.

\item
If $\bar\alpha \in \nacc(C_\alpha) \cap S$,
then $C_{\bar\alpha}^\bullet = D_{\bar\alpha}$,
so that, letting $\gamma = \sup(C_\alpha \cap \bar\alpha) +1$, we have
$C_{\bar\alpha}^\bullet \setminus \gamma = D_{\bar\alpha} \setminus \gamma
\sqsubseteq C_\alpha^\bullet \setminus \gamma$,
so that again $C_{\bar\alpha}^\bullet \sqsubseteq^* C_\alpha^\bullet$.

\item
If $\bar\alpha \in \acc(D_\beta)$ for some $\beta \in \nacc(C_\alpha) \cap S$,
then in particular $\cf(\beta)  = \chi$,
so that
$\cf(\sup(C_{\bar\alpha}^\bullet)) = \cf(\bar\alpha) \leq \otp(D_\beta \cap \bar\alpha)
< \otp(D_\beta) = \cf(\beta) = \chi$,
and it follows that $C_{\bar\alpha}^\bullet \sqx^* C_\alpha^\bullet$.
\end{itemize}
\end{itemize}

Thus, to verify that $\langle C^\bullet_\alpha\mid\alpha<\lambda^+\rangle$ witnesses
$\p^-(\lambda^+, 2, {\sqx^*}, 1, \{E^{\lambda^+}_{\lambda}\}, 2, \chi)$,
consider any cofinal subset $A_0 \subseteq \lambda^+$ and any club $E \subseteq \lambda^+$.
Let $\alpha \in E^{\lambda^+}_\lambda$ be given by Clause~(5) of Lemma~\ref{l11}
applied with $X := A_0$ and $D := E \cap \acc^+(A_0)$.
Define
\[
Z := \{ \beta \in \nacc(C_\alpha) \cap E \cap \acc^+(A_0) \cap S \mid X_\beta = A_0 \cap \beta \},
\]
so that $\sup(Z) = \alpha$.
Since $Z \subseteq E$ and $E$ is club, it follows that $\alpha \in \acc^+(E) \subseteq E$.

Since $S \subseteq E^{\lambda^+}_\chi$ and $\chi < \lambda =\cf(\alpha)$,
we clearly have $\alpha \in E^{\lambda^+}_{\geq\chi} \setminus S$.
Now, consider any $\beta \in Z$, and let $\beta^- := \sup(C_\alpha \cap \beta)$.
Since $\beta \in \nacc(C_\alpha) \cap S$, it follows that $\beta^- \in C_\alpha \cap \beta$ and
$D_\beta \cap (\beta^-, \beta) = C_\alpha^\bullet \cap (\beta^-, \beta)$.
Furthermore, since $\beta \in S \subseteq E^{\lambda^+}_\chi$, we have $\otp(D_\beta) = \cf(\beta) = \chi$.
Thus also $\otp(D_\beta \cap (\beta^-, \beta)) = \chi$,
and it follows
that $\suc_\chi(C_\alpha^\bullet \setminus \beta^-) \subseteq \nacc(D_\beta)$.
Since $\beta \in Z \subseteq \acc^+(A_0)$, we have $\sup(X_\beta) = \sup(A_0 \cap \beta) = \beta$,
so that $\nacc(D_\beta) \subseteq X_\beta$.
Altogether, we have
$\suc_\chi(C_\alpha^\bullet \setminus \beta^-) \subseteq \nacc(D_\beta) \subseteq X_\beta \subseteq A_0$.
Since $\sup(Z) = \alpha$, it is clear that also $\sup\{ \beta^- \mid \beta \in Z \} = \alpha$.

(2) The definition of $\langle C_\alpha^\bullet\mid \alpha<\lambda^+\rangle$ that witnesses
$\p^-(\lambda^+, 2, {\sqx^*}, \chi, \{E^{\lambda^+}_{\lambda}\}, 2, 1)$ is as in the previous case, modulo the fact that the clubs $D_\beta$ need to be chosen somewhat differently.
Fix a bijection $\pi:\chi\times\lambda^+\leftrightarrow\lambda^+$. Denote $X_\beta^i:=\{\alpha<\beta\mid \pi(i,\alpha)\in X_\beta\}$.
Then, for all limit $\beta<\lambda^+$, let $D_\beta$ be some club in $\beta$ of order-type $\cf(\beta)$,
with the additional constraint that if $\cf(\beta)=\chi$, then $\sup(\nacc(D_\beta)\cap X_\beta^i)=\beta$ for all $i<\chi$ such that $\sup(X_\beta^i)=\beta$.
For any sequence $\langle A_i \mid i < \chi \rangle$ of cofinal subsets of $\lambda^+$
and any club $E \subseteq \lambda^+$,
the verification includes applying Clause~(5) of Lemma~\ref{l11} with
$X := \bigcup_{i<\chi} \pi [ \{i\} \times A_i ]$ and
$D := E \cap \bigcap_{i<\chi} \acc^+(A_i) \cap \{ \beta < \lambda^+ \mid \pi[\chi\times\beta] = \beta \}$.
\end{proof}

\begin{thm}\label{thm317} If $\ch_\lambda$ holds for a given regular uncountable cardinal $\lambda$,
and there exists a nonreflecting stationary subset of $E^{\lambda^+}_{<\lambda}$, then
$\p(\lambda^+,2,{\sql}^*,\theta,\{E^{\lambda^+}_{\lambda}\},2,\sigma)$ holds for all regular cardinals
$\theta,\sigma<\lambda$.
\end{thm}
\begin{proof} Suppose that $S\s E^{\lambda^+}_{<\lambda}$ is stationary and nonreflecting.
By Fact \ref{fact922}, $\ch_\lambda$ entails $\diamondsuit(S)$.
Recalling that $E^{\lambda^+}_\lambda$ is always a nonreflecting stationary subset of $\lambda^+$
and $S \cap E^{\lambda^+}_\lambda = \emptyset$,
we apply Theorem 3 of \cite{KjSh:449} with $S^* = E^{\lambda^+}_\lambda$
to obtain sequences $\langle C_\alpha\mid\alpha\in E^{\lambda^+}_\lambda\rangle$ and
$\langle Z_\gamma\mid\gamma<\lambda^+\rangle$
that satisfy the following:
\begin{enumerate}
\item $C_\alpha$ is a club in $\alpha$ for all $\alpha\in E^{\lambda^+}_\lambda$;
\item if $\alpha\in E^{\lambda^+}_\lambda$ and $\bar\alpha\in\acc(C_\alpha)\cap E^{\lambda^+}_\lambda$,
then $C_{\bar\alpha} \sqsubseteq^* C_\alpha$;
\item for every subset $Z\s\lambda^+$ and every club $D\s\lambda^+$, there exists some $\alpha\in E^{\lambda^+}_\lambda$
such that$$\sup(\{\gamma\in\nacc(C_\alpha)\cap D\cap S\mid Z_\gamma=Z\cap\gamma\})=\alpha.$$
\end{enumerate}

For every $\alpha\in E^{\lambda^+}_\lambda$, let
$$C_\alpha^\bullet:=\{ \min((Z_\gamma\cup\{\gamma\})\setminus\sup(C_\alpha\cap\gamma))\mid \gamma\in C_\alpha\}.$$
Then:
\begin{enumerate}
\item $C^\bullet_\alpha$ is a club in $\alpha$ for every $\alpha\in E^{\lambda^+}_\lambda$,
with $\acc(C^\bullet_\alpha) = \acc(C_\alpha)$;
\item if $\alpha\in E^{\lambda^+}_\lambda$ and $\bar\alpha\in\acc(C^\bullet_\alpha)\cap E^{\lambda^+}_\lambda$,
then  $C^\bullet_{\bar\alpha} \sqsubseteq^* C^\bullet_\alpha$;
\item for every cofinal subset $Z\s\lambda^+$, there exists some $\alpha\in E^{\lambda^+}_\lambda$
such that $$\sup(\nacc(C^\bullet_\alpha)\cap Z)=\alpha.$$
\end{enumerate}
(To see (3), let $D := \acc^+(Z)$.)

By Fact~\ref{fact922},
$\ch_\lambda$ entails $\diamondsuit(E^{\lambda^+}_\theta)$ for every regular cardinal $\theta < \lambda$.
Thus, we can
fix a matrix $\langle X_\gamma^i\mid i<\lambda,\gamma<\lambda^+\rangle$ with the property that for every sequence $\langle A_i\mid i<\lambda\rangle$ of subsets of $\lambda^+$,
and every regular $\theta<\lambda$, the following set is stationary:
$$\{\gamma\in E^{\lambda^+}_\theta\mid \forall i<\lambda(A_i\cap\gamma=X^i_\gamma)\}.$$

\medskip

We shall construct $D_\alpha \subseteq \alpha$ for every $\alpha < \lambda^+$, as follows:
First, let $D_{\alpha+1} = \{\alpha\}$ for all $\alpha < \lambda^+$.

Then, consider any limit $\alpha\in E^{\lambda^+}_{<\lambda}$.

$\br$ If there exists some club $d$ in $\alpha$ such that $\otp(d)<\lambda$ and for all $i<\cf(\alpha)$
$$\sup\{\beta\in d\mid \suc_{\cf(\alpha)}(d\setminus\beta)\s X_\alpha^i\}=\alpha,$$
then pick such a $d$, and call it $D_\alpha$.

$\br$ Otherwise, let $D_\alpha$ be an arbitrary
club subset of $\alpha$ of order-type $\cf(\alpha)$.

Finally, consider any $\alpha\in E^{\lambda^+}_\lambda$. Put
$$D_\alpha:=C^\bullet_\alpha\cup \bigcup \{ D_\gamma\setminus\sup(C^\bullet_\alpha\cap\gamma)\mid \gamma\in\nacc(C^\bullet_\alpha)\cap E^{\lambda^+}_{<\lambda}\}.$$
Clearly, $D_\alpha$ is a club subset of $\alpha$ for every limit $\alpha < \lambda^+$.

Fix $\alpha < \lambda^+$ and $\bar\alpha \in \acc(D_\alpha)$.
If $\cf(\bar\alpha) < \lambda$, then automatically $D_{\bar\alpha} \sql^* D_\alpha$.
If $\cf(\alpha) < \lambda$, then by construction of $D_\alpha$ we have $\otp(D_\alpha) < \lambda$,
so that $\cf(\bar\alpha) \leq \otp(D_\alpha \cap \bar\alpha) < \otp(D_\alpha) < \lambda$,
and again $D_{\bar\alpha} \sql^* D_\alpha$.
Otherwise, we have $\bar\alpha, \alpha \in E^{\lambda^+}_\lambda$,
so that $C^\bullet_{\bar\alpha} \sqsubseteq^* C^\bullet_\alpha$,
and it follows from the definition of $D_\alpha$ that $D_{\bar\alpha} \sqsubseteq^* D_\alpha$ in this case.

\begin{claim} For all regular cardinals $\sigma,\theta<\lambda$
and every sequence $\langle A_i\mid i<\theta\rangle$ of cofinal subsets of $\lambda^+$,
there exist stationarily many $\alpha\in E^{\lambda^+}_\lambda$ such that for all $i<\theta$:
$$\sup\{\beta\in D_\alpha\mid \suc_\sigma(D_\alpha\setminus\beta)\s A_i\}=\alpha.$$
\end{claim}
\begin{proof} By increasing $\sigma$ and $\theta$ if necessary,
we may assume that $\sigma=\theta$ is an infinite regular cardinal.

Given $\langle A_i\mid i<\theta\rangle$ and a club $E\s\lambda^+$, consider the stationary set
$$Z:=\{\gamma\in E^{\lambda^+}_\theta\cap E\cap\bigcap_{i<\theta}\acc^+(\acc^+(A_i)\cap E^{\lambda^+}_{\sigma})\mid \forall i<\theta(A_i\cap\gamma=X^i_\gamma)\}.$$
Fix $\alpha\in E^{\lambda^+}_\lambda$ such that $\sup(\nacc(C^\bullet_\alpha)\cap Z)=\alpha$.
Then $\alpha\in \acc^+(Z) \subseteq \acc^+(E) \subseteq E$.
So, it suffices to show that for every $\gamma\in\nacc(C^\bullet_\alpha)\cap Z$
and every $i<\theta$, there exists some $\beta>\sup(C^\bullet_\alpha\cap\gamma)$ for which $\suc_\sigma(D_\alpha\setminus\beta)\s A_i$.
Fix $\gamma\in\nacc(C^\bullet_\alpha)\cap Z$.
As $\alpha \in E^{\lambda^+}_\lambda$ and $\gamma\in\nacc(C^\bullet_\alpha)\cap E^{\lambda^+}_{<\lambda}$,
recalling the definition of $D_\alpha$, it then suffices to prove that for all $i<\theta$
$$\sup\{\beta\in D_\gamma\mid \suc_\sigma(D_\gamma\setminus\beta)\s A_i\}=\gamma.$$

Fix a surjection $f:\theta\rightarrow\theta$ such that the preimage of any singleton has size $\theta$.
Since $\gamma\in E^{\lambda^+}_\theta \cap \bigcap_{i<\theta} \acc^+(\acc^+(A_i)\cap E^{\lambda^+}_{\sigma})$,
let us pick a strictly increasing, continuous, and cofinal function $g:\theta\rightarrow \gamma$ with the property that
$g(j+1)\in\acc^+(A_{f(j)})\cap E^{\lambda^+}_{\sigma}$ for all $j<\theta$.
Next, for all $j<\theta$, fix $Y_j\s A_{f(j)}\cap(g(j),g(j+1))$ of order-type $\sigma$, and let $d$ denote the closure of $\bigcup_{j<\theta}Y_j$.
Then $d$ is a club in $\gamma$ of order-type $\sigma \cdot \theta <\lambda$, and for all $i<\theta$,
$$\sup\{\beta\in d\mid \suc_\sigma(d\setminus\beta)\s X_\gamma^i\}=\gamma,$$
so $D_\gamma$ was chosen to satisfy our needs.
\end{proof}
Thus, $\langle D_\alpha\mid\alpha<\lambda^+\rangle$ witnesses simultaneously
$\p^-(\lambda^+,2,{\sql}^*,\theta,\{E^{\lambda^+}_{\lambda}\},2,\sigma)$
for all regular cardinals $\theta,\sigma < \lambda$.
\end{proof}

\begin{Q}
Does $\square_\lambda+\ch_\lambda$ (for $\lambda$ regular) entail
$\p(\lambda^+,2,{\sql}^*,\lambda,\{E^{\lambda^+}_{\lambda}\},2,\omega)$?
\end{Q}
\begin{thm}\label{cor329} If $\lambda$ is a successor of a regular cardinal $\theta$, and $\ns\restriction E^{\lambda}_{\theta}$ is saturated,
then $\ch_\lambda$ entails $\p(\lambda^+,2,{\sql}^*,\theta,\{E^{\lambda^+}_{\lambda}\},2,\theta)$.

\end{thm}
\begin{proof} Recalling Theorem~\ref{thm317}, we may assume that every stationary subset of $E^{\lambda^+}_\theta$ reflects.
Recalling Theorem~\ref{thm616a}, it suffices to prove that $\diamondsuit(E^{\lambda^+}_\lambda)$ holds.
Using~Fact~\ref{fact922}, let us fix a sequence, $\langle Z_\beta\mid \beta\in E^{\lambda^+}_\theta\rangle$
such that $G_Z:=\{ \beta\in E^{\lambda^+}_\theta\mid Z\cap\beta=Z_\beta\}$ is stationary for every $Z\s\lambda^+$.
As in \cite{sh:e4}, for every $\delta\in E^{\lambda^+}_\lambda$, we let $$\mathcal S_\delta:=\{ Z\s\delta\mid G_Z\cap\delta\text{ is stationary in }\delta\}.$$
Of course, if $Z,Z'$ are distinct elements of $\mathcal S_\delta$ then $G_{Z}\cap\delta$ and $G_{Z'}\cap\delta$ are stationary subsets of $E^\delta_\theta$ with $\sup(G_Z\cap G_{Z'})<\delta$.
As $\cf(\delta)=\lambda$ and $\ns\restriction E^{\lambda}_{\theta}$ is saturated, it follows that $|\mathcal S_\delta|\le\lambda$.

\begin{claim} For every $Z\s\lambda^+$, there exist stationarily many  $\delta\in E^{\lambda^+}_\lambda$ for which $Z\cap\delta\in\mathcal S_\delta$.
\end{claim}
\begin{proof} Let $Z\s\lambda^+$ be arbitrary. Let $D\s\lambda^+$ be some club.
Then $G:=G_Z\cap D$ is a stationary subset of $E^{\lambda^+}_\theta$.
As every stationary subset of $E^{\lambda^+}_\theta$ reflects,
pick $\delta\in E^{\lambda^+}_\lambda$ such that $G\cap\delta$ is stationary in $\delta$. Then $\delta\in D$ and $Z\cap\delta\in\mathcal S_\delta$.
\end{proof}

So $\langle\mathcal S_\delta\mid \delta\in E^{\lambda^+}_\lambda\rangle$ forms a $\diamondsuit^-(E^{\lambda^+}_\lambda)$-sequence,
and hence $\diamondsuit(E^{\lambda^+}_\lambda)$ holds.\footnote{See \cite{MR597342} for the definition of $\diamondsuit^-(S)$
and the proof that it implies $\diamondsuit(S)$.}
\end{proof}

\section{Appendix:  Combinatorial Principles}\label{Preliminaries}

\begin{definition}[Jensen, \cite{MR309729}]\label{def_diamond} For a regular uncountable cardinal $\kappa$, and a stationary subset $S\s\kappa$,
$\diamondsuit(S)$ asserts the existence of a sequence $\langle Z_\beta \mid \beta \in S \rangle$
such that
\begin{itemize}
\item $Z_\beta \subseteq \beta$ for every $\beta\in S$;
\item for every $Z\subseteq \kappa$, the set $\{\beta \in S \mid Z \cap \beta = Z_\beta \}$ is stationary in $\kappa$.
\end{itemize}
\end{definition}

\begin{fact}[Shelah, \cite{Sh:922}]\label{fact922}
For every infinite cardinal $\lambda$ and every stationary subset $S \subseteq E^{\lambda^+}_{\neq\cf(\lambda)}$,
the following are equivalent:
\begin{enumerate}
\item $\ch_\lambda$;
\item $\diamondsuit(S)$.
\end{enumerate}

In particular, $\ch_\lambda$ is equivalent to $\diamondsuit(\lambda^+)$  for every uncountable cardinal $\lambda$.
\end{fact}

\begin{definition}[Ostaszewski, \cite{MR0438292}]\label{def_clubsuit} For a stationary $S\s\omega_1$, $\clubsuit(S)$ asserts the existence of a sequence $\langle X_\alpha\mid\alpha\in S\rangle$ such that:
\begin{itemize}
\item for all nonzero limit $\alpha\in S$, $X_\alpha$ is a cofinal subset of $\alpha$ of order-type $\omega$;
\item for every uncountable $X\s\omega_1$, there exists a nonzero limit $\alpha\in S$ such that $X_\alpha\s X$.
\end{itemize}
\end{definition}

In \cite{MR1489304}, Fuchino, Shelah and Soukup studied a weakening of the preceding principle, which they denoted by $\clubsuit_w(S)$.
Their principle admits a natural generalization to arbitrary regular uncountable cardinals $\kappa$, as follows.
\begin{definition}\label{def_clubsuitw_revised} For a regular uncountable cardinal $\kappa$, and a stationary $S\s\kappa$,
$\clubsuit_w(S)$ asserts the existence of a sequence $\langle X_\alpha\mid\alpha\in S\rangle$ such that:
\begin{itemize}
\item for all limit $\alpha\in S$, $X_\alpha$ is a cofinal subset of $\alpha$ of order-type $\cf(\alpha)$;
\item for every cofinal  $X\s\kappa$, there exists a limit $\alpha\in S$ such that $\sup(X_\alpha\setminus X)<\alpha$.
\end{itemize}
\end{definition}

\begin{fact}[Devlin, \cite{MR0438292} and \cite{MR0491194}]\label{lemma613} For every regular uncountable cardinal $\kappa$ and stationary $S\s\kappa$, the following are equivalent:
\begin{enumerate}
\item $\diamondsuit(S)$;
\item $\diamondsuit(\kappa)+\clubsuit_w(S)$;
\item There exists a partition $\langle S_i\mid i<\kappa\rangle$ of $S$ such that $\diamondsuit(S_i)$ holds for all $i<\kappa$.
\end{enumerate}
\end{fact}

\begin{definition}[K\"onig-Larson-Yoshinobu, \cite{MR2320769}]
Suppose that $\chi<\kappa$ are uncountable cardinals with $\kappa$ regular, and $S\s \kappa$ stationary.

$\curlywedge^*(\chi,S)$ asserts the existence of a sequence $\langle \mathcal C_\delta\mid \delta\in S\rangle$ such that:
\begin{enumerate}
\item for every $\delta\in S$,  $\mathcal C_\delta$ is a club in $[\delta]^{<\chi}$;
\item for every club $\mathcal D$ in $[\kappa]^{<\chi}$, there exists  a club $C\s\kappa$ such that for all $\delta\in C\cap S$, there is $x_\delta\in[\delta]^{<\chi}$
with $\{ y\in \mathcal C_\delta\mid x_\delta\s y\}\s\mathcal D$.
\end{enumerate}
\end{definition}
\begin{definition}[Rinot, \cite{rinot09}]\label{1005}
Suppose that $\chi<\kappa$ are uncountable cardinals with $\kappa$ regular, and $S\s \kappa$ is stationary.

 $\curlywedge^-(\chi,S)$ asserts the existence of a matrix $\langle \mathcal C_\delta^i\mid \delta\in S, i<|\delta|\rangle$ such that:
\begin{enumerate}
\item for every $\delta\in S$ and $i<|\delta|$, $\mathcal C^i_\delta$ is cofinal in $[\delta]^{<\chi}$;
\item for every club $\mathcal D$ in $[\kappa]^{<\chi}$, the following set is stationary:
$$ \{ \delta\in S\mid \exists i<|\delta|\ \mathcal C^i_\delta\s\mathcal  D\}.$$
\end{enumerate}
\end{definition}

\begin{definition}[Rinot, \cite{rinot09}]\label{1007} Suppose that $\lambda$ is a regular uncountable cardinal, $T$ is a stationary subset of $\lambda$,
and $S$ is a stationary subset of $E^{\lambda^+}_\lambda$.

The \emph{reflected-diamond} principle, denoted $\langle T\rangle_S$, asserts the existence of
a sequence $\langle C_\delta\mid \delta\in S\rangle$ and a matrix $\langle A^\delta_i\mid \delta\in S, i<\lambda\rangle$
such that:
\begin{enumerate}
\item for all $\delta\in S$, $C_\delta$ is  a club subset of $\delta$ of order-type $\lambda$;
\item for every club $D\s\lambda^+$ and every subset $A\s\lambda^+$, there exist stationarily many $\delta\in S$ for which:
$$\{ i\in T\mid C_\delta(i+1)\in D\ \&\ A\cap (C_\delta(i+1))=A^{\delta}_{i+1}\}\text{ is stationary in }\lambda.$$
\end{enumerate}
\end{definition}
\begin{fact}[Rinot, \cite{rinot09}]\label{factRinot09}
Suppose that $\lambda$ is a regular uncountable cardinal, and $S\s E^{\lambda^+}_\lambda$ is stationary.
\begin{itemize}
\item {\cite[Remark on p.~567]{rinot09}} $\curlywedge^*(\lambda,S)$ entails $\curlywedge^-(\lambda,S)$;
\item {\cite[Theorem~2.5]{rinot09}}
$\diamondsuit(S)$ entails $\curlywedge^-(\chi,S)$ for all uncountable cardinals $\chi\le\lambda$;
\item {\cite[Theorem~2.4]{rinot09}}
if $\chi<\lambda$ is an uncountable cardinal or $\chi=\lambda$ is a successor cardinal,
then $\curlywedge^-(\chi,S) + \ch_\lambda$ entails $\langle T\rangle_S$ for every stationary $T\s\lambda$.
\end{itemize}
\end{fact}

We now introduce a variation of $\langle T\rangle_S$ that makes sense also in the context of $\lambda$ singular.\footnote{Note that  $\langle\lambda\rangle^-_S$ is equivalent to $\langle\lambda\rangle_S$ for every successor cardinal $\lambda$.}
\begin{definition}\label{1007minus} Suppose that $\lambda$ is an infinite cardinal, $T$ is a cofinal subset of $\lambda$,
and $S$ is a stationary subset of $E^{\lambda^+}_{\cf(\lambda)}$.

$\langle T\rangle^-_S$ asserts the existence of
a sequence $\langle C_\delta\mid \delta\in S\rangle$ and a matrix $\langle A^\delta_i\mid \delta\in S, i<\lambda\rangle$
such that:
\begin{enumerate}
\item for all $\delta\in S$, $C_\delta$ is  a club subset of $\delta$ of order-type $\lambda$;
\item for every club $D\s\lambda^+$ and every subset $A\s\lambda^+$, there exist stationarily many $\delta\in S$ for which:
$$\otp(\{ i\in T\mid C_\delta(i+1)\in D\ \&\ A\cap (C_\delta(i+1))=A^{\delta}_{i+1}\})=\lambda.$$
\end{enumerate}
\end{definition}

\begin{definition}[Jensen, \cite{MR309729}]\label{def64} For an infinite cardinal $\lambda$, $\square_\lambda$ asserts the existence of a sequence $\langle C_\alpha\mid\alpha<\lambda^+\rangle$ such that:
\begin{itemize}
\item $C_\alpha$ is a club in $\alpha$ for all limit $\alpha<\lambda^+$;
\item if $\bar\alpha\in\acc(C_\alpha)$, then $C_{\bar\alpha}=C_\alpha\cap\bar\alpha$;
\item $\otp(C_\alpha)\le\lambda$ for all $\alpha<\lambda^+$.
\end{itemize}
\end{definition}

We now give three variations of the preceding. The first one is essentially due to Baumgartner (cf., \cite[$\S2.2$]{MR1740483}).

\begin{definition}\label{def109} For infinite cardinals $\chi\le\lambda$, $\boxminus_{\lambda,\ge\chi}$ asserts the existence of a sequence $\langle C_\alpha\mid\alpha\in E^{\lambda^+}_{\ge\chi}\rangle$ such that:
\begin{itemize}
\item $C_\alpha$ is a club in $\alpha$ for all $\alpha\in E^{\lambda^+}_{\ge\chi}$;
\item $C_\alpha\cap\gamma=C_\beta\cap\gamma$ for all $\alpha,\beta\in E^{\lambda^+}_{\ge\chi}$ and $\gamma\in\acc(C_\alpha)\cap\acc(C_\beta)$;
\item $\otp(C_\alpha)\le\lambda$ for all $\alpha\in E^{\lambda^+}_{\ge\chi}$.
\end{itemize}
\end{definition}

\begin{definition}[Todorcevic, \cite{MR908147}] For a regular uncountable cardinal $\kappa$, $\square(\kappa)$ asserts the existence of a sequence $\langle C_\alpha\mid\alpha<\kappa\rangle$ such that:
\begin{itemize}
\item $C_\alpha$ is a club in $\alpha$ for all limit $\alpha<\kappa$;
\item if $\bar\alpha\in\acc(C_\alpha)$, then $C_{\bar\alpha}=C_\alpha\cap\bar\alpha$;
\item there exists no club $C\s\kappa$ such that $C_\alpha=C\cap\alpha$ for all $\alpha\in\acc(C)$.
\end{itemize}
\end{definition}

\begin{definition}[Rinot, \cite{rinot11}] For an uncountable cardinal $\lambda$, $\sqc_\lambda$ asserts the existence of a $\square_\lambda$-sequence $\langle C_\alpha \mid \alpha < \lambda^+ \rangle$
having the following additional feature.
For every sequence $\langle A_i \mid \iota < \lambda \rangle$ of cofinal subsets of $\lambda^+$,
every limit ordinal $\theta < \lambda$, and every club $D \subseteq \lambda^+$,
there exists some limit ordinal $\alpha < \lambda^+$ such that:
\begin{enumerate}
\item[(i)] $\otp(C_\alpha) = \theta$;
\item[(ii)] $C_\alpha (i+1) \in A_i$ for all $i < \theta$;
\item[(iii)] for every $i < \theta$ there is some $\beta \in D$ such that $C_\alpha(i) < \beta < C_\alpha(i+1)$.
\end{enumerate}
\end{definition}

\begin{fact}[{Rinot, \cite{rinot11}}]\label{fact_rinot11} For every uncountable cardinal $\lambda$, $\square_\lambda+\ch_\lambda$ entails $\sqc_\lambda$.
\end{fact}

\begin{definition}[Gray, \cite{MR2940957}]\label{def_gray}
For a regular uncountable cardinal $\lambda$, $\sd_\lambda$ asserts the existence of a sequence
$\langle (D_\alpha,X_\alpha)\mid \alpha<\lambda^+\rangle$
such that:
\begin{enumerate}
\item  for every limit $\alpha<\lambda^+$, $D_\alpha$ is a club in $\alpha$ of order-type $\le\lambda$, and $X_\alpha\s\alpha$;
\item if $\bar\alpha\in\acc(D_\alpha)$, then
$D_\alpha\cap\bar\alpha=D_{\bar\alpha}$ and $X_\alpha\cap\bar\alpha=X_{\bar\alpha}$;
\item for every subset  $X\s\lambda^+$ and club $E\s\lambda^+$,
there exists a limit $\alpha<\lambda^+$ with $\cf(\alpha)=\lambda$
such that $X\cap\alpha=X_\alpha$ and $\acc(D_\alpha)\s E$.
\end{enumerate}
\end{definition}

We take the liberty of generalizing the preceding, as follows.\footnote{The generalization from $\lambda$ regular uncountable to $\lambda$ singular was done in \cite[$\S2$]{AShS:221},
by replacing $\cf(\alpha)=\lambda$ with $\otp(D_\alpha)=\lambda$ in Item (3) of Definition \ref{def_gray}. Here, we also address the overlooked case $\lambda=\omega$,
by taking into account Item (2)(d) of \cite[Definiton 1.4]{rinot11}.
Recalling the proof of \cite[Lemma 2.8]{rinot11}, this extra requirement is not entirely trivial.
Nevertheless, one would expect that $\fsd_\omega$ should follow from $\diamondsuit(\omega_1)$.
This is indeed the case, as can be seen by combining \cite[Claim 4]{rinot19} with Lemma \ref{thm62} below.}

\begin{definition}\label{def-sd}
For an infinite cardinal $\lambda$, $\sd_\lambda$ asserts the existence of a sequence
$\langle (D_\alpha,X_\alpha)\mid \alpha<\lambda^+\rangle$
such that:
\begin{enumerate}
\item  for every limit $\alpha<\lambda^+$, $D_\alpha$ is a club in $\alpha$ of order-type $\le\omega\cdot\lambda$, and $X_\alpha\s\alpha$;
\item if $\bar\alpha\in\acc(D_\alpha)$, then
$D_\alpha\cap\bar\alpha=D_{\bar\alpha}$ and $X_\alpha\cap\bar\alpha=X_{\bar\alpha}$;
\item for every subset  $X\s\lambda^+$ and club $E\s\lambda^+$,
there exists a limit $\alpha<\lambda^+$ with $\otp(\acc(D_\alpha))=\lambda$
such that $X\cap\alpha=X_\alpha$ and $\acc(D_\alpha)\s E$.
\end{enumerate}
\end{definition}

\begin{fact}[Rinot, \cite{rinot11} and \cite{rinot19}]\label{fact_rinot19} For every singular cardinal $\lambda$, the following are equivalent:
\begin{itemize}
\item $\square_\lambda+\ch_\lambda$;
\item $\sqc_\lambda$;
\item $\sd_\lambda$.
\end{itemize}
\end{fact}

\section*{Acknowledgements}
This work was partially supported by German-Israeli Foundation for Scientific Research and Development,
Grant No.~I-2354-304.6/2014. Some of the results of this paper were announced by
the second author at the \emph{P.O.I Workshop in pure and descriptive set theory}, Torino, September 2015,
and by the first author at the \emph{8th Young Set Theory Workshop}, Jerusalem, October 2015.
The authors thank the organizers of the corresponding meetings for providing a joyful and stimulating environment.

\def\germ{\frak} \def\scr{\cal} \ifx\documentclass\undefinedcs
  \def\bf{\fam\bffam\tenbf}\def\rm{\fam0\tenrm}\fi
  \def\defaultdefine#1#2{\expandafter\ifx\csname#1\endcsname\relax
  \expandafter\def\csname#1\endcsname{#2}\fi} \defaultdefine{Bbb}{\bf}
  \defaultdefine{frak}{\bf} \defaultdefine{=}{\B}
  \defaultdefine{mathfrak}{\frak} \defaultdefine{mathbb}{\bf}
  \defaultdefine{mathcal}{\cal}
  \defaultdefine{beth}{BETH}\defaultdefine{cal}{\bf} \def\bbfI{{\Bbb I}}
  \def\mbox{\hbox} \def\text{\hbox} \def\om{\omega} \def\Cal#1{{\bf #1}}
  \def\pcf{pcf} \defaultdefine{cf}{cf} \defaultdefine{reals}{{\Bbb R}}
  \defaultdefine{real}{{\Bbb R}} \def\restriction{{|}} \def\club{CLUB}
  \def\w{\omega} \def\exist{\exists} \def\se{{\germ se}} \def\bb{{\bf b}}
  \def\equivalence{\equiv} \let\lt< \let\gt>

\end{document}